\documentclass[a4paper,11pt]{article}

\usepackage[
    a4paper,
    top=1.05in,
    bottom=1.05in,
    left=1.00in,
    right=1.00in
]{geometry}
\usepackage[T1]{fontenc}

\usepackage{amsmath,amsthm,mathtools}
\usepackage{amssymb,mathrsfs}
\usepackage{txfonts}
\InputIfFileExists{ursfs.fd}{}{}
\DeclareFontShape{U}{rsfs}{m}{n}{<-6.5>rsfs5<6.5-8>rsfs7<8->rsfs10}{}
\DeclareMathAlphabet{\mathsf}{OT1}{phv}{m}{n}
\SetMathAlphabet{\mathsf}{bold}{OT1}{phv}{b}{n}
\usepackage{microtype}
\microtypesetup{protrusion=true,expansion=true}
\allowdisplaybreaks[4]

\usepackage{enumitem}
\setlist[enumerate]{leftmargin=2.5em,itemsep=0.25em,topsep=0.35em}
\setlist[itemize]{leftmargin=2.2em,itemsep=0.25em,topsep=0.35em}
\usepackage{booktabs,array,longtable}
\usepackage[noautomatic,quiet]{imakeidx}
\makeindex[title={Index of notation and terminology},intoc,columns=2]
\indexsetup{othercode={\small\setlength{\baselineskip}{11.5pt}}}

\usepackage{titlesec}
\usepackage[titletoc]{appendix}
\titleformat{\section}{\normalfont\Large\bfseries}{\thesection}{0.75em}{}
\titleformat{\subsection}{\normalfont\large\bfseries}{\thesubsection}{0.75em}{}
\titlespacing*{\section}{0pt}{2.6ex plus 0.8ex minus 0.2ex}{1.2ex plus 0.2ex}
\titlespacing*{\subsection}{0pt}{2.0ex plus 0.6ex minus 0.2ex}{0.9ex plus 0.2ex}

\usepackage{fancyhdr}
\fancypagestyle{plain}{%
  \fancyhf{}%
  \fancyfoot[C]{\thepage}%
}

\usepackage[dvipsnames]{xcolor}
\definecolor{linkblue}{RGB}{20,67,128}
\definecolor{citegreen}{RGB}{25,105,80}
\definecolor{urlpurple}{RGB}{100,55,125}
\usepackage{hyperref}
\hypersetup{colorlinks=true,
            linkcolor=linkblue,
            citecolor=citegreen,
            urlcolor=urlpurple,
            plainpages=false}
\usepackage{bookmark}
\usepackage[nameinlink,noabbrev,capitalise]{cleveref}
\hypersetup{
  pdftitle={Parametrization of ancient Ricci flows with smooth compact tangent shrinkers},
  pdfauthor={Yu Li},
  pdfsubject={Ancient Ricci flows with smooth compact tangent shrinkers},
  pdfkeywords={ancient Ricci flow, compact tangent shrinker, Kuranishi reduction, characteristic exponent, finite-order parametrization}
}

\numberwithin{equation}{section}
\makeatletter
\renewcommand{\theequation}{\thesection.\ifnum\value{equation}<10 0\fi\arabic{equation}}
\makeatother

\newtheorem{theorem}{Theorem}[section]
\newtheorem{proposition}[theorem]{Proposition}
\newtheorem{lemma}[theorem]{Lemma}
\newtheorem{corollary}[theorem]{Corollary}
\theoremstyle{definition}
\newtheorem{definition}[theorem]{Definition}

\theoremstyle{remark}
\newtheorem{remark}[theorem]{Remark}

\crefname{hypothesis}{Hypothesis}{Hypotheses}
\Crefname{hypothesis}{Hypothesis}{Hypotheses}
\crefname{appendix}{Appendix}{Appendices}
\Crefname{appendix}{Appendix}{Appendices}

\newcommand{\R}{\mathbb R}
\newcommand{\C}{\mathbb C}

\newcommand{\CP}{\mathbb {CP}}
\newcommand{\bg}{\bar g}
\newcommand{\bfv}{\bar f}
\newcommand{\Sli}{\mathcal S}
\newcommand{\Comp}{\mathcal C}
\newcommand{\Ob}{\mathcal O}
\newcommand{\Ent}{\mathcal E}
\newcommand{\Lie}{\mathcal L}

\newcommand{\ip}[2]{\left\langle #1,#2\right\rangle}
\newcommand{\dd}{\,\mathrm{d}}
\DeclareMathOperator{\Ric}{Ric}
\DeclareMathOperator{\Rm}{Rm}
\DeclareMathOperator{\tr}{tr}
\DeclareMathOperator{\Spec}{Spec}
\DeclareMathOperator{\Ran}{Ran}
\DeclareMathOperator{\Id}{Id}
\DeclareMathOperator{\Isom}{Isom}
\DeclareMathOperator{\Diff}{Diff}
\DeclareMathOperator{\Herm}{Herm}
\DeclareMathOperator{\Gr}{Gr}
\DeclareMathOperator{\diag}{diag}

\newcommand{\kk}{\mathfrak k}

\newcommand{\PN}{\mathscr P_N}
\newcommand{\UN}{\mathscr U_N}
\newcommand{\eps}{\epsilon}
\newcommand{\Gop}{\mathcal G}
\newcommand{\Hop}{\mathcal H}
\newcommand{\XP}{\mathsf X}
\newcommand{\YP}{\mathsf Y}

\title{Parametrization of ancient Ricci flows\\with smooth compact tangent shrinkers}
\author{Yu Li}
\date{October 22, 2026}

\begin{document}

\maketitle

\begin{abstract}
We study ancient Ricci flows whose tangent flow at backward infinity is a smooth compact Ricci shrinker. After normalization, every nontrivial solution either converges exponentially or has polynomial decay governed by neutral directions. We recover the finite-dimensional parametrization of the exponentially convergent solutions and establish finite-order asymptotic expansions for all slowly convergent solutions. For principal slow solutions, we obtain a locally bijective finite-dimensional parametrization compatible with diffeomorphisms, parabolic rescaling, and time translation.
\end{abstract}

\tableofcontents

\section{Introduction}\label{sec:intro}
\subsection*{Background and motivation}
Ancient solutions to Ricci flow arise as limits of parabolic rescalings near singularities. Their geometry is therefore central to the analysis of singularity formation and to the continuation of the flow. Hamilton's Ricci flow and Perelman's entropy and noncollapsing theorems provide the foundation for this study \cite{Hamilton,Perelman}.

The classical ancient $\kappa$-solutions admit a complete classification in dimensions two and three. Building on Hamilton's work on surfaces \cite{HamiltonSurfaces}, Perelman \cite[Section~11.3]{Perelman} identified the shrinking round sphere and its quotient as the nonflat two-dimensional $\kappa$-solutions. Daskalopoulos--Hamilton--\v{S}e\v{s}um \cite{DHS} established the full classification on $S^2$, including the collapsed King--Rosenau solution. In dimension three, Brendle \cite{Brendle3} classified the noncompact $\kappa$-solutions as shrinking round cylinders and their quotients or the Bryant soliton. Brendle--Daskalopoulos--\v{S}e\v{s}um \cite{BDS} classified the compact solutions on $S^3$ as shrinking round spheres or Perelman's ancient oval, using the asymptotic analysis of Angenent--Brendle--Daskalopoulos--\v{S}e\v{s}um \cite{ABDS}.

In dimensions $n\ge4$, substantial structure and classification results for $\kappa$-noncollapsed ancient solutions are available under additional curvature assumptions. Under uniform positive isotropic curvature and weak $\mathrm{PIC}_2$, Brendle--Naff \cite{BN} classified the noncompact ancient $\kappa$-solutions, while Brendle--Daskalopoulos--Naff--\v{S}e\v{s}um \cite{BDNS} classified those on $S^n$. Cho--Li \cite{ChoLi} classified complete noncompact $\kappa$-noncollapsed ancient solutions under uniform positive isotropic curvature in dimensions $n=4$ or $n\ge12$. Chen \cite{ChenPIC} extended this classification to dimensions $9\le n\le11$, and Cho \cite{ChoPIC} established the case $n=8$. The recent preprint of Brendle--Tsiamis \cite{BT} extends the classification to every $n\ge4$. In the K\"ahler setting, Li \cite{LiKRF} classified complete $\kappa$-noncollapsed ancient K\"ahler--Ricci flows with nonnegative bisectional curvature in all complex dimensions.

An ancient solution can itself be studied by rescaling toward backward infinity. In the noncollapsed setting, the resulting tangent flows are modeled on Ricci shrinkers, possibly with singularities; see Bamler \cite{Bamler} and Fang--Li \cite{FL}. The tangent shrinker describes the leading geometry of the ancient solution. A fixed tangent shrinker may arise from different ancient solutions, and finer asymptotic data are needed to distinguish them.

We study this question when the tangent shrinker is smooth and compact. Fix a closed normalized Ricci shrinker $(M,\bg,\bfv)$ satisfying
\begin{equation}
 \Ric_{\bg}+\nabla^2_{\bg}\bfv=\frac12\bg.
 \label{eq:shrinker}\index{background data@$\bg,\bfv$}\index{background data@$\bg,\bfv$}
\end{equation}
Throughout the paper, manifolds are connected. The ancient Ricci flows are noncollapsed in the entropy sense recalled in Section~\ref{sec:geometry}, and their backward tangent flow is the canonical shrinking flow associated with \eqref{eq:shrinker}. Dynamical instability provides a source of ancient solutions emerging from compact shrinkers. Kr\"oncke \cite[Theorem~1.3]{KronckeSolitons} proved that a compact Ricci shrinker which is not a local maximum of Perelman's shrinker entropy $\boldsymbol{\nu}$ admits a nontrivial ancient normalized Ricci flow converging to it backward, modulo diffeomorphisms. Our aim is to identify finite-dimensional data that determine these ancient flows and to characterize which data are realized.

Chan--Ma--Zhang \cite{CMZ} prove uniqueness of a smooth closed tangent flow and convergence to it after the canonical modification, using the entropy and the \L{}ojasiewicz inequality of Sun--Wang \cite{SW}. This convergence places an entire ancient tail in a neighborhood of the fixed shrinker. Choi--Lai \cite{CL} construct a harmonic map heat flow gauge and parametrize exponentially convergent ancient Ricci flows. When the shrinker is integrable, they also show that every convergent ancient flow has exponential decay. The remaining slow dynamics involve the neutral directions and their nonlinear interactions.

There are close precedents in mean curvature flow. Schulze \cite{Schulze} proves uniqueness when a tangent flow is a smooth closed multiplicity-one shrinker. Choi--Mantoulidis \cite{CMancient} construct and classify ancient gradient flows of elliptic functionals using the Morse index. Sun--Xue \cite{SXclosed} use invariant manifold theory to study the local dynamics near a nonspherical closed shrinker and the effect of perturbing the initial hypersurface. Chodosh--Choi--Mantoulidis--Schulze \cite{CCMS} establish existence and uniqueness of ancient mean curvature flows lying on one side of a compact shrinker, with applications to generic singularities. These works show how the geometry and spectrum of a shrinker organize the nearby flow.

Ancient mean curvature flows with cylindrical backward limits provide further examples of this relation. Haslhofer--Hershkovits \cite{HH} construct ancient ovals with such limits. Choi--Haslhofer--Hershkovits--White \cite{CHHW} classify ancient asymptotically cylindrical flows in the neck setting, while Du--Haslhofer \cite{DH} construct families of ancient ovals. These families illustrate why a tangent shrinker alone need not determine an ancient solution. More recently, Bamler--Lai \cite{BLodi} introduce the PDE--ODI principle and obtain asymptotic expansions to arbitrarily high order for cylindrical mean curvature flows. Their subsequent work \cite{BLclass} classifies and parametrizes ancient asymptotically cylindrical flows. The common question is how finer asymptotic data distinguish solutions with the same tangent model and determine which solutions occur. Extending the present Ricci-flow parametrization to noncompact tangent shrinkers is a natural direction for future work; it requires corresponding gauge constructions and estimates at infinity.

We work with ancient germs, meaning that solutions are identified when they agree on a sufficiently negative tail. Ricci-flow uniqueness on a closed manifold then determines their maximal smooth continuations. The geometric equivalences considered below are time-independent diffeomorphisms, parabolic rescaling, and translation of the original Ricci-flow time.

\subsection*{Exponential and slow convergence}
To study slow solutions, we need to relate the nonlinear neutral motion to the entropy and to the static obstruction of the background shrinker. This guides the choice of gauge. Choi--Lai use harmonic map heat flow to obtain a Ricci--DeTurck equation. Here we first pass to the entropy-gradient equation, called \emph{modified Ricci flow} (MRF), and then choose a time-dependent diffeomorphism that preserves the affine weighted slice
\[
 \bg+\Sli,\qquad \Sli=\ker\mathrm{div}_{\bfv}.
\]
The resulting equation is \emph{special modified Ricci flow} (SMRF). The slice removes infinitesimal diffeomorphism directions and gives a self-adjoint linearization in the fixed weighted pairing. It also places the neutral evolution and the static deformation problem in the same coordinates. We use this relation to compare the reduced motion with the entropy and the Li--Zhang obstruction. Section~\ref{sec:geometry} constructs the gauge and proves the correspondence
\begin{equation}
 \mathrm{RF}/\Diff(M)\cong\mathrm{MRF}/\Diff(M)
 \cong\mathrm{SMRF}/G,
 \qquad G=\Isom(M,\bg,\bfv),
 \label{eq:intro-quotients}\index{G isometries@$G$}\index{S slice@$\Sli$ (weighted slice)}
\end{equation}
where $\Diff(M)$ acts by time-independent diffeomorphisms. The effects of parabolic rescaling and original-time translation are tracked separately.

The normalized logarithmic time is $s=-\log(-t)$. Exponential and polynomial decay below refer to this variable. Writing the slice metric as $\bg+h$, the equation and its linearization are
\begin{equation}
 \begin{aligned}
 h_s&=\mathcal V(h),\qquad \mathrm{D}\mathcal V(0)=L=\Delta_{\bfv}+2\Rm_{\bg},\\
 \mathcal V(h)&=(\bg+h)-2\Ric_{\bg+h}-2\nabla_{\bg+h}^{2}f_{\bg+h}+\Lie_{Y(h)}(\bg+h).
 \end{aligned}
 \label{eq:intro-linearization}\index{C complement@$\Comp,\Comp_j$}\index{h@$h$}\index{K neutral@$K$}
\end{equation}
Here $f_g$ is the canonical minimizer of $\boldsymbol{\mu}(g,1)$, and $Y(h)$ is the unique small vector field, orthogonal to the background Killing fields, for which $\mathrm{div}_{\bfv}\mathcal V(h)=0$. Since $Y(0)=\mathrm{D}Y(0)=0$, this additional gauge preserves the linearization. Let $E_+$ be the sum of the positive eigenspaces of $L|_{\Sli}$, and set
\[
 K=\ker(L|_{\Sli}),\qquad \Comp=K^\perp\cap\Sli.
\]
Orthogonality and spectral projections use the fixed weighted $L^2$ pairing of $(\bg,\bfv)$. Both $E_+$ and $K$ are finite-dimensional. We write $b=\Pi_Kh$ and $|b|=\|b\|_{L^2}$.

\begin{theorem}[Exponential--neutral dichotomy]\label[theorem]{thm:dichotomy-intro}
Let $h$ be an ancient SMRF converging smoothly to zero as $s\searrow-\infty$. If $h\not\equiv0$, exactly one of the following holds on a sufficiently negative tail:
\begin{enumerate}
\item There is $\gamma>0$ such that $\|h(s)\|_{C^k}=O(e^{\gamma s})$ for every fixed $k$.
\item The neutral component is nonzero, and
\[
 h(s)=b(s)+O_{C^k}(|b(s)|^2),\qquad
 |b_s|\le C|b|^2,\qquad |b(s)|\ge\frac{c}{|s|},
\]
for every fixed $k$, with $C,c>0$.
\end{enumerate}
\end{theorem}
The linear equation $h_s=Lh$ explains the spectral alternatives. An eigenmode with eigenvalue $\lambda$ evolves as $e^{\lambda s}$. Positive modes decay exponentially as $s\searrow-\infty$, whereas negative homogeneous modes grow and cannot give the leading term of a solution converging to zero. Neutral modes have no linear evolution, so their decay is governed by the nonlinear terms. This suggests exponential decay when positive modes dominate and slow evolution when neutral modes dominate. The proof in \cref{prop:dichotomy} controls the nonlinear coupling and establishes this alternative. We call solutions in the second case neutral-dominated or slow. Their precise polynomial rates will be identified below.

\begin{theorem}[Exponential parametrization]\label[theorem]{thm:fast-intro}
At a sufficiently negative terminal time $s_*$, the exponentially convergent ancient SMRFs form a smooth local family parametrized by $\xi=\Pi_+h(s_*)\in E_+$. Their terminal values have the form
\[
 h_\xi(s_*)=\xi+\psi(\xi),\qquad
 \psi(\xi)\in E_+^\perp\cap\Sli,\qquad
 \psi(0)=\mathrm{D}\psi(0)=0.
\]
The parametrization is $G$-equivariant, and every exponentially convergent ancient germ belongs to this family after restriction to a sufficiently negative tail.
\end{theorem}
The zero parameter corresponds to the canonical shrinking Ricci flow. This recovers the construction and classification of exponentially convergent solutions obtained by Choi--Lai \cite[Theorems~1.3 and~6.1]{CL} in the fixed slice; their classification allows nonintegrable shrinkers. We prove the fixed-slice statement in \cref{thm:fast-classification}. The main purpose of the present paper is to parametrize slow solutions.

\subsection*{Asymptotic expansions of slow solutions}
The tangent shrinker gives the limiting metric. Higher-order expansions describe how an ancient solution departs from it and which spatial tensors govern that departure. Precise ancient asymptotics have played a central role in the work of Angenent--Brendle--Daskalopoulos--\v{S}e\v{s}um on Ricci flow \cite{ABDS} and Angenent--Daskalopoulos--\v{S}e\v{s}um on mean curvature flow \cite{ADS}. Those works analyze cylindrical backward asymptotics. Here the compact tangent shrinker gives a discrete spectrum and a finite-dimensional neutral space on which to organize the expansion.

For each finite integer $N\ge2$, consider polynomial maps
\begin{equation}
 \Psi^{[N]}(b)=b+\sum_{j=2}^{N}H_j(b),\qquad
 F^{[N]}(b)=\sum_{j=2}^{N}F_j(b),
 \label{eq:intro-jets}\index{F reduced@$F^{[N]},F_j$}\index{F reduced@$F^{[N]},F_j$}\index{Psi graph@$\Psi^{[N]},H_j$ (approximate graph)}\index{Psi graph@$\Psi^{[N]},H_j$ (approximate graph)}
\end{equation}
where $H_j:K\to\Comp$ and $F_j:K\to K$ are homogeneous of degree $j$. For this slice and normalization, the coefficients depend only on $(M,\bg,\bfv)$ and agree as $N$ increases. The first map describes the metric in terms of its neutral component, and the second describes the reduced motion.

\begin{theorem}[Finite-order expansion]\label[theorem]{thm:expansion-intro}
For each integer $N\ge2$, there are unique maps of the form \eqref{eq:intro-jets} satisfying
\[
 \mathcal V(\Psi^{[N]}(b))-\mathrm{D}\Psi^{[N]}(b)F^{[N]}(b)=O(|b|^{N+1}).
\]
They are $G$-equivariant. Every neutral-dominated ancient SMRF satisfies
\[
 h(s)=\Psi^{[N]}(b(s))+O_{C^k}(|b(s)|^{N+1}),\qquad b_s=F^{[N]}(b)+O(|b|^{N+1})
\]
on a sufficiently negative tail, for every fixed spatial regularity $k$.
\end{theorem}
This combines \cref{thm:formal,thm:actual-expansion}. To construct the coefficients, cancel the invariance defect successively in each homogeneous degree. Its $K$-component determines $F_j$, while inversion of $L$ on $\Comp$ determines $H_j$. For an actual solution, the spectral estimates control the complement by the remaining defect. Parabolic regularity then gives the expansion in each fixed spatial norm.

Sun--Xue \cite{SXcyl} obtain asymptotic normal forms for rescaled mean curvature flows approaching generalized cylinders. Their analysis identifies the neutral modes that govern the leading profile and derives geometric consequences from that profile. Theorem~\ref{thm:expansion-intro} likewise organizes the asymptotics through the neutral dynamics, here giving an expansion to every prescribed finite order around a compact Ricci shrinker.

It is an expansion in the neutral coordinate $b(s)$; further analysis of the reduced equation determines its expression in time. Exponentially small complementary data are below every finite polynomial order. Recovering these data is part of the parametrization problem.

\subsection*{Parametrization of principal slow solutions}
Choi--Hung \cite{CH} prove a version of Thom's gradient conjecture for nonlinear evolution equations, determining a limiting direction and a rational characteristic exponent for slowly converging solutions. Their finite-dimensional perturbed-gradient theorem is the model for our analysis. We verify its hypotheses for the SMRF reduction by relating the motion on $K$ to the reduced entropy and estimating the error to arbitrarily high order.

Let $r\ge2$ be the first degree for which $F_r\not\equiv0$. For every slow solution, \cref{thm:ell} gives a rational number $\ell_*\ge r+1$ and a nonzero $a_*\in K$ such that
\begin{equation}
 b(s)=|s|^{-1/(\ell_*-2)}(a_*+o(1)).
 \label{eq:intro-characteristic}\index{ell characteristic@$\ell_*,\ell^*$ (characteristic exponents)}
\end{equation}
Thus the characteristic exponent records the decay rate of the particular trajectory. We call a slow solution \emph{principal} when $\ell_*=r+1$. In this case its characteristic vector satisfies
\[
 a_*=\lim_{s\searrow-\infty}|s|^{1/(r-1)} b(s),\qquad
 F_r(a_*)=\frac{a_*}{r-1}.
\]
If the background has finite Li--Zhang obstruction order $k_{\mathrm{LZ}}$, then \cref{prop:sandwich} gives
\[
 r+1\le\ell_*\le k_{\mathrm{LZ}}+1.
\]
The first nonzero reduced term supplies the lower bound, while the uniform static obstruction supplies the upper bound. In particular, if $k_{\mathrm{LZ}}=r$, every slow solution is principal. When $k_{\mathrm{LZ}}=2$, we have $r=2$ and $\ell_*=3$, as in the even-dimensional projective applications below. The inequality $\ell_*\ge r+1$ can be strict: \cref{prop:cp-nonprincipal} constructs slow solutions with $r=2$ and $\ell_*=4$.

For each principal characteristic vector $a_*$, choose one sufficiently large finite order $N$. The required accuracy is determined by the linearization $\mathrm{D}F_r(a_*)$, as specified in \eqref{eq:indicial-gap}. The parameters consist of an exact reduced trajectory and a complementary parameter. Evaluation at one time determines the reduced trajectory, while the complementary parameter is the terminal datum of an exponentially decaying homogeneous solution of the complement equation along that trajectory. Thus both kinds of data are finite-dimensional.

\begin{theorem}[Parametrization of principal slow solutions]\label[theorem]{thm:main-intro}\index{Principal slow solutions}
Allow the principal characteristic vector $a_*$ to vary. For each bounded local family, choose a sufficiently large finite order $N$ as in \cref{thm:parametrization}. On a sufficiently negative common tail, the map
\begin{equation}
 \Theta_N(h)=\bigl(\beta_N(h),\zeta_N(h)\bigr)
 \label{eq:theta-intro}\index{Theta parametrization@$\Theta_N$}
\end{equation}
is a bijection between the bounded local solution class of \cref{thm:parametrization} and its parameter space. These bijections form a $G$-equivariant parametrization. Here $a_*=a_*(h)$ is the characteristic vector of $h$, and $\beta_N$ is the unique exact trajectory of
\[
 \beta_s=F^{[N]}(\beta),\qquad |s|^{1/(r-1)}\beta(s)\longrightarrow a_*,
\]
satisfying
\[
 |\Pi_Kh(s)-\beta_N(s)|=O(|s|^{-\kappa_N}),\qquad
 \kappa_N=\frac{N+1}{r-1}-1,
\]
and $\zeta_N$ lies in the finite-dimensional space of terminal values of exponentially decaying solutions of the complementary linear equation along $\beta_N$.
Every principal ancient germ is represented in one of these bounded local parametrizations on a common ancient interval. The corresponding geometric quotients give the unmarked Ricci-flow parametrization. If $E_+=\R\Ric_{\bg}$, original-time translation removes its scalar complementary parameter.
\end{theorem}
The parameter space is defined in Section~\ref{sec:parameters}, and \cref{rem:principal-germs} gives the coverage of principal germs.

Approximation by an exact reduced trajectory reduces the neutral dynamics to a finite-dimensional ODE. The equation for $b=\Pi_Kh$ still contains a remainder inherited from the PDE. Its associated exact reduced trajectory $\beta_N$ satisfies the exact equation $\beta_s=F^{[N]}(\beta)$ and is uniquely determined by the prescribed approximation rate. We can therefore study admissible trajectories of this ODE in $K$, each determined by its value at one time. The complementary parameter $\zeta_N$ records the exponentially decaying homogeneous part of the complementary equation.

Conversely, a prescribed exact reduced trajectory and complementary parameter determine an approximate solution, and the coupled correction equations recover an exact ancient SMRF. The main estimates control the interaction of the neutral and complementary directions at a fixed finite order, yielding both realization and uniqueness. Thus the principal parametrization reduces the study of ancient solutions to admissible ODE trajectories, finite-dimensional complementary parameters, and their geometric quotients.

\subsection*{Complex projective examples}
Kr\"oncke \cite[Corollary~1.8]{KronckeEinstein} proved dynamical instability of the Fubini--Study metric on every $\CP^N$, $N\ge2$, yielding nontrivial ancient normalized Ricci flows converging to it backward modulo diffeomorphisms. Knopf--\v{S}e\v{s}um \cite{KnopfSesum} gave an independent proof of its entropy instability. Here we determine the local moduli on $\CP^{2n}$, obtain uniqueness of the nontrivial class on $\CP^2$ modulo the geometric symmetries below, and construct a nonprincipal branch with $|s|^{-1/2}$ decay in every odd complex dimension at least three.

For the Fubini--Study shrinker on $\CP^{2n}$, normalized by $\Ric_{\bg}=\frac12\bg$, put $d=2n+1$. Its neutral space $K=\ker(L|_{\Sli})$ is identified with $\Herm_0(d)$, the real vector space of traceless Hermitian $d\times d$ matrices, by
\[
 A\longmapsto h_A=u_A\bg+2\bar\nabla^2u_A,
 \qquad u_A([z])=\frac{z^*Az}{|z|^2}.
\]
Here $[z]\in\CP^{2n}$ denotes homogeneous coordinates, and $\bar\nabla$ is the Levi-Civita connection of $\bg$. Under this correspondence, the quadratic reduced field on $K$ is
\begin{equation}
 F_2(A)=\frac{2}{d+2}\left(A^2-\frac{\tr(A^2)}{d}I\right).
 \label{eq:intro-projective-field}
\end{equation}
Here $I$ is the identity matrix, and the expression in parentheses is the traceless part of $A^2$. Thus \eqref{eq:intro-projective-field} gives the leading neutral evolution in matrix coordinates. Since $d$ is odd, this field has no nonzero zero and gives a uniform quadratic obstruction. Every slow solution is therefore principal. Section~\ref{sec:cp} identifies the possible characteristic matrices and analyzes the stable fibers of the reduced equation.

\begin{theorem}[Projective parametrization]\label[theorem]{thm:cp-intro}
Every nontrivial slow ancient Ricci-flow germ with Fubini--Study tangent on $\CP^{2n}$ has a characteristic type $p\in\{1,\ldots,n\}$. Modulo time-independent diffeomorphisms, parabolic rescaling, and original-time translation, its local type-$p$ moduli are modeled on
\begin{equation}
 \Herm_0(p)/U(p)\cong
 \left\{(\xi_1,\ldots,\xi_p):\xi_1\le\cdots\le\xi_p,
 \ \sum_{i=1}^{p}\xi_i=0\right\}.
 \label{eq:intro-weyl}
\end{equation}
The generic dimension is $p-1$. In particular, ancient Ricci flows with Fubini--Study tangent on $\CP^2$ have exactly two classes modulo these symmetries: the canonical shrinking Ricci flow and one nontrivial slow class.
\end{theorem}
This is proved in \cref{thm:cp-moduli,cor:cp2}. The characteristic matrices form Grassmannians of marked limits. Fixing one limit gives a stable fiber, and the isometry and time quotients yield the Weyl chambers in \eqref{eq:intro-weyl}. The local asymptotic data determine the maximal smooth ancient continuations.

In odd complex dimension, nonprincipal slow solutions occur.
Section~\ref{sec:cp-nonprincipal} constructs an ancient Ricci flow with
Fubini--Study backward tangent on every $\CP^{2n+1}$, $n\ge1$, whose
normalized perturbation decays as $|s|^{-1/2}$. For these solutions,
$r=2$ and $\ell_*=4$. Thus the equality $\ell_*=r+1$ does not hold for
all slow solutions, even when the compact tangent shrinker is rigid.

\subsection*{Convergence at a finite singular time}
The same method describes convergence to a compact tangent shrinker
at a finite singular time. With $\tau=T-t=e^{-s}$, logarithmic time
now tends to $+\infty$. The fixed-slice equation and its reduced
coefficients are unchanged. The comparison of the three spectral components gives the
exponential--neutral alternative, and the same finite-dimensional
reduction determines the slow rate and limiting tensor.

\begin{theorem}[Forward convergence]\label[theorem]{thm:forward-intro}
Let $(M,g_{\mathrm{RF}}(t))$, $0\le t<T<\infty$, be a closed Ricci flow
with first singular time $T$. Suppose one tangent flow at $T$ is smooth
and compact, and identify its normalized shrinker with $(M,\bg,\bfv)$.
Put $\tau=T-t$. In a horizontal gauge,
$h=\tau^{-1}\phi_t^*g_{\mathrm{RF}}(t)-\bg$ has exactly one of the
following forms as $\tau\searrow0$:
\begin{enumerate}
\item $h=0$ on a tail, corresponding to the canonical shrinking flow.
\item For some $\lambda>0$ and nonzero $a\in\Sli$ with $La=-\lambda a$,
\[
 h=\tau^\lambda(a+o(1)).
\]
\item For some rational $\ell^*\ge3$ and nonzero $a^*\in K$,
\[
 h=|\log\tau|^{-1/(\ell^*-2)}(a^*+o(1)).
\]
\end{enumerate}
All limits hold in every fixed spatial norm.
The possible exponents $\ell^*$ lie in a finite set determined by
the reduced entropy, and every slow solution admits the finite-order
expansion of \cref{thm:expansion-intro}, with
$b=\Pi_Kh$ and $s=-\log\tau$.
\end{theorem}

The full statement, including the marking and the equation for the leading coefficient in the principal case,
is \cref{thm:forward}.
For the Fubini--Study shrinker on $\CP^{2n}$, the forward slow profiles
are the negatives of the ancient profiles from Section~\ref{sec:cp};
see \cref{cor:forward-cp}.

The finite-order reduction, characteristic asymptotics, and parametrization of principal slow solutions developed here also apply to analytic parabolic gradient flows with the corresponding spectral decomposition and parabolic estimates. Rescaled mean curvature flow near a smooth closed shrinker provides a natural setting for this method: normal graphs give local coordinates, and the Gaussian area supplies the analytic functional. The gradient asymptotics of Choi--Hung \cite{CH} and the invariant-manifold analysis of Sun--Xue \cite{SXclosed} provide the relevant framework.

\subsection*{Organization of the paper}
Section~\ref{sec:geometry} establishes the geometric correspondences and uniqueness of continuation from an ancient interval. Section~\ref{sec:dichotomy} proves the exponential--neutral dichotomy and treats exponentially convergent solutions. Section~\ref{sec:normal} constructs and realizes the finite-order expansion. Section~\ref{sec:exponent} establishes the characteristic exponent, Section~\ref{sec:principal} proves the parametrization of principal slow solutions, and Section~\ref{sec:cp} develops the projective applications. Section~\ref{sec:forward} treats forward convergence to compact tangent shrinkers. Appendix~\ref{app:parabolic} establishes the weighted linear estimates. An index of notation and terminology precedes the references.

\par\vspace{\baselineskip}
\noindent\textbf{Acknowledgements.}
The author thanks Hanbing Fang for helpful discussions during the early stages of this work.
The author is supported by the National Key Research and Development Program of China (Grant No.~2025YFA1018200), NSFC-12522105, and research funds from the University of Science and Technology of China and the Chinese Academy of Sciences.

\par\vspace{\baselineskip}
\noindent\textbf{AI declaration.}
ChatGPT (GPT-6 Astra) was used to assist with drafting, mathematical exposition, revision, and reference checking, in particular in Section~6 and the computations in Section~7. The mathematical ideas are the author's own. Most of the techniques are drawn from the author's previous works and those of other researchers, as cited in the text. The author takes full responsibility for the mathematical content and correctness.

\section{Ricci flow, modified flow, and the geometric quotients}\label{sec:geometry}
We represent ancient Ricci flows in a fixed weighted slice, where the background linearization and its spectral projections remain fixed. The gauge constructions identify Ricci-flow germs modulo time-independent diffeomorphisms with convergent SMRF germs modulo background isometries. The main step is the correspondence in \cref{thm:correspondence}, including convergence of the inverse gauge.

\subsection{Noncollapsed ancient Ricci flows}
We work with smooth ancient Ricci flows $g_{\mathrm{RF}}(t)$ on a closed manifold $M$. Restricting to an interval ending at a smooth time and translating the Ricci-flow time, we take their domain to be $(-\infty,0]$. Write $n=\dim M$. If
\[
 \mathrm{d}\nu_{x,t;t-\tau}
 =(4\pi\tau)^{-n/2}e^{-f_{x,t;\tau}}\mathrm{d}V_{g_{\mathrm{RF}}(t-\tau)}
\]
is the conjugate heat kernel measure based at $(x,t)$, its pointed Nash entropy is
\begin{equation}
 \mathcal N_{x,t}(\tau)=\int_M f_{x,t;\tau}\,\mathrm{d}\nu_{x,t;t-\tau}-\frac n2,
 \qquad \tau>0.
 \label{eq:nash-definition}\index{nu heat kernel@$\mathrm{d}\nu_{x,t;t-\tau}$ (conjugate heat kernel)}\index{N Nash@$\mathcal N_{x,t}(\tau)$}\index{Noncollapsed ancient flow}
\end{equation}
We call the ancient flow \emph{noncollapsed} if there are a spacetime point $(x_*,t_*)$ and $Y<\infty$ such that $\mathcal N_{x_*,t_*}(\tau)\ge-Y$ for every $\tau>0$. This is the entropy formulation used in the tangent-flow theory of Bamler \cite{Bamler} and Fang--Li \cite{FL}.

For closed ancient flows this condition is equivalent to each of
\begin{equation}
 \inf_{\substack{x\in M,\ t\le0\\\tau>0}}\mathcal N_{x,t}(\tau)>-\infty,
 \qquad
 \inf_{t\le0}\boldsymbol{\nu}(g_{\mathrm{RF}}(t))>-\infty,
 \qquad \boldsymbol{\nu}(g)=\inf_{\tau>0}\boldsymbol{\mu}(g,\tau).
 \label{eq:entropy-noncollapse}\index{mu entropy@$\boldsymbol{\mu}(g,\tau)$}\index{nu entropy@$\boldsymbol{\nu}(g)$}\index{Noncollapsed ancient flow}
\end{equation}
Here $\boldsymbol{\mu}(g,\tau)$ and $\boldsymbol{\nu}(g)$ are Perelman's entropy functionals. In fact, Chan--Ma--Zhang \cite[Theorem 1.1]{CMZsob} identify $\inf_{t\le0}\boldsymbol{\nu}(g_{\mathrm{RF}}(t))$ with $\lim_{\tau\nearrow\infty}\mathcal N_{x_*,t_*}(\tau)$; the same lower bound then holds at every spacetime point. See also \cite[Example 11.5]{FL}. We consider such flows with a smooth closed tangent shrinker $(M,\bg,\bfv)$ at backward infinity.

\subsection{Modified Ricci flow and geometric convergence}\label{sec:mrf}
For a metric $g$ near $\bg$, let $f_g$ be the unique minimizer of Perelman's functional
\begin{equation}
 \mathcal W(g,f,1)=(4\pi)^{-n/2}
 \int_M\bigl(\mathrm{R}_g+|\nabla f|_g^2+f-n\bigr)e^{-f}\mathrm{d}V_g,
 \qquad (4\pi)^{-n/2}\int_Me^{-f}\mathrm{d}V_g=1.
 \label{eq:canonical-potential}\index{W Perelman@$\mathcal W(g,f,1)$}\index{background data@$\bg,\bfv$}\index{f potential@$f_g$}\index{mu entropy@$\boldsymbol{\mu}(g,\tau)$}
\end{equation}
The map $g\mapsto f_g$ is analytic on a sufficiently small neighborhood of the normalized shrinker \cite{SW}. We normalize $\bfv=f_{\bg}$ and define
\begin{equation}
 \Phi(g)=\frac12g-\Ric_g-\nabla_g^2f_g.
 \label{eq:Phi-definition}\index{Phi@$\Phi$}
\end{equation}

For the flows under consideration, the minimizer of $\boldsymbol{\mu}(g_{\mathrm{RF}}(t),|t|)$ is unique and varies smoothly on a sufficiently remote ancient tail \cite{SW,CMZ}. We choose the parabolic time unit so that this holds for $t\le-1$. Diffeomorphism invariance of the entropy makes this definition independent of the identifications used in the tangent-flow convergence. Set
\begin{equation}
 t=-e^{-s},\qquad \tau=|t|=e^{-s},\qquad
 \hat g(s)=e^s g_{\mathrm{RF}}(-e^{-s}),\qquad s\le0.
 \label{eq:time-general}\index{g normalized@$\hat g(s)$}\index{s time@$s,t$}
\end{equation}
The minimizer at scale $|t|$ is the scale-one potential $f_{\hat g(s)}$. Since $\mathrm{d}t/\mathrm{d}s=e^{-s}$, direct differentiation gives
\begin{equation}
 \hat g_s=\hat g-2\Ric_{\hat g}.
 \label{eq:normalized-rf}
\end{equation}
Let $\varphi_s$ solve
\begin{equation}
 \partial_s\varphi_s=-\nabla^{\hat g(s)}f_{\hat g(s)}\circ\varphi_s,
 \qquad \varphi_0=\Id.
 \label{eq:potential-gauge}
\end{equation}
Then
\begin{equation}
 g(s)=\varphi_s^*\hat g(s),\qquad
 \partial_sg=g-2\Ric_g-2\nabla_g^2f_g=2\Phi(g).
 \label{eq:mrf-definition}\index{Modified Ricci flow (MRF)}
\end{equation}
This is the modified Ricci flow (MRF). Conversely, pulling an MRF back by the diffeomorphisms generated by $\nabla^{g(s)}f_{g(s)}$, with identity initial value at a finite base time, removes the Hessian term; \eqref{eq:time-general} then recovers a Ricci flow. Different initial diffeomorphisms give representatives related by a time-independent diffeomorphism.

\begin{proposition}\label[proposition]{prop:entrance}
Let $g_{\mathrm{RF}}(t)$ be a noncollapsed ancient Ricci flow on a closed manifold. Suppose one tangent flow at backward infinity is the smooth canonical shrinking flow of $(M,\bg,\bfv)$. After a time-independent spatial identification, the construction \eqref{eq:time-general}--\eqref{eq:mrf-definition} gives an ancient MRF such that
\begin{equation}
 \lim_{s\searrow-\infty}\|g(s)-\bg\|_{k}=0
 \quad\text{for every integer }k\ge0.
 \label{eq:geometric-convergence}
\end{equation}
The tangent flow is independent of the blow-down sequence and the base point. In particular, every sufficiently small neighborhood of $\bg$ in a prescribed spatial regularity contains a whole ancient tail of $g$.
\end{proposition}
\begin{proof}
Chan--Ma--Zhang \cite[Theorem 1.1 and Corollary 1.3]{CMZ} prove uniqueness of the tangent flow and convergence of the canonically modified flow, using the \L{}ojasiewicz inequality of Sun--Wang \cite{SW}. Their estimates apply in arbitrarily high fixed spatial regularity, giving \eqref{eq:geometric-convergence}. A time-independent pullback identifies the limiting metric with $\bg$; applying it also to $g_{\mathrm{RF}}$ conjugates \eqref{eq:potential-gauge} and preserves its identity initial value.
\end{proof}

\subsection{The weighted slice and its evolution}
We seek a vector field $Y(h)$ that generates a pullback keeping MRF in $\bg+\Sli$. The slice condition gives an elliptic equation for $Y(h)$; \cref{prop:slice-field} constructs it, and \eqref{eq:smrf-definition} gives the resulting evolution.

All tensor operations in this subsection use $\bg$. Set
\begin{equation}
 \mathrm{d}\nu=(4\pi)^{-n/2}e^{-\bfv}\mathrm{d}V_{\bg},\qquad
 \ip{u}{v}_{L^2}=\int_M\ip{u}{v}_{\bg}\,\mathrm{d}\nu.
 \label{eq:weighted-pairing}\index{d nu@$\mathrm{d}\nu$}
\end{equation}
Throughout the paper, unlabelled spatial $L^2$ pairings and Sobolev norms use $\bg$, its covariant derivative, and the fixed probability measure $\mathrm{d}\nu$.

We identify one-forms and vector fields through $\bg$ and use the same notation for them. The weighted divergence and its $L^2$ adjoint are
\begin{align}
 (\mathrm{div}_{\bfv}h)_i
 &=\bar\nabla^jh_{ji}-\bar\nabla^j\bfv\,h_{ji},\label{eq:weighted-divergence}\index{div weighted@$\mathrm{div}_{\bfv},\mathrm{div}^{*}$}\\
 (\mathrm{div}^{*}\omega)_{ij}
 &=-\frac12(\bar\nabla_i\omega_j+\bar\nabla_j\omega_i)
 =-\frac12(\Lie_{\omega^\sharp}\bg)_{ij}.
 \label{eq:divergence-adjoint}\index{div weighted@$\mathrm{div}_{\bfv},\mathrm{div}^{*}$}
\end{align}
On one-forms we also use $\mathrm{div}_{\bfv}\omega=\bar\nabla^i\omega_i-\bar\nabla^i\bfv\,\omega_i$. Set
\[
 \Sli=\ker\mathrm{div}_{\bfv},\qquad
 \kk=\ker\mathrm{div}^{*},\qquad
 P=\mathrm{div}_{\bfv}\mathrm{div}^{*}.
\]
Here $\kk$ is the finite-dimensional space of background Killing fields. For every integer $k\ge0$,
\begin{equation}
 P:C^{k+2,1/2}(T^*M)\cap\kk^\perp
   \xrightarrow{\ \cong\ }C^{k,1/2}(T^*M)\cap\kk^\perp,
 \qquad
 \|P^{-1}\omega\|_{k+2}\le C_k\|\omega\|_{k}.
 \label{eq:P-inverse}\index{k Killing@$\kk$}\index{P elliptic@$P$}
\end{equation}
Indeed, $P$ is elliptic and $\ip{P\omega}{\omega}_{L^2}=\|\mathrm{div}^{*}\omega\|_{L^2}^2$. Its kernel is $\kk$, so the Fredholm alternative and Schauder estimates give \eqref{eq:P-inverse}; see \cite[Sections 3 and 5]{CMgauge} and \cite[Proposition 2.3]{LZrigid}. Since $\mathrm{div}_{\bfv}T$ is Killing-orthogonal, the weighted orthogonal projection onto $\Sli$ is
\begin{equation}
 \Pi_{\Sli}T=T-\mathrm{div}^{*}P^{-1}\mathrm{div}_{\bfv}T.
 \label{eq:hodge}\index{Pi slice@$\Pi_{\Sli}$}\index{S slice@$\Sli$ (weighted slice)}
\end{equation}

Metric perturbations lie in $C^{3,1/2}(S^2T^*M)$, and the evolution takes values in $C^{1,1/2}(S^2T^*M)$. These fixed spatial regularities are used in the Banach spaces and smallness conditions throughout the paper, independently of the expansion order. Higher spatial regularity is obtained by parabolic bootstrap. We omit bundle labels when the tensor type is clear, interpret $\Sli$ in the displayed regularity, and write
\begin{equation}
 \|u\|_{k}=\|u\|_{C^{k,1/2}(M,\bg)},\qquad
 Lh=\Delta_{\bfv}h+2\Rm_{\bg}(h),
 \quad \Delta_{\bfv}=\bar\Delta-\bar\nabla_{\bar\nabla\bfv}.
 \label{eq:slice-norms-L}\index{Delta weighted@$\Delta_{\bfv}$}\index{L operator@$L$}\index{norm spatial@$\lVert\cdot\rVert_k$}
\end{equation}
For every integer $j\ge0$, define the Banach space
\begin{equation}
 \Comp_j=\Comp\cap C^{j,1/2}(S^2T^*M),\qquad
 \|u\|_{\Comp_j}=\|u\|_{j}.
 \label{eq:complement-spaces}\index{C complement@$\Comp,\Comp_j$}\index{C complement@$\Comp,\Comp_j$}
\end{equation}
In particular, $\Comp_3$ has the $C^{3,1/2}$-norm, and $\Comp_{1}$ has the $C^{1,1/2}$-norm.
Norms on these target spaces are understood in Banach-valued
function-space notation. We use $I_s=[s-1,s]$ for the time intervals in the
strip norms of Appendix~\ref{app:parabolic}.

Li--Zhang \cite[Theorem 2.4]{LZrigid} place a nearby metric in the affine weighted slice $\bg+\Sli$ by a local diffeomorphism, using the elliptic inverse above. The following proposition constructs the vector field that preserves this slice along the modified flow.

\begin{proposition}\label[proposition]{prop:slice-field}\index{Y gauge@$Y(h)$}
On a sufficiently small neighborhood $U$ of zero in $C^{3,1/2}(S^2T^*M)\cap\Sli$, there is a unique small analytic map
\begin{equation}
 Y:U\longrightarrow C^{2,1/2}(TM)\cap\kk^\perp
 \label{eq:Y-map}\index{Y gauge@$Y(h)$}
\end{equation}
such that
\begin{equation}
 \mathrm{div}_{\bfv}\bigl(2\Phi(\bg+h)+\Lie_{Y(h)}(\bg+h)\bigr)=0.
 \label{eq:slice-equation}
\end{equation}
It satisfies $Y(0)=\mathrm{D}Y(0)=0$ and $\|Y(h)\|_{2}\le C\|h\|_{3}^2$.
\end{proposition}
\begin{proof}
The left-hand side of \eqref{eq:slice-equation} is analytic with values in $C^{1/2}\cap\kk^\perp$. Its derivative with respect to $Y$ at $(h,Y)=(0,0)$ is $-2P$. Equation~\eqref{eq:P-inverse}, with $k=0$, therefore gives the analytic implicit-function theorem.

The canonical-potential linearization is \cite[Lemma 2.6]{LZrigid}:
\begin{align*}
 \mathrm{D}\Phi_{\bg}[h]
 &=\frac12Lh+\mathrm{div}^{*}\mathrm{div}_{\bfv}h
   +\bar\nabla^2\bigl(\frac12\tr_{\bg}h-\mathrm{D}f_{\bg}[h]\bigr),\\
 (2\Delta_{\bfv}+1)\bigl(\frac12\tr_{\bg}h-\mathrm{D}f_{\bg}[h]\bigr)
 &=\mathrm{div}_{\bfv}\mathrm{div}_{\bfv}h.
\end{align*}
The scalar operator has zero kernel: if $\Delta_{\bfv}u=-u/2$, the integrated weighted Bochner identity and $\Ric_{\bg}+\bar\nabla^2\bfv=\bg/2$ give $\bar\nabla^2u=0$. Closedness then implies that $u$ is constant, and the eigenvalue equation gives $u=0$. Consequently $2\mathrm{D}\Phi_{\bg}[h]=Lh$ on $\Sli$. Since $L$ preserves $\Sli$ \cite[Lemma 2.8]{LZrigid}, differentiation of \eqref{eq:slice-equation} in $h$ gives $\mathrm{D}Y(0)=0$. The quadratic bound follows from analyticity.

We also record the weak estimate
\begin{equation}
 \|Y(h)\|_{H^1}\le C\|\Phi(\bg+h)\|_{L^2}.
 \label{eq:gauge-weak}
\end{equation}
We first establish the weighted Korn inequality on $\kk^\perp$. Since $P$ is nonnegative, self-adjoint, and elliptic on the closed manifold, with kernel $\kk$, its spectral gap gives
\[
 \|X\|_{L^2}^2\le C\ip{PX}{X}_{L^2}
 =C\|\mathrm{div}^{*}X\|_{L^2}^2,
 \qquad X\perp\kk.
\]
On the normalized shrinker, the identity in \cite[Lemma 3.1]{CMgauge} reads
\[
 -2PX=\Delta_{\bfv}X+\bar\nabla\mathrm{div}_{\bfv}X+\frac12X.
\]
Taking its weighted pairing with $X$ and integrating by parts yields
\[
 \|\bar\nabla X\|_{L^2}^2+\|\mathrm{div}_{\bfv}X\|_{L^2}^2
 =2\|\mathrm{div}^{*}X\|_{L^2}^2+\frac12\|X\|_{L^2}^2.
\]
Combining these two estimates proves
\[
 \|X\|_{H^1}\le C\|\mathrm{div}^{*}X\|_{L^2},
 \qquad X\in H^1(TM)\cap\kk^\perp.
\]
Now put $Y=Y(h)$. Testing \eqref{eq:slice-equation} against $Y$ and using $\Lie_Y\bg=-2\mathrm{div}^{*}Y$ gives
\begin{align*}
 2\|\mathrm{div}^{*}Y\|_{L^2}^2
 &=2\ip{\Phi(\bg+h)}{\mathrm{div}^{*}Y}_{L^2}
   +\ip{\Lie_Yh}{\mathrm{div}^{*}Y}_{L^2}.
\end{align*}
The Lie derivative satisfies
\[
 \|\Lie_Yh\|_{L^2}
 \le C\|h\|_{C^1}\|Y\|_{H^1}
 \le C\|h\|_{C^1}\|\mathrm{div}^{*}Y\|_{L^2}.
\]
Consequently,
\begin{align*}
 2\|\mathrm{div}^{*}Y\|_{L^2}^2
 &\le 2\|\Phi(\bg+h)\|_{L^2}\|\mathrm{div}^{*}Y\|_{L^2}
   +C\|h\|_{C^1}\|\mathrm{div}^{*}Y\|_{L^2}^2.
\end{align*}
Shrinking $U$ makes $\|h\|_{C^1}$ sufficiently small to absorb the last term. Applying the Korn inequality once more proves \eqref{eq:gauge-weak}.
\end{proof}

Define
\begin{equation}
 \mathcal V(h)=2\Phi(\bg+h)+\Lie_{Y(h)}(\bg+h),\qquad
 \partial_sh=\mathcal V(h),\quad h\in\Sli.
 \label{eq:smrf-definition}\index{h@$h$}\index{V field@$\mathcal V$}\index{Special modified Ricci flow (SMRF)}
\end{equation}
This is the special modified Ricci flow (SMRF). The map $\mathcal V$ is analytic from $C^{3,1/2}\cap\Sli$ to $C^{1,1/2}\cap\Sli$, and $\mathrm{D}\mathcal V(0)=L$.

Set $\mathcal U(h)=\mathcal V(h)-Lh$. Since
$\mathcal U(0)=\mathrm{D}\mathcal U(0)=0$, analyticity and Taylor's formula give,
on a sufficiently small ball in $C^{3,1/2}\cap\Sli$,
\begin{equation}
 \|\mathrm{D}^2\mathcal U(h)[u,v]\|_{1} \le C\|u\|_3\|v\|_3,\qquad \|\mathrm{D}\mathcal U(h)[u]\|_{1} \le C\|h\|_3\|u\|_3.
 \label{eq:spatial-nonlinear}\index{U nonlinearity@$\mathcal U(h)$}
\end{equation}
Integrating the second bound along a segment gives the corresponding
quadratic Lipschitz estimate. Taking suprema over $I_s$ gives the
same estimates in the strip norms defined in Appendix~\ref{app:parabolic},
with constants independent of $s$. The weighted linear estimates
used below are proved in that appendix.

We will use the following uniform interior bounds for the solutions
constructed in Sections~\ref{sec:dichotomy} and~\ref{sec:principal}.

\begin{lemma}[Interior smoothing]\label[lemma]{lem:interior-smoothing}
There is $\delta>0$ with the following property. Suppose $h$ solves
SMRF on $[s-2,s]$, belongs to
\[
 C([s-2,s];C^{3,1/2}\cap\Sli)
 \cap C^1([s-2,s];C^{1,1/2}\cap\Sli),
 \qquad \sup_{[s-2,s]}\|h\|_3\le\delta.
\]
Then $h$ is smooth after the initial time, and, for every fixed
integer $k\ge3$,
\begin{equation}
 \sup_{r\in[s-1,s]}
 \bigl(\|h(r)\|_k+\|h_s(r)\|_{k-2}\bigr)
 \le C_k.
 \label{eq:interior-smoothing}
\end{equation}
The constants depend on $k$ and the background, but not on $s$.
Consequently, convergence to zero in $\|\cdot\|_3$ on a tail is
smooth convergence.
\end{lemma}
\begin{proof}
We use the classical Ricci--DeTurck equation and the local elliptic
slice construction in Section~\ref{sec:horizontal}.
For initial data $\tilde g(0)=\bg+h(r_0)$, consider
\[
 \partial_\tau\tilde g
 =\tilde g-2\Ric_{\tilde g}
  +\Lie_{W_{\bg}(\tilde g)-\bar\nabla\bfv}\tilde g,
 \qquad
 W_{\bg}(g)^a=g^{ij}\bigl(\Gamma(g)^a_{ij}-\Gamma(\bg)^a_{ij}\bigr).
\]
This is a strictly parabolic quasilinear system with stationary
solution $\bg$. Classical local existence and interior estimates
\cite{DeTurck,Top} give a fixed $0<\tau_0\le1$ such that, for
sufficiently small $\delta$, the solution remains in the slice
neighborhood on $[0,\tau_0]$ and
\[
 \sup_{\tau\in[\tau_0/2,\tau_0]}
 \|\tilde g(\tau)-\bg\|_k\le C_k,\qquad k\ge3.
\]
Here the lifespan and constants depend only on the background and
the fixed low-norm bound.

Put $\tilde g(\tau)$ into the affine weighted slice using the
local construction in Section~\ref{sec:horizontal}. The displayed
equation differs from MRF by a Lie derivative. Hence the velocity
of its slice representative differs from $\mathcal V$ by
$\Lie_{K_\tau}$, with $K_\tau\in\kk$, by
\eqref{eq:kernel-Ah}. Removing this component by an ODE in the
compact group $G$, starting at the identity, gives an SMRF with
initial value $h(r_0)$. The elliptic slice estimates and the
$G$-invariance of the background norms preserve the displayed
interior bounds. To identify this solution $h_1$ with
$h_2(\tau)=h(r_0+\tau)$, fix $0<\alpha<1/2$ and set
$w=h_1-h_2$ on $J=[0,\tau_0]$. Both solutions are continuous
in $C^{3,\alpha}$ at the initial time, and
\[
 (\partial_\tau-L)w=F,\qquad w(0)=0,\qquad
 F=\mathcal U(h_1)-\mathcal U(h_2).
\]
The zero-initial-data spatial parabolic estimate and
\eqref{eq:spatial-nonlinear} at exponent $\alpha$ give
\[
 A:=\sup_{\tau\in J}\|w(\tau)\|_{C^{3,\alpha}}
 \le C\sup_{\tau\in J}\|F(\tau)\|_{C^{1,\alpha}}
 \le C\delta A.
\]
Choosing $\delta$ so that $C\delta<1$ gives $w=0$.

For each $r\in[s-1,s]$, restart at $r_0=r-\tau_0$. The same
interior estimates give $\|h(r)\|_k\le C_k$, independently of $r$
and $s$. The equation and the elliptic estimates for the potential
and slice field give the bound for $h_s$ in
\eqref{eq:interior-smoothing}. Restarting before any interior time
also proves smoothness there. Finally, interpolation between
$\|h\|_3\to0$ and these uniform higher-order bounds gives
convergence in every fixed spatial norm. The equation then gives
convergence of the time derivatives as well.
\end{proof}
\subsection{Horizontal gauges and geometric quotients}\label{sec:horizontal}\index{G isometries@$G$}
Let $G=\Isom(M,\bg,\bfv)$. Every isometry of $\bg$ preserves its normalized canonical potential, so $G=\Isom(M,\bg)$ and its Lie algebra is $\kk$. The affine slice and the fixed weighted pairing are $G$-invariant.

We first construct local representatives in the affine weighted slice,
as in Li--Zhang \cite[Theorem 2.4]{LZrigid}. Use the local diffeomorphism chart
\[
 \vartheta_X(x)=\exp_x^{\bg}X(x),\qquad
 X\in C^{2,1/2}(TM)\cap\kk^\perp,
\]
and solve
\begin{equation}
 \mathcal F(X,u):=\mathrm{div}_{\bfv}
       \bigl(\vartheta_X^*(\bg+u)-\bg\bigr)=0.
 \label{eq:static-slice-gauge}
\end{equation}
For $u\in C^{3,1/2}(S^2T^*M)$, this is a $C^1$ map with values in
$C^{1/2}(T^*M)\cap\kk^\perp$. Indeed, in coordinates the pullback has
components
\[
 ((\bg+u)_{ab}\circ\vartheta_X)
       \partial_i\vartheta_X^a\partial_j\vartheta_X^b.
\]
After taking the divergence, the expression contains at most two spatial
derivatives of $X$ and one of $u$. Its variation in $X$ contains at most
two derivatives of $u$. The extra derivative in $u\in C^{3,1/2}$ makes
these coefficient variations continuous in $C^{1/2}$, so the H\"older
product and composition estimates give the asserted $C^1$ dependence.
At the origin,
\[
 \mathrm{D}_X\mathcal F(0,0)=-2P:
 C^{2,1/2}(TM)\cap\kk^\perp
 \xrightarrow{\ \cong\ }C^{1/2}(T^*M)\cap\kk^\perp.
\]
Thus \eqref{eq:P-inverse} and the implicit-function theorem give a unique
small correction $X(u)$, with $\|X(u)\|_2\le C\|u\|_3$.
If $u\in\Sli$, then $X(u)=0$.

At a solution, \eqref{eq:static-slice-gauge} is a quasilinear elliptic
equation for $X$, with principal coefficients close to those of $-2P$.
Schauder estimates first give $X\in C^{3,1/2}$ and then
$X\in C^{4,1/2}$, with $\|X(u)\|_4\le C\|u\|_3$.
The slice representative therefore retains $C^{3,1/2}$ regularity.
Applying the same construction in higher spatial regularity and
differentiating its elliptic equation shows that smooth families of
smooth metrics have smoothly varying smooth representatives.

We now state the strong uniqueness property of the affine weighted slice.
\begin{lemma}\label[lemma]{lem:strong-affine-slice}
There exists $\epsilon>0$ such that, whenever $h_1,h_2\in C^{3,1/2}(S^2T^*M)\cap\Sli$ and $\|h_i\|_{3}<\epsilon$ for $i=1,2$, we have
\begin{equation}
 a^*(\bg+h_1)=\bg+h_2,\quad a\in\Diff(M)
 \quad\Longrightarrow\quad a\in G.
 \label{eq:strong-affine-slice}
\end{equation}
\end{lemma}
\begin{proof}
First put every such $a$ in a prescribed $C^{2,1/2}$ neighborhood of $G$. Otherwise, take $h_{1,j},h_{2,j}\to0$ in $C^{3,1/2}$ and isometries $a_j$ outside that neighborhood. The isometry equation bounds their first derivatives, and the connection transformation identity gives uniform $C^{4,1/2}$ bounds for $a_j$ and their inverses. A subsequence therefore converges in $C^{2,1/2}$ to a background isometry, a contradiction. This is the usual properness argument for the diffeomorphism action \cite{Ebin}. Near $G$, a finite-dimensional implicit-function argument for the Killing projection gives $a=b\circ\vartheta_X$, with $b\in G$ and $X\perp\kk$ small. Then
\[
 \vartheta_X^*(\bg+b^*h_1)=\bg+h_2.
\]
Since both $b^*h_1$ and $h_2$ lie in $\Sli$, $X$ and zero both solve \eqref{eq:static-slice-gauge} with $u=b^*h_1$. Local uniqueness gives $X=0$, hence $a=b\in G$.
\end{proof}

For $h\in\Sli$, set $A_hX=\mathrm{div}_{\bfv}\Lie_X(\bg+h)$. Its restriction from $C^{2,1/2}\cap\kk^\perp$ to $C^{1/2}\cap\kk^\perp$ is invertible for small $h$. Moreover,
\begin{equation}
 \ker A_h=\kk.
 \label{eq:kernel-Ah}
\end{equation}
A background Killing field preserves $\Sli$, so $\mathrm{div}_{\bfv}\Lie_Kh=0$ and $A_hK=0$. Splitting a vector field into $\kk\oplus\kk^\perp$ proves the reverse inclusion.

For a convergent ancient MRF, the local construction gives a smooth family $\chi_s$ on an ancient tail with
\begin{equation}
 \chi_s^*g(s)=\bg+\hat h(s),\qquad
 \hat h(s)\in\Sli,\qquad \hat h(s)\longrightarrow0.
 \label{eq:pointwise-slice}
\end{equation}
These are its slice representatives at each time, before any velocity condition is imposed. If $X_s=\mathrm{d}\chi_s^{-1}(\partial_s\chi_s)$ is the pulled-back velocity, then
\begin{equation}
 \hat h_s=2\Phi(\bg+\hat h)+\Lie_{X_s}(\bg+\hat h),\qquad
 X_s=Y(\hat h)+K_s,\quad K_s\in\kk.
 \label{eq:slice-velocity}
\end{equation}
The second identity follows from \eqref{eq:slice-equation} and \eqref{eq:kernel-Ah}.

A \emph{horizontal gauge} is a family $\psi_s$ satisfying
\begin{equation}
 \psi_s^*g(s)=\bg+h(s),\qquad h(s)\in\Sli,\qquad
 \mathrm{d}\psi_s^{-1}(\partial_s\psi_s)=Y(h(s)).
 \label{eq:horizontal-gauge}\index{Horizontal gauge}
\end{equation}
To construct it, solve $\partial_sa_s=-K_s\circ a_s$ in the compact Lie group $G$, with $a_{s_*}=\Id$ at a finite base time. Set $\psi_s=\chi_s\circ a_s$ and $h=a_s^*\hat h$. Uniqueness in \cref{prop:slice-field} gives
\begin{equation}
 Y(a^*h)=a^*Y(h),\qquad a\in G.
 \label{eq:Y-equiv}
\end{equation}
The composition rule for velocities now gives \eqref{eq:horizontal-gauge} and $h_s=\mathcal V(h)$. Since $G$ is compact, $h(s)\to0$ in every fixed spatial regularity.

\begin{theorem}[Geometric correspondence]\label[theorem]{thm:correspondence}\index{Ancient germ}
The gauge construction in \eqref{eq:horizontal-gauge} induces a bijection between ancient MRF germs converging smoothly to a metric isometric to $\bg$, modulo time-independent diffeomorphisms, and ancient SMRF germs converging smoothly to zero, modulo $G$.
\end{theorem}
\begin{proof}
The construction above gives the map. We first check that its value is independent of the horizontal gauge. If $\psi_s$ and $\widetilde\psi_s$ are horizontal gauges for the same MRF, put $a_s=\psi_s^{-1}\circ\widetilde\psi_s$. Their slice representatives satisfy
\[
 a_s^*(\bg+h(s))=\bg+\widetilde h(s).
\]
On a sufficiently remote tail, \cref{lem:strong-affine-slice} applies to these representatives and gives $a_s\in G$. The composition rule for velocities reads
\[
 Y(\widetilde h)=a_s^*Y(h)+\mathrm{d}a_s^{-1}(\partial_sa_s).
\]
Since $\widetilde h=a_s^*h$, equivariance \eqref{eq:Y-equiv} gives $\partial_sa_s=0$. Thus the two SMRFs differ by one time-independent element of $G$. A time-independent change of the MRF representative is absorbed by left composition of its gauge.

For the inverse, let $h_s=\mathcal V(h)$ with $h(s)\to0$ smoothly and solve
\[
 \partial_s\phi_s=-Y(h(s))\circ\phi_s,\qquad \phi_{s_*}=\Id.
\]
Then $g(s)=\phi_s^*(\bg+h(s))$ satisfies MRF. We verify its convergence. Set
\[
 e(s)=\boldsymbol{\mu}(\bg+h(s),1)-\boldsymbol{\mu}(\bg,1),\qquad
 v(s)=\|\Phi(\bg+h(s))\|_{L^2}.
\]
Entropy monotonicity and the \L{}ojasiewicz inequality \cite{SW,CMZ} give $e\ge0$, $e\to0$, and
\[
 c\,v(s)^2\le e_s(s)\le C\,v(s)^2,
 \qquad e(s)^{\theta}\le C\,v(s),\qquad \frac12\le\theta<1.
\]
The Lie derivative in SMRF contributes zero to the entropy derivative by diffeomorphism invariance. If $e\equiv0$, then $\Phi=Y=0$. Otherwise choose $\epsilon>0$ with $(1+\epsilon)\theta<1$. Changing variables from $s$ to $e$ yields
\begin{equation}
 \int_{-\infty}^{s_*}v(s)^{1-\epsilon}\dd s
 \le C\int_0^{e(s_*)}z^{-(1+\epsilon)\theta}\dd z<\infty.
 \label{eq:inverse-gauge-length}
\end{equation}
The weak estimate \eqref{eq:gauge-weak}, interpolation, and the uniform higher spatial bounds supplied by $h\to0$ smoothly imply, for each fixed $j$,
\[
 \|Y(h(s))\|_{C^j}\le C_j v(s)^{1-\epsilon}.
\]
Thus $\phi_s$ converges smoothly to a diffeomorphism $\phi_{-\infty}$, and $g(s)\to\phi_{-\infty}^*\bg$. A time-independent identification gives the required limit.

The inverse construction respects $G$ by \eqref{eq:Y-equiv}, and changing its initial diffeomorphism changes $g$ by a time-independent pullback. Starting with an MRF and undoing its horizontal gauge returns its class: the inverse gauge solves the displayed ODE for $\phi_s$. Starting with an SMRF, the inverse of $\phi_s$ is a horizontal gauge for the reconstructed MRF. Hence the two maps are inverse. Finally, any time-independent diffeomorphism relating two SMRFs converging to zero preserves $\bg$ by passage to the limit and therefore belongs to $G$.
\end{proof}

Two closed ancient Ricci flows that agree after a time-independent diffeomorphism on an ancient tail agree on their entire common interval: choose a time on the tail and apply Ricci-flow uniqueness \cite{Hamilton}. Their maximal smooth continuations therefore coincide.

Conversely, an SMRF tail with $h\to0$ reconstructs a Ricci-flow ancient germ by the two finite-time diffeomorphism integrations above. To identify its tangent flow, shift a sequence of modified times to zero and normalize the reconstruction diffeomorphisms at those times. On each bounded shifted interval the coefficients converge smoothly to those of the canonical shrinker, and continuous dependence of the diffeomorphism ODEs gives its canonical shrinking flow as the limit.

The reconstructed flow is noncollapsed: the scale-invariant $\boldsymbol{\nu}$-functional is uniformly bounded below on metrics near the compact shrinker by their positive scalar curvature and uniform Sobolev bounds \cite[Appendix B]{CMZ}, and Perelman's monotonicity propagates this bound forward from the ancient tail.

The constructions and the induced spatial quotients are summarized by
\begin{equation}
 \mathrm{RF}
 \xleftrightarrow{\ \text{normalization and potential gauge}\ }
 \mathrm{MRF}
 \xleftrightarrow{\ \text{horizontal gauge}\ }
 \mathrm{SMRF},
 \label{eq:flow-correspondence-diagram}
\end{equation}
\begin{equation}
 \mathrm{RF}/\Diff(M)
 \ \cong\ \mathrm{MRF}/\Diff(M)
 \ \cong\ \mathrm{SMRF}/G.
 \label{eq:quotient-correspondence-diagram}
\end{equation}
Here RF consists of Ricci-flow germs on intervals $(-\infty,t_*]$, $t_*<0$, with the prescribed tangent shrinker; MRF and SMRF consist of the convergent germs in \cref{thm:correspondence}. The first diagram uses finite-time choices of gauges; the second is independent of these choices. The Ricci-flow reference time is zero, and $\Diff(M)$ acts by time-independent diffeomorphisms.

\section{The exponential--neutral dichotomy}\label{sec:dichotomy}\index{b@$b$}\index{E positive negative@$E_+,E_-$}\index{K neutral@$K$}\index{L operator@$L$}\index{lambda gaps@$\lambda_+,\lambda_-$}\index{Pi spectral@$\Pi_K,\Pi_+,\Pi_-$}
Throughout this section $h_s=\mathcal V(h)$ is an ancient SMRF with $h\to0$ smoothly as $s\searrow-\infty$. The self-adjoint elliptic operator $L|_{\Sli}$ has a discrete spectrum. Let $E_+$ be the sum of its positive eigenspaces, $K=\ker(L|_{\Sli})$, and $E_-$ its negative spectral subspace. Define
\[
 \lambda_+=\min\bigl(\Spec(L|_{\Sli})\cap(0,\infty)\bigr),\qquad
 -\lambda_-=\max\bigl(\Spec(L|_{\Sli})\cap(-\infty,0)\bigr).
\]
Thus $\lambda_+>0$ is the smallest positive eigenvalue, and $-\lambda_-<0$ is the largest negative eigenvalue. Let $\Pi_+,\Pi_K,\Pi_-$ be the corresponding orthogonal projections for the fixed weighted $L^2$ pairing, and write
\begin{equation}
 h=h_++b+h_-,\qquad
 h_\pm=\Pi_\pm h,\qquad b=\Pi_Kh,\qquad
 \Comp=E_+\oplus E_-.
\end{equation}
The spaces $E_+$ and $K$ are finite-dimensional and consist of smooth tensors. The identity $L\Ric_{\bg}=\Ric_{\bg}$ shows that $E_+$ is nonzero.

\subsection{Spectral comparison}
Set
\[
 U_+=\|h_+\|_{L^2},\qquad U_0=|b|,\qquad U_-=\|h_-\|_{L^2}.
\]
The following estimates hold for a small smooth SMRF on any time interval.
We will also use them for forward convergence in Section~\ref{sec:forward}.

\begin{lemma}\label[lemma]{lem:projected-inequalities}
There are $\delta_0>0$ and $C$ such that
\begin{equation}
 \|\mathcal U(h)\|_{L^2}\le C\|h\|_3\|h\|_{L^2},
 \qquad h\in\Sli,\quad \|h\|_3\le\delta_0.
 \label{eq:L2-nonlinear}
\end{equation}
Along a smooth solution $h_s=\mathcal V(h)$ in this neighborhood, put
$\epsilon(s)=C\|h(s)\|_3$. Then
\begin{align}
 \partial_sU_+
 &\ge\lambda_+U_+-\epsilon(s)(U_++U_0+U_-),
 \label{eq:ODIplus}\\*
 |\partial_sU_0|
 &\le |b_s|\le\epsilon(s)(U_++U_0+U_-),
 \label{eq:ODIzero}\\*
 \partial_sU_-
 &\le-\lambda_-U_-+\epsilon(s)(U_++U_0+U_-),
 \label{eq:ODIminus}\\
 |\partial_sU_+|
 &\le C U_++\epsilon(s)(U_++U_0+U_-).
 \label{eq:ODIplus-upper}
\end{align}
The norms $U_+,U_0,U_-$ are locally Lipschitz. The inequalities hold almost
everywhere, with the corresponding one-sided interpretation at zeros.
\end{lemma}
\begin{proof}
Put $\delta=\|h\|_3$. The local quadratic remainders contain products
such as $h*\bar\nabla^2h$ and $\bar\nabla h*\bar\nabla h$.
The first has $L^2$ norm at most $C\delta\|h\|_{L^2}$.
For the second, weighted integration by parts gives
\[
 \int_M|\bar\nabla h|^4\,\mathrm{d}\nu
 =-\int_M\left\langle h,
   \mathrm{div}_{\bfv}\bigl(|\bar\nabla h|^2\bar\nabla h\bigr)
   \right\rangle\,\mathrm{d}\nu
 \le C\delta\|h\|_{L^2}\|\bar\nabla h\|_{L^4}^2.
\]
Here the differentiated factors are bounded by
$C(\|\bar\nabla^2h\|_{L^\infty}+\|\bar\nabla h\|_{L^\infty})
|\bar\nabla h|^2$, with fixed coefficients from $\bfv$.
Consequently
\[
 \|\bar\nabla h\|_{L^4}^2\le C\delta\|h\|_{L^2}.
\]
Lower-order quadratic products satisfy the same bound.

By \cref{prop:slice-field}, $\mathrm{D}f_{\bg}[h]=\frac12\tr_{\bg}h$ on $\Sli$.
Set
$p=f_{\bg+h}-\bfv-\frac12\tr_{\bg}h$.
Expand the potential's Euler--Lagrange equation and its normalization
to first order. The augmented linear operator on $p$ and the entropy
constant is invertible from $H^2\times\R$ to $L^2\times\R$,
by the scalar invertibility in \cref{prop:slice-field} and the normalization.
The quadratic terms involving only $h$ have the bounds just proved.
The remaining terms containing $p$ are bounded by $C\delta\|p\|_{H^2}$,
since $\|f_{\bg+h}-\bfv\|_2\le C\delta$.
The normalization remainder satisfies the corresponding
scalar bound. Elliptic inversion and absorption yield
\[
 \|p\|_{H^2}\le C\delta\|h\|_{L^2}.
\]
Expanding the Ricci tensor and the potential Hessian therefore gives
\[
 E(h):=\Phi(\bg+h)-\frac12Lh,\qquad
 \|E(h)\|_{L^2}\le C\delta\|h\|_{L^2}.
\]

Since $\mathrm{div}_{\bfv}Lh=0$, the slice equation is
\[
 -2PY=-2\mathrm{div}_{\bfv}E(h)
       -\mathrm{div}_{\bfv}\Lie_Yh.
\]
Elliptic inversion on $\kk^\perp$, from $H^{-1}$ to $H^1$, and
$\|\Lie_Yh\|_{L^2}\le C\delta\|Y\|_{H^1}$ give
\[
 \|Y\|_{H^1}\le C\bigl(\|E(h)\|_{L^2}+\delta\|Y\|_{H^1}\bigr).
\]
After absorption, $\|Y\|_{H^1}\le C\delta\|h\|_{L^2}$.
The identity $\mathcal U(h)=2E(h)+\Lie_Y(\bg+h)$ proves
\eqref{eq:L2-nonlinear}.

Project $h_s=Lh+\mathcal U(h)$ onto $E_+,K,E_-$.
Pairing the first and third equations with $h_+/U_+$ and $h_-/U_-$
where these norms are positive gives \eqref{eq:ODIplus} and
\eqref{eq:ODIminus}, by the spectral gaps and \eqref{eq:L2-nonlinear}.
The neutral equation gives \eqref{eq:ODIzero}.
Since $L|_{E_+}$ is bounded, the positive equation also gives
\eqref{eq:ODIplus-upper}. Smoothness implies local Lipschitz continuity
of the norms, and their one-sided derivatives at zero are bounded by
the norm of the projected velocity.
\end{proof}

\begin{lemma}\label[lemma]{lem:two-mode}
For a nonzero ancient solution as above,
\begin{equation}
 U_-=o(U_++U_0)\qquad(s\searrow-\infty).
 \label{eq:ancient-negative}
\end{equation}
Moreover, either $U_++U_0=O(e^{\lambda_+s/2})$, or
$U_+=o(U_0)$ as $s\searrow-\infty$.
\end{lemma}
\begin{proof}
Fix $\delta>0$ and put $F_\delta=U_--\delta(U_++U_0)$.
At a zero with $U_++U_0>0$, \cref{lem:projected-inequalities} gives
\[
 \partial_sF_\delta
 \le\bigl[-\lambda_-\delta+(1+2\delta)(1+\delta)\epsilon(s)\bigr]
       (U_++U_0)<0
\]
on a sufficiently negative tail. Thus $F_\delta\ge0$ is backward
invariant there. In this region,
$U_++U_0+U_-\le(1+\delta^{-1})U_-$, so, after moving the tail farther
backward, $\partial_sU_-\le-\lambda_-U_-/2$.
A positive value would then grow exponentially toward backward
infinity, contrary to $h\to0$. Hence $U_-\le\delta(U_++U_0)$
on that tail. Since $\delta$ is arbitrary, this proves
\eqref{eq:ancient-negative}.

Taking $U_-\le U_++U_0$ and enlarging $\epsilon(s)$, the first two
projected inequalities become
\[
 \partial_sU_+\ge(\lambda_+-\epsilon)U_+-\epsilon U_0,\qquad
 |\partial_sU_0|\le\epsilon(U_++U_0).
\]
Together with \eqref{eq:ODIplus-upper}, they give
$|\partial_s(U_++U_0)|\le C(U_++U_0)$.
Gronwall's inequality shows that $U_++U_0$ is positive on the tail:
otherwise it vanishes identically, and so does $h$.

Now set $F_\delta=U_+-\delta U_0$. At each zero,
\[
 \partial_sF_\delta
 \ge\bigl[\lambda_+\delta-(1+\delta)^2\epsilon(s)\bigr]U_0>0
\]
on a sufficiently negative tail. If $F_\delta\ge0$ on an entire
ancient tail for some $\delta$, then $U_0\le\delta^{-1}U_+$ and
$\partial_sU_+\ge\lambda_+U_+/2$ farther backward.
Integrating to a fixed terminal time proves
$U_++U_0=O(e^{\lambda_+s/2})$.
Otherwise, the crossing property forces $F_\delta<0$ on a sufficiently
negative tail for every $\delta>0$, giving $U_+=o(U_0)$.
\end{proof}
\begin{proposition}\label[proposition]{prop:dichotomy}\index{Slow solutions@Slow solutions (neutral-dominated)}
Apart from $h\equiv0$, exactly one of the following alternatives holds on a sufficiently negative tail:
\begin{equation}
 \|h(s)\|_{3}=O(e^{\lambda_+s/4});
 \label{eq:fast}
\end{equation}
or
\begin{equation}
 b(s)\ne0,\qquad
 \|\Pi_{\Comp}h(s)\|_{3}=o(|b(s)|).
 \label{eq:slow}
\end{equation}
In the second case, there are constants $C,c>0$ such that
\begin{equation}
 |b_s|\le C|b|^2,\qquad |b(s)|\ge\frac{c}{|s|},\qquad
 |(\log|b|)_s|\longrightarrow0.
 \label{eq:slow-tempered}
\end{equation}
Moreover $\Pi_{\Comp}h=O(|b|^2)$ in every fixed spatial norm.
\end{proposition}
\begin{proof}
In the first case of \cref{lem:two-mode}, \eqref{eq:ancient-negative}
gives $\|h(s)\|_{L^2}=O(e^{\lambda_+s/2})$.
Fix a spatial norm $C^{k,1/2}$. Smooth convergence gives a uniform
bound in $C^j$ for arbitrarily large $j$, and interpolation gives
\[
 \|h(s)\|_k
 \le C\|h(s)\|_{L^2}^{1-\theta}\|h(s)\|_{C^j}^{\theta},
 \qquad 0<\theta\le\frac12,
\]
after choosing $j$ sufficiently large. Hence
$\|h(s)\|_k=O(e^{\lambda_+s/4})$.
The equation gives the same decay for $h_s$ in the spatial forcing
norm, and thus for $\|h\|_{\XP(I_s)}$.

In the other case, $U_++U_-=o(|b|)$ and $b\ne0$ on a tail.
Equation~\eqref{eq:ODIzero} gives $|b_s|=o(|b|)$, so
$|\partial_s\log|b||\to0$.
Every fixed power of $|b|$ therefore satisfies
\eqref{eq:append-tempered} with arbitrarily small exponent on a
sufficiently negative tail. In particular, for every $\eta>0$,
\[
 |b(s)|\ge |b(s_*)|e^{\eta(s-s_*)},\qquad s\le s_*,
\]
after decreasing $s_*$.

We first improve the $L^2$ estimate before interpolating the complement.
Since $\|h\|_{L^2}=O(|b|)$, interpolation with smooth convergence and
\eqref{eq:L2-nonlinear} give, for every sufficiently small $\theta>0$,
\[
 \|h(s)\|_3\le C_\theta|b(s)|^{1-\theta},\qquad
 \|\mathcal U(h(s))\|_{L^2}\le C_\theta|b(s)|^{2-\theta}.
\]
Put $q=\Pi_{\Comp}h$. Its equation is
$q_s=L_{\Comp}q+\Pi_{\Comp}\mathcal U(h)$.
The fixed-operator formula \eqref{eq:append-G0} gives
\[
 q(s)=e^{(s-s_*)L_{\Comp}}h_+(s_*)
       +\Gop_0\bigl(\Pi_{\Comp}\mathcal U(h)\bigr)(s).
\]
Apply \eqref{eq:append-L2} with $w=|b|^{2-\theta}$.
The positive homogeneous term is smaller than this weight by the
preceding lower bound for $|b|$. Thus
$\|q(s)\|_{L^2}\le C_\theta|b(s)|^{2-\theta}$.
For each fixed $k$, another interpolation with a uniformly bounded
higher spatial norm yields
\[
 \|q(s)\|_k\le C|b(s)|^{(2-\theta)(1-\theta')}=o(|b(s)|),
\]
where $\theta,\theta'>0$ are chosen so that
$(2-\theta)(1-\theta')>1$.
This proves \eqref{eq:slow} and gives $\|h(s)\|_k\le C_k|b(s)|$
at every fixed spatial order.

The spatial nonlinear estimate \eqref{eq:spatial-nonlinear} now gives
$\|\Pi_{\Comp}\mathcal U(h)\|_{\YP(I_s)}\le C|b(s)|^2$,
using comparability of $|b|$ on unit intervals.
Apply the fixed-operator regularity bound
\eqref{eq:append-fixed-bound} to the same formula for $q$,
now with $w=|b|^2$. The homogeneous term is again smaller than the
weight, so $\|q\|_{\XP(I_s)}\le C|b(s)|^2$.
The argument applies at every fixed spatial order.

Finally, \eqref{eq:L2-nonlinear} and $\|h\|_3\le C|b|$ give
$|b_s|\le C|b|^2$. For a fixed time $s_0$ on this tail,
\[
 |b(s)|\ge
 \frac1{|b(s_0)|^{-1}+C(s_0-s)}
 \ge\frac{c}{|s|},\qquad s\le s_0,
\]
after decreasing $s_0$.
This lower bound excludes the exponential alternative and proves
disjointness.
\end{proof}

\subsection{Parametrization of exponentially convergent solutions}
Fix a terminal time $s_*\le-2$ and
$0<\gamma<\min\{\lambda_-,\lambda_+/4\}$.
Use the strip spaces $\XP(I_s),\YP(I_s)$ of \eqref{eq:append-norms} on fixed-slice paths, and define
\begin{equation}
 \|h\|_{\XP_\gamma^{\exp}}
 =\sup_{s\le s_*}e^{-\gamma(s-s_*)}\|h\|_{\XP(I_s)},\qquad
 \|R\|_{\YP_\gamma^{\exp}}
 =\sup_{s\le s_*}e^{-\gamma(s-s_*)}\|R\|_{\YP(I_s)}.
 \label{eq:fast-spaces}\index{X Y exponential@$\XP_\gamma^{\exp},\YP_\gamma^{\exp}$}
\end{equation}
Thus these spaces measure exponential decay at backward infinity.

The parameter is the terminal positive component
$\xi=\Pi_+h(s_*)\in E_+$.
We construct a map from a neighborhood of zero in $E_+$ to ancient paths,
$\xi\mapsto h_\xi$, whose inverse is $h\mapsto\Pi_+h(s_*)$.
The terminal metrics form a graph over $E_+$, with graph values in
$K\oplus E_-$.
This is the fixed-slice version of the construction and classification of exponentially convergent solutions in
\cite[Theorems 1.3 and 6.1]{CL}.

For $\xi\in E_+$, the equation and the required ancient behavior lead to
\begin{equation}
 h(s)={}e^{(s-s_*)L}\xi +\int_{-\infty}^{s}e^{(s-u)L}(\Pi_K+\Pi_-)\mathcal U(h(u))\dd u -\int_s^{s_*}e^{(s-u)L}\Pi_+\mathcal U(h(u))\dd u.
 \label{eq:fast-LP}
\end{equation}
The two integrals define a linear operator $\mathcal G R$ when
$\mathcal U(h)$ is replaced by a forcing $R$.

\begin{theorem}[Exponential parametrization]\label[theorem]{thm:fast-classification}\index{psi exponential graph@$\psi$ (exponential graph)}\index{E positive negative@$E_+,E_-$}
There are $\rho,\epsilon>0$ such that every $\xi\in E_+$ with $|\xi|<\epsilon$
determines a unique solution $h_\xi$ of \eqref{eq:fast-LP} with
$\|h_\xi\|_{\XP_\gamma^{\exp}}<\rho$.
The map $\xi\mapsto h_\xi$ is smooth and $G$-equivariant, and
\[
 h_\xi(s_*)=\xi+\psi(\xi),\qquad
 \psi(\xi)\in K\oplus E_-,\qquad \psi(0)=\mathrm{D}\psi(0)=0.
\]
Its terminal values form a local manifold of real dimension $\dim E_+$,
tangent to $E_+$ at zero.
Every exponentially convergent ancient SMRF belongs to this family after
restriction to a sufficiently negative tail.
\end{theorem}
\begin{proof}
We first justify the integrals. For $R\in\YP_\gamma^{\exp}$, put
$\bar R=\sup_{u\le s_*}e^{-\gamma(u-s_*)}\|R(u)\|_{L^2}$.
The spectral theorem gives
\begin{align*}
 e^{-\gamma(s-s_*)}
 \int_{-\infty}^s\|e^{(s-u)L}\Pi_-R(u)\|_{L^2}\dd u
 &\le \bar R\int_0^\infty e^{-(\lambda_-+\gamma)v}\dd v
 =\frac{\bar R}{\lambda_-+\gamma},\\
 e^{-\gamma(s-s_*)}
 \int_{-\infty}^s|\Pi_KR(u)|\dd u
 &\le \bar R\int_0^\infty e^{-\gamma v}\dd v
 =\frac{\bar R}{\gamma},\\
 e^{-\gamma(s-s_*)}
 \int_s^{s_*}\|e^{(s-u)L}\Pi_+R(u)\|_{L^2}\dd u
 &\le \bar R\int_0^{s_*-s}e^{-(\lambda_+-\gamma)v}\dd v
 \le\frac{\bar R}{\lambda_+-\gamma}.
\end{align*}
Thus the improper integrals converge absolutely in $L^2$.
Testing the resulting formula against smooth spatial tensors and compactly
supported time test functions gives
$(\partial_s-L)\mathcal G R=R$ in distributions.
The local parabolic estimate \eqref{eq:append-regularity}, which also applies
to $L$ on the full fixed slice, then gives
\begin{equation}
 \|\mathcal G R\|_{\XP_\gamma^{\exp}}
 \le C_\gamma\|R\|_{\YP_\gamma^{\exp}}.
 \label{eq:fast-green-bound}
\end{equation}
The finite-dimensional term $S\xi(s)=e^{(s-s_*)L}\xi$ satisfies
$\|S\xi\|_{\XP_\gamma^{\exp}}\le C_0|\xi|$.

On the closed ball of radius $\rho$ in $\XP_\gamma^{\exp}$,
Taylor's formula, \eqref{eq:spatial-nonlinear}, and
$e^{\gamma(s-s_*)}\le1$ give
\[
 \|\mathcal U(h)\|_{\YP_\gamma^{\exp}}\le C_1\rho^2,\qquad
 \|\mathcal U(h)-\mathcal U(\tilde h)\|_{\YP_\gamma^{\exp}}
 \le2C_1\rho\|h-\tilde h\|_{\XP_\gamma^{\exp}}.
\]
Choose $\rho$ so that $2C_\gamma C_1\rho<1/2$, and then take
$\epsilon\le\rho/(2C_0)$. The map
$h\mapsto S\xi+\mathcal G\mathcal U(h)$ maps this ball into its interior
and has Lipschitz constant less than $1/2$. The contraction theorem gives
$h_\xi$, which solves SMRF and has $\Pi_+h_\xi(s_*)=\xi$.
\Cref{lem:interior-smoothing} gives uniform higher-order bounds and
makes $h_\xi$ converge smoothly to zero on its ancient tail.
Its exponential decay in $\|\cdot\|_3$ rules out the neutral case
of \cref{prop:dichotomy}, which therefore gives exponential decay
in every fixed spatial norm.

The nonlinear map is smooth between the displayed strip spaces, by the
potential and gauge elliptic equations and their successive variations.
The parameter-dependent contraction theorem gives smooth dependence on
$\xi$. At $\xi=0$, $\mathrm{D}\mathcal U(0)=0$, so
$\mathrm{D}_\xi h_0[\xi]=S\xi$. Evaluation at $s_*$ gives the graph statement and
$\mathrm{D}\psi(0)=0$. The terminal projection supplies the inverse, so this graph
is a local manifold of the stated dimension.

Let $h$ be any exponentially convergent ancient SMRF. By
\cref{prop:dichotomy} and its proof,
$\|h\|_{\XP(I_s)}\le Ce^{\lambda_+s/4}$ on a sufficiently negative tail.
Since $\gamma<\lambda_+/4$, its weighted norm on $(-\infty,s_*]$ satisfies
\[
 \|h\|_{\XP_\gamma^{\exp}}
 \le Ce^{\lambda_+s_*/4}
       \sup_{s\le s_*}e^{(\lambda_+/4-\gamma)(s-s_*)}
 =Ce^{\lambda_+s_*/4}.
\]
Choose $s_*$ sufficiently negative so that
$\|h\|_{\XP_\gamma^{\exp}}<\rho$ and
$|\Pi_+h(s_*)|<\epsilon$, and set $\xi=\Pi_+h(s_*)$.

We derive \eqref{eq:fast-LP} from $h_s=Lh+\mathcal U(h)$.
For $a<s$, variation of constants on $K\oplus E_-$ gives
\[
 (\Pi_K+\Pi_-)h(s) ={}e^{(s-a)L}(\Pi_K+\Pi_-)h(a) +\int_a^s e^{(s-u)L}(\Pi_K+\Pi_-) \mathcal U(h(u))\dd u.
\]
Since $L$ has nonpositive spectrum on this subspace,
\[
 \bigl\|e^{(s-a)L}(\Pi_K+\Pi_-)h(a)\bigr\|_{L^2}
 \le\|h(a)\|_{L^2}\longrightarrow0
 \qquad(a\searrow-\infty).
\]
The quadratic estimate for $\mathcal U$ and exponential decay of $h$ imply
\[
 \int_{-\infty}^{s_*}\|\mathcal U(h(u))\|_{L^2}\dd u<\infty.
\]
The same semigroup bound therefore justifies taking $a\searrow-\infty$
in the integral. We obtain
\[
 (\Pi_K+\Pi_-)h(s)
 =\int_{-\infty}^s e^{(s-u)L}(\Pi_K+\Pi_-)
      \mathcal U(h(u))\dd u.
\]
On the finite-dimensional space $E_+$,
\[
 \frac{\mathrm{d}}{\mathrm{d}u}
 \left(e^{(s-u)L}\Pi_+h(u)\right)
 =e^{(s-u)L}\Pi_+\mathcal U(h(u)).
\]
Integrating from $s$ to $s_*$ and using $\Pi_+h(s_*)=\xi$ gives
\[
 \Pi_+h(s)
 =e^{(s-s_*)L}\xi
  -\int_s^{s_*}e^{(s-u)L}\Pi_+\mathcal U(h(u))\dd u.
\]
Adding the two projected formulas gives \eqref{eq:fast-LP}.
Since $h$ lies in the contraction ball with parameter $\xi$,
uniqueness gives $h=h_\xi$.

Finally, $G$ preserves $L$, its spectral projections, and all fixed
background norms. Equivariance of $\mathcal V$ makes \eqref{eq:fast-LP} equivariant,
and uniqueness gives $h_{g\cdot\xi}=g\cdot h_\xi$.
\end{proof}

\begin{corollary}[Linearly stable case]\label[corollary]{cor:fast-Ricci}
If $E_+=\R\Ric_{\bg}$, every exponentially convergent germ, after quotienting original-time translation, is the germ of the canonical shrinking Ricci flow.
\end{corollary}
\begin{proof}
The original-time translates of the canonical shrinking Ricci flow give a one-dimensional family through zero with tangent $-2e^s\Ric_{\bg}$. Its terminal positive derivative is nonzero. By the inverse-function theorem its image fills the local manifold of exponentially convergent solutions, and restriction to sufficiently ancient tails gives the assertion for germs.
\end{proof}

\section{Finite-order expansion}\label{sec:normal}
The static graph relates the shrinker defect to an analytic entropy on $K$ and will be used in Section~\ref{sec:exponent} to determine the decay rate. For the evolution, we construct an approximate invariant graph whose invariance defect vanishes through any prescribed finite order. An estimate for the complement equation then turns this construction into the expansion of an actual slow solution in \cref{thm:actual-expansion}.

\subsection{Static Kuranishi reduction}
The static graph is an analytic map $W:K\supset U_K\to\Comp_3$ with $W(0)=\mathrm{D}W(0)=0$ such that
\begin{equation}
 G_b=\bg+b+W(b),\qquad \Pi_{\Comp}\Pi_{\Sli}\Phi(G_b)=0.\label{eq:static-graph}\index{Pi complement@$\Pi_{\Comp}$}\index{G static metric@$G_b$}\index{O obstruction@$\Ob$}\index{W static graph@$W(b)$}
\end{equation}
It follows by the elliptic implicit-function theorem, since the derivative in $W$ is $L|_{\Comp}/2$.

\begin{lemma}\label[lemma]{lem:actual-defect}
There are analytic maps $\Ob(b)\in K$ and $\omega_b\in\kk^\perp$ such that
\begin{equation}
 \Phi(G_b)=\Ob(b)+\mathrm{div}^{*}\omega_b,\qquad
 \|\mathrm{div}^{*}\omega_b\|_{1}\le C|b|\,|\Ob(b)|.\label{eq:actual-defect}
\end{equation}
Furthermore, with $h_b=b+W(b)$,
\begin{align}
 \|Y(h_b)\|_{2}&\le C|b|\,|\Ob(b)|,\label{eq:gauge-static}\\
 \|Y(h_b+q)-Y(h_b)\|_{2}
 &\le C(|b|+\|q\|_{3})\|q\|_{3}.\label{eq:gauge-difference}
\end{align}
\end{lemma}
\begin{proof}
Use the fixed Hodge decomposition \eqref{eq:hodge} and \eqref{eq:static-graph}.
For $g$ near $\bg$, the nonlinear weighted Bianchi identity \cite[Lemma~3.5]{PS} is
\begin{equation}
 \mathrm{div}^{g}_{f_g}\Phi(g)=0.
 \label{eq:nonlinear-bianchi}
\end{equation}
Since $\mathrm{div}_{\bfv}\Ob=0$,
\[
 P\omega_b=(\mathrm{div}_{\bfv}-\mathrm{div}^{G_b}_{f_{G_b}})\Phi(G_b).
\]
The coefficient difference is $O(|b|)$ as an operator from
$C^{1,1/2}$ to $C^{1/2}$. Elliptic inversion yields
\[
 \|\mathrm{div}^{*}\omega_b\|_{1}
 \le C|b|\bigl(|\Ob(b)|+\|\mathrm{div}^{*}\omega_b\|_{1}\bigr).
\]
Absorption proves \eqref{eq:actual-defect}. The equation for $Y(h_b)$ is an elliptic perturbation of $-2P$, with right side $-2P\omega_b$, proving \eqref{eq:gauge-static}. Finally $\mathrm{D}Y(0)=0$ and boundedness of $\mathrm{D}^2Y$ give \eqref{eq:gauge-difference} by integration along the segment $h_b+\theta q$.
\end{proof}

To describe evolving solutions, seek a graph $h=\Psi(b)$ and a reduced
field $b_s=F(b)$. The chain rule gives $h_s=\mathrm{D}\Psi(b)F(b)$, so invariance
of the graph under SMRF requires
\[
 \mathcal V(\Psi(b))=\mathrm{D}\Psi(b)F(b).
\]
For the static graph $\Psi(b)=b+W(b)$, the right-hand side contains the
transport term $\mathrm{D}W(b)F(b)$, which is absent from \eqref{eq:static-graph}.
We therefore construct polynomial corrections that satisfy the invariance
equation through each prescribed finite degree. The remaining defect
will be controlled by the complement equation below.

\subsection{The finite-order recursion}
Write
\begin{equation}
 \mathcal V(h)=Lh+\sum_{j\ge2}\mathcal V_j(h,\ldots,h),\qquad \mathcal V_j=\frac1{j!}\mathrm{D}^j\mathcal V(0).
\end{equation}
For an analytic map $A$ defined near zero in $K$, denote its homogeneous
term of degree $j$ by
\[
 [A]_j(b):=\frac1{j!}\mathrm{D}^jA(0)[b,\ldots,b].\index{A homogeneous term@$[A]_j$}
\]

\begin{theorem}[Approximate invariant graph]\label[theorem]{thm:formal}
For each integer $N\ge2$ there are unique homogeneous polynomial
maps $H_j:K\to\Comp_3$ and $F_j:K\to K$, of degree $j$,
$2\le j\le N$, such that the maps in \eqref{eq:intro-jets} have defect
\begin{equation}
 E_N(b):=\mathcal V(\Psi^{[N]}(b))-\mathrm{D}\Psi^{[N]}(b)F^{[N]}(b)
 \label{eq:defectN}\index{E defect@$E_N$}
\end{equation}
satisfying, for all sufficiently small $b\in K$ and every $v\in K$,
\[
 \|E_N(b)\|_{1}\le C_N|b|^{N+1},\qquad
 \|\mathrm{D}E_N(b)[v]\|_{1}\le C_N|b|^N|v|.
\]
The construction is $G$-equivariant.
\end{theorem}
\begin{proof}
Start with $\Psi^{[1]}(b)=b$ and $F^{[1]}=0$. Since $Lb=0$,
the initial defect is $E_1(b)=\mathcal V(b)=O(|b|^2)$ in $C^{1,1/2}$.
Suppose that the coefficients through degree $j-1$ have been chosen so
that $E_{j-1}$ has no terms of degree less than $j$. Its degree-$j$ term is
\begin{equation}
 R_j=\bigl[\mathcal V(\Psi^{[j-1]})-\mathrm{D}\Psi^{[j-1]}F^{[j-1]}\bigr]_j.
\end{equation}
This is a homogeneous polynomial with values in $\Sli\cap C^{1,1/2}$.
The elliptic inverse on $\Comp$ satisfies
\[
 \|L_{\Comp}^{-1}u\|_{3}\le C\|u\|_{1},
 \qquad u\in\Comp_{1}.
\]
Thus the corrections
\begin{equation}
 F_j=\Pi_KR_j,\qquad H_j=-L_{\Comp}^{-1}\Pi_{\Comp}R_j
 \label{eq:recursion}\index{F reduced@$F^{[N]},F_j$}\index{Psi graph@$\Psi^{[N]},H_j$ (approximate graph)}\index{L complement@$L_{\Comp}$}
\end{equation}
have the stated regularities and are homogeneous of degree $j$.
Set $\Psi^{[j]}=\Psi^{[j-1]}+H_j$ and
$F^{[j]}=F^{[j-1]}+F_j$. In $C^{1,1/2}$, expansion gives
\[
 E_j =\mathcal V(\Psi^{[j-1]}+H_j) -\mathrm{D}(\Psi^{[j-1]}+H_j)(F^{[j-1]}+F_j) =E_{j-1}+LH_j-F_j+O(|b|^{j+1}).
\]
Indeed, $\mathcal V-L$ starts quadratically and $\Psi^{[j-1]}=O(|b|)$, so
$\mathcal V(\Psi^{[j-1]}+H_j)-\mathcal V(\Psi^{[j-1]})-LH_j$ is $O(|b|^{j+1})$.
The remaining transport terms are
\[
 (\mathrm{D}\Psi^{[j-1]}-\Id)F_j,\qquad
 \mathrm{D}H_jF^{[j-1]},\qquad \mathrm{D}H_jF_j.
\]
They have degree at least $j+1$, because
$\mathrm{D}\Psi^{[j-1]}-\Id=O(|b|)$, $F^{[j-1]}=O(|b|^2)$, and
$2j-1\ge j+1$ for $j\ge2$. Consequently
\[
 [E_j]_j=R_j+LH_j-F_j=0
\]
by \eqref{eq:recursion}, and the lower-degree terms remain zero.
Induction constructs the coefficients through degree $N$.

For uniqueness, let $j$ be the first degree at which two choices differ.
Their degree-$j$ equations give
\[
 L(H_j-\widetilde H_j)=F_j-\widetilde F_j.
\]
The left-hand side lies in $\Comp$ and the right-hand side in $K$.
Both therefore vanish, and invertibility of $L_{\Comp}$ gives equality
of the corrections.

The map $E_N:K\to\Sli\cap C^{1,1/2}$ is analytic and its Taylor
coefficients through degree $N$ vanish. Taylor expansion of $E_N$ and
$\mathrm{D}E_N$ gives the two stated remainder estimates.
Finally, $\mathcal V$, the spectral projections, and $L_{\Comp}^{-1}$ commute
with the action of $G$. Starting from the equivariant maps
$\Psi^{[1]}$ and $F^{[1]}$, the recursion preserves equivariance.
\end{proof}

\subsection{Exact coordinates and the complement equation}
Fix $N$ for the remainder of the construction, and abbreviate $\Psi=\Psi^{[N]}$, $F=F^{[N]}$. For an actual solution define
\begin{equation}
 b=\Pi_Kh,\qquad q=h-\Psi(b)\in\Comp,\qquad Q_N(b)=\Id-\mathrm{D}\Psi(b)\Pi_K.
\end{equation}
Then
\begin{equation}
 b_s=\mathcal V_K(b,q):=\Pi_K\mathcal V(\Psi(b)+q),\qquad
 q_s=\mathcal V_{\Comp}(b,q):=Q_N(b)\mathcal V(\Psi(b)+q).\label{eq:exact-coordinates}\index{V coordinate fields@$\mathcal V_K,\mathcal V_{\Comp}$}\index{q complement@$q$}
\end{equation}
Recall from \eqref{eq:defectN} that
$E_N(b)=\mathcal V(\Psi(b))-\mathrm{D}\Psi(b)F(b)$.
Since $\Pi_K\mathrm{D}\Psi=\Id_K$ and $Q_N\mathrm{D}\Psi=0$,
\begin{equation}
 \mathcal V_K(b,0)=F(b)+\Pi_KE_N(b),\qquad
 \mathcal V_{\Comp}(b,0)=Q_N(b)E_N(b).\label{eq:coordinate-defect}\index{Q projection@$Q_N(b)$}
\end{equation}
Taylor's formula gives the following equations for the neutral component
and the remainder.
\begin{lemma}\label[lemma]{lem:exact-coordinates}
Let $h=\Psi(b)+q$ be an ancient SMRF converging smoothly to zero.
Then
\begin{align}
 b_s-F(b)&=\Pi_KE_N(b)
 +\Pi_K\int_0^1\mathrm{D}\mathcal V(\Psi(b)+\theta q)[q]\dd\theta,
 \label{eq:exact-center}\\
 q_s&=B_h(s)q+R_N(s),\label{eq:exact-complement}
\end{align}
where
\begin{align}
 B_h(s)v&=Q_N(b(s))\int_0^1\mathrm{D}\mathcal V(\Psi(b(s))+\theta q(s))[v]\dd\theta,
 \label{eq:exact-complement-operator}\\
 R_N(s)&=Q_N(b(s))E_N(b(s)).\label{eq:exact-complement-defect}
\end{align}
Moreover, $B_h=L_{\Comp}+o(1)$ in the operator norm from $\XP(I_s)$ to $\YP(I_s)$, and $\|R_N(s)\|_{1}\le C_N|b(s)|^{N+1}$.
\end{lemma}

\subsection{Weighted complement estimates and asymptotic expansion}
A positive weight $w$ is called exponentially tempered on a tail if
\begin{equation}
 w(u)\le C_we^{\eps|u-s|}w(s).
\end{equation}
If $|\partial_s\log v|\to0$, then every fixed power $v^j$ is tempered with arbitrarily small $\eps$ after restriction to a sufficiently negative tail.

\begin{lemma}[Weighted complement estimate]\label[lemma]{lem:slaving}
Suppose $B=L_{\Comp}+o(1)$, $w\to0$ is tempered with exponent smaller than the background spectral gap, and $\|R\|_{\YP_w}<\infty$. Every ancient solution of $q_s=Bq+R$ with uniformly bounded unit-strip solution norms has a unique decomposition
\begin{equation}
 q=\Gop_BR+z_+,\qquad \|\Gop_BR\|_{\XP_w}\le C\|R\|_{\YP_w},\label{eq:slaving}
\end{equation}
after a positive terminal normalization has been fixed. Here $z_+$ is an exponentially decaying homogeneous solution. If $w=|b|^j$ and \eqref{eq:slow-tempered} holds, then $z_+=o(|b|^A)$ for every fixed $A>0$.
\end{lemma}
\begin{proof}
By \cref{prop:weighted-green}, the normalized forced solution is the unique fixed point
\begin{equation}
 \Gop_BR=\Gop_0\bigl((B-L_\Comp)\Gop_BR+R\bigr),\qquad
 \Pi_+\Gop_BR(s_*)=0.
 \label{eq:slaving-Green}
\end{equation}
The same proposition gives its weighted solution estimate. Since $w\to0$, subtracting $\Gop_BR$ from $q$ leaves a bounded homogeneous ancient solution, which decays as $e^{\gamma s}$ for some $\gamma>0$. Finally either the polynomial lower bound in \eqref{eq:slow-tempered} or direct integration of its logarithmic estimate implies $e^{\gamma s}=o(|b|^A)$ for every fixed $A>0$.
\end{proof}

\begin{theorem}[Finite-order expansion]\label[theorem]{thm:actual-expansion}\index{Psi graph@$\Psi^{[N]},H_j$ (approximate graph)}
Every neutral-dominated ancient solution satisfies, for each integer $N\ge2$,
\begin{align}
 h(s)&=\Psi^{[N]}(b(s))+\rho_N(s)+z_{+,N}(s),\\
 \|\rho_N(s)\|_{3}&\le C_N|b(s)|^{N+1},\qquad
 \|z_{+,N}(s)\|_{3}=o(|b(s)|^A),\\
 b_s&=F^{[N]}(b)+O(|b|^{N+1}).\label{eq:actual-normal-reduction}\index{b@$b$}
\end{align}
The estimate for $z_{+,N}$ holds for every $A>0$.
The same estimates hold in every fixed higher spatial norm.
\end{theorem}
\begin{proof}
Set $w(s)=|b(s)|^{N+1}$. The weight is tempered by
\eqref{eq:slow-tempered}. The estimate for $E_N$ gives
\[
 \|R_N(s)\|_{1}\le C_N|b(s)|^{N+1}.
\]
Comparability of $w$ on unit intervals gives
$\|R_N\|_{\YP_w}<\infty$.
Apply \cref{lem:slaving} to the complement equation in
\cref{lem:exact-coordinates} to obtain
$q=\rho_N+z_{+,N}$ with the stated estimates. In particular,
\[
 \|q(s)\|_{3}\le C_N|b(s)|^{N+1}
\]
on a sufficiently negative tail.
For the neutral equation in \cref{lem:exact-coordinates},
$\Pi_KLq=0$ implies
\[
 \left|\Pi_K\int_0^1
       \bigl(\mathrm{D}\mathcal V(\Psi(b)+\theta q)-L\bigr)[q]\dd\theta\right|
 \le C\bigl(|b|+\|q\|_{3}\bigr)\|q\|_{3}
 =O(|b|^{N+2}).
\]
The defect term satisfies $|\Pi_KE_N(b)|\le C_N|b|^{N+1}$.
This proves \eqref{eq:actual-normal-reduction}.
\end{proof}

\subsection{The first coefficients}
We record the first coefficients of the approximate invariant graph. The quadratic coefficient is used in Section~\ref{sec:cp}, and the cubic comparison exhibits the transport correction to the static graph.

\begin{lemma}\label[lemma]{lem:low-order}
Write $\Phi^{(j)}=\mathrm{D}^j\Phi_{\bg}$ and $W_j=[W]_j$.
For $N\ge4$, the coefficients of \cref{thm:formal} satisfy
$\Phi^{(2)}(b,b)\in\Sli$ and
\begin{align}
 H_2(b)&=-L_{\Comp}^{-1}\Pi_{\Comp}\Phi^{(2)}(b,b),\qquad
 F_2(b)=\Pi_K\Phi^{(2)}(b,b),\label{eq:quadratic-coeff}\\
 F_3(b)&=\Pi_K\left(2\Phi^{(2)}(b,H_2(b))+\frac13\Phi^{(3)}(b,b,b)\right).
\end{align}
Moreover, $H_2=W_2$ and
\begin{equation}
 H_3-W_3=L_{\Comp}^{-1}\mathrm{D}W_2F_2.
 \label{eq:static-dynamic-difference}\index{W static graph@$W(b)$}
\end{equation}
Along the static graph, $Y(b+W(b))=Y_3(b)+O(|b|^4)$ in $C^{2,1/2}$,
where $Y_3(b):=[Y(b+W(b))]_3$.
The quartic reduced coefficient is
\begin{align}
 F_4(b)=\Pi_K\bigl(&2\Phi^{(2)}(b,H_3(b))+\Phi^{(2)}(H_2(b),H_2(b))\notag\\
 &+\Phi^{(3)}(b,b,H_2(b))+\frac1{12}\Phi^{(4)}(b,b,b,b)
 +\Lie_{Y_3(b)}b\bigr).
 \label{eq:quartic}
\end{align}
\end{lemma}
\begin{proof}
Differentiate the nonlinear Bianchi identity \eqref{eq:nonlinear-bianchi}
twice along $\bg+tb$. Since $\Phi(\bg)=0$ and $\mathrm{D}\Phi_{\bg}[b]=0$,
this gives $\mathrm{div}_{\bfv}\Phi^{(2)}(b,b)=0$.
The degree-two part of the slice equation \eqref{eq:slice-equation}
then gives $\mathrm{D}^2Y(0)[b,b]=0$. Hence $\mathcal V_2(b,b)=\Phi^{(2)}(b,b)$,
and \eqref{eq:recursion} gives \eqref{eq:quadratic-coeff}.
The degree-two part of \eqref{eq:static-graph} gives the same equation
for $W_2$, so $H_2=W_2$.

For every vector field $X$, the fixed Hodge decomposition gives
$\Pi_{\Sli}\Lie_X\bg=0$. Also,
$\Pi_K\mathrm{D}\Psi^{[N]}=\Id_K$, so projection of the invariance equation
onto $K$ gives $F_j=[\Pi_K\mathcal V(\Psi^{[N]})]_j$ for $j\le N$.
The quadratic gauge term vanishes, and every cubic gauge contribution
is a Lie derivative of $\bg$, with zero $K$-projection. Expanding $2\Phi$ to degree
three therefore gives the stated formula for $F_3$.
Projection of the degree-three equation onto $\Comp$ gives
\[
 L_{\Comp}H_3+
 \Pi_{\Comp}\left(2\Phi^{(2)}(b,H_2)+\frac13\Phi^{(3)}(b,b,b)\right)
 =\mathrm{D}H_2F_2.
\]
The static equation has the same left-hand side with $W_3$ in place
of $H_3$ and zero on the right. Subtracting proves
\eqref{eq:static-dynamic-difference}.

Since $\Ob(b)=O(|b|^2)$, \eqref{eq:gauge-static} gives
$Y(b+W(b))=O(|b|^3)$ in $C^{2,1/2}$.
Furthermore, $H_2=W_2$ implies
$\Psi^{[N]}(b)-(b+W(b))=O(|b|^3)$ in $C^{3,1/2}$.
Equation~\eqref{eq:gauge-difference} therefore gives
\[
 \|Y(\Psi^{[N]}(b))-Y(b+W(b))\|_{2}
 \le C|b|^4.
\]
Thus $Y_3$ is also the cubic gauge coefficient along the approximate invariant graph.
At degree four, the term involving the quartic gauge coefficient is a
Lie derivative of $\bg$ and has zero $K$-projection. Thus the projected
gauge contribution is $\Pi_K\Lie_{Y_3}b$.
Expanding $2\Phi$ to degree four gives the other four terms in
\eqref{eq:quartic}.
\end{proof}

\section{The characteristic exponent of a slow ancient solution}\label{sec:exponent}
We prove the existence of the characteristic exponent $\ell_*$ by deriving an analytic-gradient equation on $K$ with an error controlled by the reduced gradient and higher powers of $|b|$. For this purpose we return to the static graph: its reduced entropy supplies the analytic potential needed to determine the decay rate and limiting direction.
Recall the static graph from \eqref{eq:static-graph}:
\[
 G_b=\bg+b+W(b),\qquad
 W:K\supset U_K\longrightarrow\Comp_3,\qquad
 W(0)=\mathrm{D}W(0)=0.
\]
It satisfies $\Pi_{\Comp}\Pi_{\Sli}\Phi(G_b)=0$, and its obstruction is
$\Ob(b)=\Pi_K\Phi(G_b)$.

\subsection{Reduced entropy}
Normalize the reduced entropy by
\begin{equation}
 \Ent(b)=2\bigl(\boldsymbol{\mu}(G_b,1)-\boldsymbol{\mu}(\bg,1)\bigr),\qquad \nabla\Ent(b)=2\Pi_K\Phi(G_b) +O\bigl(|b|\,|\Pi_K\Phi(G_b)|\bigr).
 \label{eq:entropy-gradient}\index{E entropy@$\Ent$ (reduced entropy)}
\end{equation}
The gradient is defined using the fixed pairing in
\eqref{eq:weighted-pairing}: for every $\xi\in K$,
\[
 \langle\nabla\Ent(b),\xi\rangle_{L^2}
 =\mathrm{d}\Ent_b[\xi]
 :=\left.\frac{\mathrm{d}}{\mathrm{d}t}\right|_{t=0}\Ent(b+t\xi).
\]
The gradient identity in \eqref{eq:entropy-gradient} follows from the next lemma.

\begin{lemma}[Gradient--obstruction relation]\label[lemma]{lem:gradient-obstruction}\index{O obstruction@$\Ob$}
There is an analytic family $T(b)\in\operatorname{End}(K)$ with
$T(0)=2\Id$ such that
\begin{equation}
 \nabla\Ent(b)=T(b)\Ob(b),\qquad
 c|\Ob(b)|\le|\nabla\Ent(b)|\le C|\Ob(b)|.
\end{equation}
Suppose at least one homogeneous reduced coefficient $F_j$ from
\cref{thm:formal} is nonzero, and set
\[
 r=\min\{j\ge2:F_j\not\equiv0\}.\index{r order@$r$}
\]
Then the first nonconstant homogeneous term of $\Ent$ has degree $r+1$,
and $F_r=\nabla\Ent_{r+1}$.
\end{lemma}
\begin{proof}
The elliptic equation in the proof of \cref{lem:actual-defect} is
\[
 P\omega_b
 =(\mathrm{div}_{\bfv}-\mathrm{div}^{G_b}_{f_{G_b}})
       \bigl(\Ob(b)+\mathrm{div}^{*}\omega_b\bigr).
\]
Project this equation onto the fixed space $\kk^\perp$ and move the term
containing $\omega_b$ on the right to the left. The resulting operator is
an analytic small perturbation of the invertible operator $P$ on
$\kk^\perp$. Its right-hand side is linear in $\Ob(b)$, with coefficients
vanishing at $b=0$. Thus $\omega_b$ is obtained by applying an analytic
family of linear operators of size $O(|b|)$ to $\Ob(b)$. The elliptic equation
for $Y(h_b)$, where $h_b=b+W(b)$, gives the same factorization for
$Y(h_b)$.

The entropy first variation \cite[Section~3]{Perelman}, restricted to
$G_b$, gives
\[
 \mathrm{d}\Ent_b[\xi]
 =2\int_M
 \left\langle\Phi(G_b),\xi+\mathrm{D}W(b)[\xi]\right\rangle_{G_b}
 \mathrm{d}\nu_{G_b},\qquad \xi\in K,
\]
where $\mathrm{d}\nu_g=(4\pi)^{-n/2}e^{-f_g}\mathrm{d}V_g$.
Substitute $\Phi(G_b)=\Ob(b)+\mathrm{div}^{*}\omega_b$ and the
factorization of $\omega_b$. The integral is linear in $\Ob(b)$ and
in $\xi$, and all its coefficients are analytic in $b$.
Riesz representation in the fixed pairing therefore gives
$\nabla\Ent(b)=T(b)\Ob(b)$ with analytic $T$.
At $b=0$, the moving pairing equals the fixed pairing,
$\mathrm{D}W(0)=0$, and the operator giving $\omega_b$ vanishes. Hence
$T(0)=2\Id$. In particular $T(b)=2\Id+O(|b|)$ and is invertible
near zero, proving \eqref{eq:entropy-gradient} and the norm comparison.

If $\Ob\equiv0$, these factorizations give $\Phi(G_b)=0$ and
$Y(h_b)=0$. The static graph is then invariant with zero reduced field,
and uniqueness in \cref{thm:formal} gives $F_j=0$ for every $j$.
Under the hypothesis of the lemma, there is therefore a first nonzero
homogeneous obstruction term $\Ob_q$, with $q\ge2$.
By \eqref{eq:actual-defect}, \eqref{eq:gauge-static}, and the definition
$\mathcal V(h_b)=2\Phi(G_b)+\Lie_{Y(h_b)}G_b$,
\[
 \|\mathcal V(h_b)-2\Ob(b)\|_{1} \le 2\|\mathrm{div}^{*}\omega_b\|_{1} +\|\Lie_{Y(h_b)}G_b\|_{1} \le C|b|\,|\Ob(b)|.
\]
Consequently
\[
 \mathcal V(b+W(b))=2\Ob_q(b)+O_{C^{1,1/2}}(|b|^{q+1}).
\]
Since $\mathrm{D}W(b)=O(|b|)$, the transport term $\mathrm{D}W(b)[2\Ob_q(b)]$
is also $O(|b|^{q+1})$. Replacing $W$ by its Taylor polynomial
through degree $q$ changes the invariance defect only at degree $q+1$
or higher. Thus the graph
$b+\sum_{j=2}^q[W]_j(b)$, with reduced field $2\Ob_q$, satisfies
the invariance equation through degree $q$.
Uniqueness in \cref{thm:formal} gives
\[
 H_j=[W]_j\quad(2\le j\le q),\qquad
 F_j=0\quad(j<q),\qquad F_q=2\Ob_q.
\]
It follows that $q=r$. Finally
$\nabla\Ent(b)=2\Ob_q(b)+O(|b|^{q+1})$, so the first nonconstant
term of $\Ent$ has degree $q+1$ and
$\nabla\Ent_{q+1}=2\Ob_q=F_q$.
\end{proof}

\subsection{Existence of the characteristic exponent}
Let $h(s)$ be a neutral-dominated ancient solution. Write it in static
coordinates as
\begin{equation}
 h=b+W(b)+v,\qquad b=\Pi_Kh,\qquad v\in\Comp.
\end{equation}
By \cref{prop:dichotomy}, $b(s)\ne0$ on a sufficiently negative tail,
$|b_s|\le C|b|^2$, and
\[
 \|v\|_{3}
 \le\|\Pi_{\Comp}h\|_{3}+\|W(b)\|_{3}
 \le C|b|^2.
\]

\begin{lemma}\label[lemma]{lem:static-equations}
On a sufficiently negative tail,
\begin{align}
 |b_s-\nabla\Ent(b)|
 &\le C|b|\,|\nabla\Ent(b)|+C|b|\|v\|_{3},
 \label{eq:gradient-error-basic}\\
 v_s&=(L_{\Comp}+A(s))v+R(b(s)),
 \label{eq:static-v}\\
 \|R(b)\|_{1}&\le C|b|\,|\nabla\Ent(b)|.
 \label{eq:static-forcing}
\end{align}
Here, with $h_b=b+W(b)$, the analytic forcing is
\[
 R(b)=\Pi_{\Comp}\mathcal V(h_b)-\mathrm{D}W(b)\Pi_K\mathcal V(h_b),
\]
and $A$ satisfies
\[
 \|Az\|_{\YP(I_s)}\le\delta(s_*)\|z\|_{\XP(I_s)},\qquad
 s\le s_*,\qquad \delta(s_*)\longrightarrow0
 \quad(s_*\searrow-\infty).
\]
\end{lemma}
\begin{proof}
Project $h_s=\mathcal V(h)$ onto $K$ and differentiate
$v=\Pi_{\Comp}h-W(b)$ to obtain the exact identities
\[
 b_s=\Pi_K\mathcal V(h_b+v),\qquad
 v_s=(\Pi_{\Comp}-\mathrm{D}W(b)\Pi_K)\mathcal V(h_b+v).
\]
For $z\in\Comp$, define
\[
 A(s)z=(\Pi_{\Comp}-\mathrm{D}W(b)\Pi_K)
       \int_0^1\bigl(\mathrm{D}\mathcal V(h_b+\theta v)-L\bigr)[z]\dd\theta.
\]
Since $L\Comp\subset\Comp$ and $\Pi_KL|_{\Comp}=0$, the mean-value
formula gives \eqref{eq:static-v} with the stated $R$.
The bounds \eqref{eq:spatial-nonlinear}, smooth convergence of $h$,
and the evolution equation give the smallness of $A$ in the stated
norms after taking suprema on unit intervals.

The static-graph estimate in the proof of
\cref{lem:gradient-obstruction} and $\nabla\Ent=T(b)\Ob$ give
\[
 \|\Pi_{\Comp}\mathcal V(h_b)\|_{1}\le C|b|\,|\Ob(b)|,
 \qquad
 |\Pi_K\mathcal V(h_b)-\nabla\Ent(b)|\le C|b|\,|\nabla\Ent(b)|.
\]
Also $|\Pi_K\mathcal V(h_b)|\le C|\Ob(b)|$ and $\mathrm{D}W(b)=O(|b|)$.
The formula for $R$ therefore proves \eqref{eq:static-forcing}.
For the change in the neutral equation, $\mathrm{D}\mathcal V(0)=L$ and
$\Pi_KLv=0$ imply
\[
 |\Pi_K\mathcal V(h_b+v)-\Pi_K\mathcal V(h_b)| \le C\bigl(|b|+\|v\|_{3}\bigr)\|v\|_{3} \le C|b|\|v\|_{3},
\]
where the last inequality uses $\|v\|_{3}\le C|b|^2$.
Combining these estimates proves \eqref{eq:gradient-error-basic}.

\end{proof}

The proof of the following proposition is similar to the iteration of
Choi--Hung \cite[Lemma~6.9 and Proposition~6.1]{CH}.

\begin{proposition}\label[proposition]{prop:gradient-error}
For every integer $J\ge2$ there are $C_J$ and a sufficiently negative
tail such that
\begin{equation}
 |b_s-\nabla\Ent(b)|
 \le C_J\bigl(|b|\,|\nabla\Ent(b)|+|b|^J\bigr).
 \label{eq:gradient-subordinate}
\end{equation}
In particular, the estimate
\begin{equation}
 |b_s-\nabla\Ent(b)|
 \le C_\eps\bigl(|b|^{1-\eps}|\nabla\Ent(b)|+|b|^{1/\eps}\bigr)
 \label{eq:manuscript-error}
\end{equation}
holds for every sufficiently small $\eps>0$.
\end{proposition}
\begin{proof}
The induction uses a velocity bound to make the forcing weight slowly varying. The weighted complement estimate then improves the remainder, which gives the next velocity bound.

For each integer $j\ge2$, put
\[
 w_j(s)=|\nabla\Ent(b(s))|+|b(s)|^j.
\]
We prove by induction over the even integers $j\ge2$ that
\[
 |b_s(s)|\le C_jw_j(s)
\]
on a sufficiently negative tail. The case $j=2$ follows from
$|b_s|\le C|b|^2$ in \cref{prop:dichotomy}.

Assume the estimate holds for some even $j\ge2$.
Since $\|\nabla^2\Ent(b)\|\le C|b|$, the inductive velocity bound gives
\[
 \bigl|\partial_s|\nabla\Ent(b)|\bigr| \le C|b|\,|b_s|\le C_j|b|w_j,\qquad \bigl|\partial_s|b|^j\bigr| \le j|b|^{j-1}|b_s|\le C_j|b|w_j.
\]
These derivative inequalities hold almost everywhere. Because $b\ne0$,
$w_j>0$, and hence
\[
 |\partial_s\log w_j|\le C_j|b|\longrightarrow0.
\]
Together with \eqref{eq:slow-tempered}, this implies
$|\partial_s\log(|b|w_j)|\to0$.
Thus $|b|w_j$ is exponentially tempered with arbitrarily small exponent
on a sufficiently negative tail. In particular it is comparable at
any two times in a unit interval.

By \eqref{eq:static-forcing},
\[
 \|R(b(s))\|_{1}\le C|b(s)|w_j(s).
\]
The weight comparability just proved using the inductive velocity
bound therefore gives
\[
 \|R(b(\cdot))\|_{\YP(I_s)}
 \le C_j|b(s)|w_j(s).
\]
The coefficient $A$ satisfies the smallness condition in
\cref{lem:static-equations}, and $v$ has uniformly bounded unit-strip
solution norms by smooth convergence and its evolution equation.
Apply Proposition~\ref{prop:weighted-green} in \cref{app:parabolic},
as formulated in \cref{lem:slaving},
to \eqref{eq:static-v} with weight $|b|w_j$. It gives
\[
 \|v(s)\|_{3}
 \le C_j|b(s)|w_j(s)+C_j e^{\gamma(s-s_*)}
\]
for some $\gamma>0$.
By \eqref{eq:slow-tempered}, $|b(s)|\ge c|s|^{-1}$, so
\[
 |b(s)|w_j(s)\ge|b(s)|^{j+1}\ge c_j|s|^{-(j+1)}.
\]
The exponential term is therefore absorbed on a farther tail, and
\[
 \|v\|_{3}
 \le C_j|b|\,|\nabla\Ent(b)|+C_j|b|^{j+1}.
\]
Substitute this bound into \eqref{eq:gradient-error-basic} to obtain
\[
 |b_s-\nabla\Ent(b)| \le C_j\bigl(|b|\,|\nabla\Ent(b)| +|b|^2|\nabla\Ent(b)|+|b|^{j+2}\bigr) \le C'_j\bigl(|b|\,|\nabla\Ent(b)|+|b|^{j+2}\bigr).
\]
In particular $|b_s|\le C_{j+2}w_{j+2}$, closing the induction.
For any prescribed $J\ge2$, choose an even $j$ with $j+2\ge J$.
The preceding error estimate and $|b|<1$ prove
\eqref{eq:gradient-subordinate}.
Finally, for $0<\eps<1$ and $0<|b|<1$,
$|b|\le|b|^{1-\eps}$; choosing an integer $J\ge1/\eps$ in
\eqref{eq:gradient-subordinate} proves \eqref{eq:manuscript-error}.
\end{proof}

The polynomial convergence estimate of Chan--Ma--Zhang
\cite[Corollary~1.3]{CMZ}, followed by the smooth local slice construction
in Section~\ref{sec:horizontal}, gives
$|b(s)|\le C|s|^{-\alpha_0}$ for some $\alpha_0>0$.
Together with \eqref{eq:slow-tempered},
\begin{equation}
 c|s|^{-1}\le |b(s)|\le C|s|^{-\alpha_0}.
 \label{eq:coarse-rates}
\end{equation}

\begin{theorem}[Characteristic exponent and asymptotic direction]\label[theorem]{thm:ell}\index{ell characteristic@$\ell_*,\ell^*$ (characteristic exponents)}\index{r order@$r$}\index{Slow solutions@Slow solutions (neutral-dominated)}
Every nontrivial neutral-dominated ancient solution has a rational
exponent $\ell_*\ge r+1\ge3$ and $a_*\in K\setminus\{0\}$ such that
\begin{equation}
 b(s)=|s|^{-1/(\ell_*-2)}(a_*+o(1)).
 \label{eq:ell-asymptotic}
\end{equation}
Here $r=\min\{j\ge2:F_j\not\equiv0\}$. On a sufficiently negative tail,
\begin{equation}
 \Ent(b(s))\asymp |b(s)|^{\ell_*}.
 \label{eq:entropy-characteristic}\index{E entropy@$\Ent$ (reduced entropy)}
\end{equation}
\end{theorem}
\begin{proof}
First, $\Ent$ is nonconstant near zero. Otherwise
\eqref{eq:gradient-subordinate} gives $|b_s|\le C_J|b|^J$ for every $J$.
Integration yields $|b(s)|\ge c_J|s|^{-1/(J-1)}$.
Choosing $J$ with $1/(J-1)<\alpha_0$ contradicts
\eqref{eq:coarse-rates}.

Fix a tail on which $b\ne0$. For each $J$, enlarging $C_J$
over the remaining compact interval makes
\eqref{eq:gradient-subordinate} valid on this same tail: its
right-hand weight is positive there because $|b|^J>0$.
Reversing time turns the reduced equation into a descending
analytic-gradient equation. Let $\theta<1$ be a \L{}ojasiewicz gradient exponent of $\Ent$
and put $\rho_0=\max\{\theta,3/4\}<1$.
Since $|b|\le|b|^{\rho_0}$ for $|b|<1$,
\eqref{eq:gradient-subordinate} gives the error estimate at every integer
order required by Choi--Hung \cite[Theorem~7.1]{CH}.
Together with \eqref{eq:coarse-rates}, this verifies its hypotheses.
The theorem gives a rational $\ell_*\ge3$, convergence of $b/|b|$,
and a positive limit of $|s|^{1/(\ell_*-2)}|b(s)|$.
These yield \eqref{eq:ell-asymptotic}, and
\cite[Lemma~7.10]{CH} gives \eqref{eq:entropy-characteristic} on the whole tail.

The coefficient comparison in \cref{lem:gradient-obstruction} shows
that nonconstancy of $\Ent$ gives a first nonzero reduced coefficient
$F_r$ and $\Ent(b)=O(|b|^{r+1})$.
Thus \eqref{eq:entropy-characteristic} implies
\[
 c|b(s)|^{\ell_*}\le\Ent(b(s))\le C|b(s)|^{r+1}.
\]
Letting $s\searrow-\infty$ proves $\ell_*\ge r+1$.
\end{proof}

In the principal case $\ell_*=r+1$, the decay rate is $|s|^{-1/(r-1)}$.
By \cref{lem:gradient-obstruction} and
\eqref{eq:gradient-subordinate} with $J=r+1$,
\[
 b_s=F_r(b)+O(|b|^{r+1}).
\]
Set $A(s)=|s|^{1/(r-1)} b(s)$. Homogeneity gives
\[
 s A_s(s)=\frac1{r-1} A(s)-F_r(A(s))+O(|s|^{-1/(r-1)})
 \longrightarrow\frac{a_*}{r-1}-F_r(a_*).
\]
Since $A(s)\to a_*$, its derivative with respect to $\log|s|$ cannot tend
to a nonzero vector. Hence $F_r(a_*)=\frac{a_*}{r-1}$, and the reduced
equation gives
\begin{equation}
 b_s=\frac1{r-1}|s|^{-r/(r-1)}(a_*+o(1)),\qquad
 \frac{b_s}{|b_s|}\longrightarrow\frac{a_*}{|a_*|}.
 \label{eq:principal-velocity}
\end{equation}
\begin{definition}\label[definition]{def:principal}\index{a characteristic@$a_*,a^*$ (characteristic vectors)}\index{r order@$r$}\index{Principal slow solutions}
A nontrivial neutral-dominated ancient SMRF is \emph{principal} if its
characteristic exponent in \cref{thm:ell} satisfies $\ell_*=r+1$.
A nonzero vector $a_*\in K$ satisfying
\[
 F_r(a_*)=\frac{a_*}{r-1},
\]
is called a \emph{principal characteristic vector}.
For a principal ancient SMRF, its characteristic vector is
\[
 a_*=\lim_{s\searrow-\infty}|s|^{1/(r-1)}\Pi_Kh(s).
\]
By \cref{thm:actual-expansion},
$|s|^{1/(r-1)} h(s)\to a_*$ in $C^{3,1/2}$.
\end{definition}

\begin{remark}\label[remark]{rem:nonprincipal-direction}
If a slow ancient solution is nonprincipal, then $F_r(a_*)=0$.
Indeed, the reduced equation and \cref{thm:ell} give
\[
 b_s=F_r(b)+O(|b|^{r+1}),\qquad
 b(s)=|s|^{-1/(\ell_*-2)}(a_*+o(1)).
\]
Integrating over $[2s,s]$ and using homogeneity, we obtain
\[
 |s|^{r/(\ell_*-2)-1}\bigl(b(s)-b(2s)\bigr)
 =\left(\int_1^2t^{-r/(\ell_*-2)}\dd t\right)F_r(a_*)+o(1).
\]
The left side is $O(|s|^{(r-1)/(\ell_*-2)-1})$, which tends to zero
because $\ell_*>r+1$. The integral is positive, so $F_r(a_*)=0$.
\end{remark}

\subsection{The Li--Zhang obstruction bound}
Write $\Ob_j=[\Ob]_j$ for the homogeneous static obstruction terms.\index{k LZ@$k_{\mathrm{LZ}}$ (Li--Zhang order)}\index{O homogeneous terms@$\Ob_j$}
For $K\ne\{0\}$, the minimal finite Li--Zhang obstruction order
$k_{\mathrm{LZ}}$ is the least integer $k\ge2$ for which there are
$c,\delta>0$ such that
\[
 \left|\sum_{j=2}^k\Ob_j(b)\right|\ge c|b|^k,
 \qquad 0<|b|<\delta.
\]
This is the obstruction condition of Li--Zhang
\cite[Definitions~2.17 and~2.20]{LZrigid}:
by uniqueness of the static recursion, the sequence induced by $b$
is $b,2![W]_2(b),3![W]_3(b),\ldots$.
The Taylor remainder is $O(|b|^{k+1})$, so the condition is equivalent,
after adjusting $c$ and $\delta$, to
\[
 |\Ob(b)|\ge c|b|^k,\qquad 0<|b|<\delta.
\]

\begin{proposition}\label[proposition]{prop:sandwich}\index{k LZ@$k_{\mathrm{LZ}}$ (Li--Zhang order)}
If the background has finite obstruction order $k_{\mathrm{LZ}}$, then
every nontrivial slow solution satisfies
\begin{equation}
 \ell_*\le k_{\mathrm{LZ}}+1.
 \label{eq:obstruction-upper}
\end{equation}
If $k_{\mathrm{LZ}}=2$, then $r=2$, $\ell_*=3$, and
\[
 b(s)=|s|^{-1}(a_*+o(1)),\qquad F_2(a_*)=a_*.
\]
\end{proposition}
\begin{proof}
Put $k=k_{\mathrm{LZ}}$. The obstruction inequality and
\cref{lem:gradient-obstruction} give
$|\nabla\Ent(b)|\ge c|b|^k$.
Taking $J=k+1$ in \eqref{eq:gradient-subordinate} therefore yields
\[
 |b_s-\nabla\Ent(b)|
 \le C\bigl(|b|\,|\nabla\Ent(b)|+|b|^{k+1}\bigr)
 \le C'|b|\,|\nabla\Ent(b)|.
\]
On a sufficiently negative tail,
\[
 \frac{\mathrm{d}}{\mathrm{d}s}\Ent(b(s))
 =\langle\nabla\Ent(b),b_s\rangle
 \ge\frac12|\nabla\Ent(b)|^2
 \ge c|b|^{2k}.
\]
By \cref{thm:ell},
\[
 |b(s)|\asymp|s|^{-1/(\ell_*-2)},\qquad
 \Ent(b(s))\asymp|s|^{-\ell_*/(\ell_*-2)}.
\]
Integrating over $[2s,s]$ and using $\Ent(b(2s))>0$ gives
\[
 c|s|^{1-2k/(\ell_*-2)}
 \le\Ent(b(s))-\Ent(b(2s))
 \le C|s|^{-\ell_*/(\ell_*-2)}.
\]
Letting $s\searrow-\infty$ shows that
$1-2k/(\ell_*-2)\le-\ell_*/(\ell_*-2)$.
Rearranging proves
\eqref{eq:obstruction-upper}.

If $k_{\mathrm{LZ}}=2$, then \cref{thm:ell} and
\eqref{eq:obstruction-upper} give
$3\le r+1\le\ell_*\le3$.
Thus $r=2$ and $\ell_*=3$; the equation for the leading coefficient gives
$F_2(a_*)=a_*$.
\end{proof}

\section{Parametrization of principal slow solutions}\label{sec:principal}
We prove the parametrization of principal slow solutions of \cref{thm:main-intro}. A principal ancient solution determines two parameters: the exact reduced trajectory that approximates its neutral component at the prescribed rate and the terminal datum of an exponentially decaying homogeneous complementary solution. We first extract these data, then reconstruct a unique solution from every admissible pair on a common tail. Finally, we identify the actions of background isometries and the two time symmetries.

Fix a principal characteristic vector $a_*\ne0$, and set
\begin{equation}
 F_r(a_*)=\frac{a_*}{r-1},\qquad
 A_{a_*}=\mathrm{D}F_r(a_*),\qquad \rho(a_*)=\lambda_{\max}(A_{a_*}).\index{A characteristic@$A_{a_*},\rho(a_*)$}\index{A characteristic@$A_{a_*},\rho(a_*)$}
\end{equation}
Here $\lambda_{\max}(A_{a_*})$ denotes the largest eigenvalue of $A_{a_*}$.
By \cref{lem:gradient-obstruction}, $F_r=\nabla\Ent_{r+1}$, so $A_{a_*}$
is self-adjoint in the fixed $L^2$ inner product on $K$.
With $|\xi|=\|\xi\|_{L^2}$ on $K$,
\[
 \rho(a_*)=\max_{|\xi|=1}\langle A_{a_*}\xi,\xi\rangle_{L^2}.
\]
Choose an integer $N\ge r$, abbreviate $\Psi=\Psi^{[N]}$ and $F=F^{[N]}$, and set
\begin{equation}
 \kappa=\frac{N+1}{r-1}-1,\qquad \kappa>\rho(a_*).\label{eq:indicial-gap}\index{kappa accuracy@$\kappa$}
\end{equation}
Homogeneity gives $A_{a_*}a_*=\frac{r}{r-1} a_*$, so
\begin{equation}
 \rho(a_*)\ge \frac{r}{r-1},\qquad \kappa>1,\qquad N\ge2r-1.\label{eq:radial-indicial}
\end{equation}
We keep this $N$ fixed throughout the section.

\subsection{Estimates in reduced coordinates}
Recall $h=\Psi(b)+q$. The exact coordinate equations
\eqref{eq:exact-coordinates} are
\[
 \begin{aligned}
 b_s&=\mathcal V_K(b,q)=\Pi_K\mathcal V(\Psi(b)+q),\\
 q_s&=\mathcal V_{\Comp}(b,q)=Q_N(b)\mathcal V(\Psi(b)+q),
 \end{aligned}
 \qquad Q_N(b)=\Id-\mathrm{D}\Psi(b)\Pi_K.
\]
In the following lemma, variations in $b$ use the fixed norm $|\cdot|$
on $K$, and variations in $q\in\Comp$ use $\|\cdot\|_{3}$.
Outputs in $K$ use $|\cdot|$, and outputs in $\Comp$ use
$\|\cdot\|_{1}$. All operator norms are induced by these norms.

\begin{lemma}\label[lemma]{lem:derivative-ledger}
For $(b,q)$ in a sufficiently small neighborhood of $(0,0)$ in
$K\times\Comp_3$, the following estimates hold; both pairs in the
last estimate belong to this neighborhood:
\begin{align}
 \|\mathrm{D}_q\mathcal V_K(b,q)\|&\le C(|b|+\|q\|_{3}),\label{eq:ledger1}\\
 \|\mathrm{D}_b\mathcal V_K(b,q)-\mathrm{D}F(b)\|&\le C(|b|^N+\|q\|_{3}),\label{eq:ledger2}\\
 \|\mathrm{D}_b\mathcal V_{\Comp}(b,q)\|&\le C(|b|^N+\|q\|_{3}),\label{eq:ledger3}\\
 \|\mathrm{D}_q\mathcal V_{\Comp}(b,q)-\mathrm{D}_q\mathcal V_{\Comp}(\widetilde b,\widetilde q)\|
 &\le C(|b-\widetilde b|+\|q-\widetilde q\|_{3}).\label{eq:ledger4}
\end{align}
\end{lemma}
\begin{proof}
At zero $\mathrm{D}_q\mathcal V_K=\Pi_KL|_{\Comp}=0$, proving \eqref{eq:ledger1}
by the mean-value theorem. Differentiate \eqref{eq:coordinate-defect}:
$\mathrm{D}_b\mathcal V_K(b,0)=\mathrm{D}F(b)+O(|b|^N)$ and
$\mathrm{D}_b\mathcal V_{\Comp}(b,0)=O(|b|^N)$.
Moving from $q=0$ to $q$ costs $O(\|q\|_{3})$ by the
second derivative bound. Finally \eqref{eq:ledger4} is the local
Lipschitz estimate for $\mathrm{D}_q\mathcal V_{\Comp}$.
\end{proof}

\subsection{Approximation by exact reduced trajectories}\label{sec:shadow}
We replace the approximate reduced motion by an exact trajectory with the same asymptotic data. An \emph{associated exact reduced trajectory of order $N$} for a path $b$ is a trajectory
$\beta_s=F^{[N]}(\beta)$ such that
$|b(s)-\beta(s)|=O(|s|^{-\kappa})$ on a common ancient tail,
where $\kappa=\frac{N+1}{r-1}-1$.
The evolution estimate below controls differences of reduced trajectories and will give existence and uniqueness of the exact reduced trajectory.

For a linear map $B:K\to K$, write
\[
 \|B\|_{\mathrm{op}}=\sup_{|\xi|=1}|B\xi|.
\]
Given a continuous family $A(s)\in\operatorname{End}(K)$, its evolution
operator $U_A(s,u)$ maps data at time $u$ to the homogeneous solution
at time $s$. It is determined by
\[
 \partial_sU_A(s,u)=A(s)U_A(s,u),\qquad U_A(u,u)=\Id_K.\index{U evolution general@$U_A(s,u)$}
\]
For $x_s=A(s)x+v(s)$, variation of constants
\cite[Section~3.4, Theorem~3.12]{Teschl} gives
\[
 x(s)=U_A(s,u)x(u)+\int_u^sU_A(s,\sigma)v(\sigma)\dd\sigma.
\]

\begin{lemma}\label[lemma]{lem:poly-evolution}
Suppose $A:(-\infty,s_*]\to\operatorname{End}(K)$ is continuous and
$A(s)=|s|^{-1}A_{a_*}+o(|s|^{-1})$ in operator norm.
For every $\bar\rho>\rho(a_*)$, on a sufficiently negative tail,
\begin{equation}
 \|U_A(s,u)\|_{\mathrm{op}}
 \le\left(\frac{|u|}{|s|}\right)^{\bar\rho},\qquad u\le s\le s_*.
 \label{eq:poly-evolution}
\end{equation}
More precisely, for any $0<\epsilon_0\le\bar\rho-\rho(a_*)$,
the bound holds on the whole tail whenever
\[
 \sup_{s\le s_*}\||s|A(s)-A_{a_*}\|_{\mathrm{op}}\le\epsilon_0.
\]
It is uniform over all coefficient families satisfying this inequality.
If $x\in C^1((-\infty,s_*];K)$ satisfies $x_s=Ax+v$ and
$|x(s)|=O(|s|^{-\kappa})$ for $\kappa>\bar\rho$, then
\begin{equation}
 x(s)=\int_{-\infty}^sU_A(s,u)v(u)\dd u.
 \label{eq:center-infinity}
\end{equation}
If moreover $|v(s)|\le C_v|s|^{-\kappa-1}$, this integral converges
absolutely and satisfies
\begin{equation}
 \int_{-\infty}^s|U_A(s,u)v(u)|\dd u
 \le\frac{C_v}{\kappa-\bar\rho}|s|^{-\kappa}.
 \label{eq:center-absolute-bound}
\end{equation}
\end{lemma}
\begin{proof}
Given $\bar\rho>\rho(a_*)$, fix
$0<\epsilon_0\le\bar\rho-\rho(a_*)$ and choose the tail so that
\[
 \||s|A(s)-A_{a_*}\|_{\mathrm{op}}\le\epsilon_0.
\]
Self-adjointness of $A_{a_*}$ implies that every homogeneous solution
$x_s=A(s)x$ satisfies
\[
 \frac12\frac{\mathrm{d}}{\mathrm{d}s}|x(s)|^2
 =\ip{x(s)}{A(s)x(s)}
 \le\frac{\bar\rho}{|s|}|x(s)|^2.
\]
Integrating from $u$ to $s$ gives \eqref{eq:poly-evolution}.

Now let $x_s=Ax+v$ with $|x(s)|=O(|s|^{-\kappa})$ and
$\kappa>\bar\rho$. At a finite initial time $u_0<s$, variation of
constants gives
\[
 x(s)=U_A(s,u_0)x(u_0)
      +\int_{u_0}^sU_A(s,u)v(u)\dd u.
\]
The initial-data term satisfies
\[
 |U_A(s,u_0)x(u_0)| \le\left(\frac{|u_0|}{|s|}\right)^{\bar\rho}|x(u_0)| \le C|s|^{-\bar\rho}|u_0|^{\bar\rho-\kappa} \longrightarrow0\qquad(u_0\searrow-\infty).
\]
Thus the finite-time identity shows that the improper integral
converges to $x(s)$, proving \eqref{eq:center-infinity}.
For the forcings used below, $|v(u)|\le C_v|u|^{-\kappa-1}$,
and the convergence is absolute:
\[
 \int_{-\infty}^s|U_A(s,u)v(u)|\dd u \le C_v|s|^{-\bar\rho} \int_{|s|}^{\infty}t^{\bar\rho-\kappa-1}\dd t =\frac{C_v}{\kappa-\bar\rho}|s|^{-\kappa}.
\]
\end{proof}

\begin{theorem}[Exact reduced trajectory]\label[theorem]{thm:shadow}\index{F reduced@$F^{[N]},F_j$}\index{beta reduced@$\beta_N$ (exact reduced trajectory)}
Fix $\rho(a_*)<\bar\rho<\kappa$ and
$0<\epsilon_0\le\bar\rho-\rho(a_*)$. Suppose $b$ is a $C^1$ path with
\begin{equation}
 b(s)=|s|^{-1/(r-1)}(a_*+o(1)),\qquad b_s=F(b)+R_b(s),\qquad
 |R_b(s)|\le C_R|s|^{-\kappa-1}.
\end{equation}
On a sufficiently negative tail there is a unique exact trajectory
$\beta_s=F(\beta)$ with $b-\beta=O(|s|^{-\kappa})$, and it satisfies
\begin{equation}
 |b(s)-\beta(s)|
 \le\frac{C_R}{\kappa-\bar\rho}|s|^{-\kappa}.
 \label{eq:shadow-bound}
\end{equation}
It has the same leading coefficient $a_*$.

The estimate is uniform for paths satisfying the preceding
asymptotic hypothesis. There is $\epsilon>0$, depending only on
$F$, $a_*$, and $\epsilon_0$, such that for every common bound
$C_R\ge0$ there is $s_0\le-2$ with this property: if $s_*\le s_0$ and
\begin{equation}
 \sup_{s\le s_*}||s|^{1/(r-1)} b(s)-a_*|\le\frac\epsilon2,
 \qquad
 \sup_{s\le s_*}|s|^{\kappa+1}|R_b(s)|\le C_R,
 \label{eq:shadow-common-bounds}
\end{equation}
then the exact reduced trajectory exists and \eqref{eq:shadow-bound} holds on all of
$(-\infty,s_*]$. The time $s_0$ depends only on the fixed data,
$\bar\rho$, $\epsilon$, and $C_R$.
\end{theorem}
\begin{proof}
We prove the assertion under the common bounds; the asymptotic
assumption on an individual path gives the first of these bounds
after restriction to a tail.
Write $F=\sum_{j=r}^NF_j$. Homogeneity gives
\begin{equation}
 |s|\mathrm{D}F(|s|^{-1/(r-1)}z)
 =\mathrm{D}F_r(z)+\sum_{j=r+1}^N|s|^{-(j-r)/(r-1)}\mathrm{D}F_j(z).
 \label{eq:shadow-homogeneity}
\end{equation}
For $|z-a_*|\le\epsilon\le1$, the mean-value theorem therefore gives
\[
 \||s|\mathrm{D}F(|s|^{-1/(r-1)}z)-A_{a_*}\|_{\mathrm{op}}
 \le C\bigl(\epsilon+|s|^{-1/(r-1)}\bigr),
\]
where $C$ depends only on the polynomial $F$ and $a_*$.
Choose $\epsilon$ small and then $s_0$ sufficiently negative that
the right side is at most $\epsilon_0$ for $s\le s_0$.
Decrease $s_0$ further so that
\begin{equation}
 \frac{C_R}{\kappa-\bar\rho}|s_0|^{1/(r-1)-\kappa}
 \le\frac\epsilon4.
 \label{eq:shadow-tail-choice}
\end{equation}
This is possible because $\kappa>\frac1{r-1}$.

Let $s_i\searrow-\infty$, and solve $\beta_i'=F(\beta_i)$ with
$\beta_i(s_i)=b(s_i)$. On any interval where
$||s|^{1/(r-1)}\beta_i(s)-a_*|<\epsilon$, put $w_i=b-\beta_i$. Then
\[
 (w_i)_s=A_i(s)w_i+R_b(s),\qquad
 A_i(s)=\int_0^1\mathrm{D}F(\beta_i+\theta w_i)\dd\theta.
\]
Every rescaled point $|s|^{1/(r-1)}(\beta_i+\theta w_i)$ lies in
the $\epsilon$-ball about $a_*$, since it is a convex combination
of $|s|^{1/(r-1)}\beta_i$ and $|s|^{1/(r-1)} b$.
The preceding estimate bounds
$\||s|A_i(s)-A_{a_*}\|_{\mathrm{op}}$ by
$\epsilon_0$. Lemma~\ref{lem:poly-evolution} and $w_i(s_i)=0$
give
\begin{equation}
 |w_i(s)| \le C_R|s|^{-\bar\rho} \int_{s_i}^s|u|^{\bar\rho-\kappa-1}\dd u \le\frac{C_R}{\kappa-\bar\rho}|s|^{-\kappa}.
 \label{eq:shadow-integral}
\end{equation}
Consequently, for $s\le s_*\le s_0$,
\[
 ||s|^{1/(r-1)}\beta_i(s)-a_*|
 \le\frac\epsilon2+
 \frac{C_R}{\kappa-\bar\rho}|s|^{1/(r-1)-\kappa}
 \le\frac{3\epsilon}{4}.
\]
Thus $\beta_i$ cannot reach the boundary of the $\epsilon$-ball
in rescaled coordinates. Its values stay in a bounded region of
the polynomial vector field, so the ODE continuation theorem gives
existence up to $s_*$. The terminal values $\beta_i(s_*)$ have a
convergent subsequence. The same bounds and continuous dependence
give a limiting exact solution on every compact subinterval of
$(-\infty,s_*]$. It satisfies \eqref{eq:shadow-bound}, which also
implies $|s|^{1/(r-1)}\beta(s)\to a_*$.

For uniqueness, let $w=\beta_1-\beta_2$ for two exact reduced trajectories with the prescribed approximation rate.
Then $w_s=A(s)w$, where the mean-value formula and
\eqref{eq:shadow-homogeneity} give
$A(s)=|s|^{-1}A_{a_*}+o(|s|^{-1})$.
Also $|w(u)|\le C|u|^{-\kappa}$. For fixed $s$ on a sufficiently
negative tail, \cref{lem:poly-evolution} yields
\[
 |w(s)|\le\left(\frac{|u|}{|s|}\right)^{\bar\rho}|w(u)|
 \le C|s|^{-\bar\rho}|u|^{\bar\rho-\kappa}
 \longrightarrow0\qquad(u\searrow-\infty).
\]
Hence the exact reduced trajectories agree on that tail, and ODE uniqueness extends
their equality over the common interval.
\end{proof}

By \cref{thm:actual-expansion}, every principal ancient SMRF in the sense of \cref{def:principal} satisfies
the hypothesis of \cref{thm:shadow}. For such a solution $h$, denote
the unique exact reduced trajectory associated with $b=\Pi_Kh$ by $\beta_N(h)$.
The homogeneity identity \eqref{eq:shadow-homogeneity} in the proof
of \cref{thm:shadow}, together with $|s|^{1/(r-1)}\beta(s)\to a_*$, gives
\begin{equation}
 \mathrm{D}F(\beta(s))=\frac{A_{a_*}}{|s|}+o(|s|^{-1}).\label{eq:DF-shadow}
\end{equation}

\subsection{The parameter space}\label{sec:parameters}
The second parameter records the exponentially decaying complementary data that lie below every finite polynomial order. We obtain it by separating the forced and homogeneous solutions of the complement equation along an exact reduced trajectory.

Fix $\rho(a_*)<\bar\rho<\kappa$ and
$0<\epsilon_0\le\bar\rho-\rho(a_*)$ as in \cref{lem:poly-evolution}.
Choose a trajectory radius $\epsilon>0$ as in the proof of
\cref{thm:shadow} and a common terminal time $s_*\le-2$ so that
\[
 C\bigl(\epsilon+|s_*|^{-1/(r-1)}\bigr)\le\epsilon_0.
\]
Here $C$ includes the coefficient bounds from
\eqref{eq:shadow-homogeneity} and \eqref{eq:ledger2}. These choices
give the uniform polynomial evolution bounds used below.

\begin{definition}\label[definition]{def:admissible}\index{s terminal@$s_*$}
Let $\UN(a_*)$ be the set of all exact trajectories $\beta$ on
$(-\infty,s_*]$ satisfying
\begin{equation}
 \beta_s=F(\beta),\qquad |s|^{1/(r-1)}\beta(s)\longrightarrow a_*,
 \qquad \sup_{s\le s_*}||s|^{1/(r-1)}\beta(s)-a_*|<\epsilon.
 \label{eq:admissible}\index{U trajectories@$\UN(a_*)$ (admissible trajectories)}
\end{equation}
\end{definition}

We use the strip norms $\XP(I_s),\YP(I_s)$ of
\eqref{eq:append-norms}, with spatial target spaces $\Comp_3$ and
$\Comp_{1}$ as in \eqref{eq:complement-spaces}.
Along $\beta\in\UN(a_*)$ define
\begin{equation}
 B_\beta(s)=\mathrm{D}_q\mathcal V_{\Comp}(\beta(s),0)
 =Q_N(\beta(s))\mathrm{D}\mathcal V(\Psi(\beta(s)))|_{\Comp}.
 \label{eq:beta-linearization}\index{B beta@$B_\beta$}
\end{equation}
At $\beta=0$ this operator is $L_\Comp$. The bound
\eqref{eq:admissible} and the equation for $\beta$ give
\[
 |\beta(s)|\le C|s|^{-1/(r-1)},\qquad
 |\beta_s(s)|\le C|s|^{-r/(r-1)}.
\]
Analytic dependence of the coordinate field and $B_0=L_\Comp$ give
\begin{equation}
 \|(B_\beta-L_\Comp)u\|_{\YP(I_s)}
 \le C\|\beta\|_{C(I_s;K)}\|u\|_{\XP(I_s)}
 \le C|s|^{-1/(r-1)}\|u\|_{\XP(I_s)}.
 \label{eq:beta-perturbation}
\end{equation}
This verifies the smallness hypothesis of
Proposition~\ref{prop:weighted-green}, uniformly in $\beta$, after
choosing the common tail sufficiently negative.

Let $P_+^\beta(s_*)$ be the projection supplied by that proposition
and put $E_+^\beta(s_*)=\Ran P_+^\beta(s_*)$. This is the space
of terminal values of bounded homogeneous ancient solutions of
$z_s=B_\beta z$; it has dimension $\dim E_+$, and the projection
is along $E_-$. For every $p\ge0$, recall the weighted norms
\begin{equation}
 \|u\|_{\XP_p}=\sup_{s\le s_*}|s|^p\|u\|_{\XP(I_s)},\qquad
 \|R\|_{\YP_p}=\sup_{s\le s_*}|s|^p\|R\|_{\YP(I_s)}.
 \label{eq:polynomial-norms}\index{X Y polynomial@$\XP_p,\YP_p$}
\end{equation}
For each fixed $p$, Proposition~\ref{prop:weighted-green}
with weight $|s|^{-p}$ gives the unique normalized inverse
$y=\Gop_\beta R$ of $y_s=B_\beta y+R$ and
\begin{equation}
 \|\Gop_\beta R\|_{\XP_p}\le C_{G,p}\|R\|_{\YP_p},\qquad
 P_+^\beta(s_*)\Gop_\beta R(s_*)=0.
 \label{eq:complement-Green-estimate}\index{G Green@$\Gop_\beta$}\index{z homogeneous@$z_{\beta,\zeta}$}
\end{equation}
The constant is independent of $\beta$ and of the terminal time
once the common tail bounds are fixed.

For $\zeta\in E_+^\beta(s_*)$, let $z_{\beta,\zeta}$ be the unique
bounded homogeneous ancient solution with terminal value $\zeta$.
The exponential estimate in Proposition~\ref{prop:weighted-green}
gives $z_{\beta,\zeta}\in\XP_{\kappa+1}$. Set
\begin{equation}
 |\zeta|_+=\|z_{\beta,\zeta}\|_{\XP_{\kappa+1}}.
\end{equation}
This is a norm on the finite-dimensional fiber.
Define
\begin{equation}
 \PN(a_*)=\bigsqcup_{\beta\in\UN(a_*)}
 \{\zeta\in E_+^\beta(s_*):|\zeta|_+<1\}.
 \label{eq:parameter-bundle}\index{E positive beta@$E_+^\beta(s_*),P_+^\beta(s_*)$}\index{P parameter@$\PN(a_*)$ (parameter space)}\index{E positive beta@$E_+^\beta(s_*),P_+^\beta(s_*)$}
\end{equation}

For a principal ancient SMRF with characteristic vector $a_*$, put
$b=\Pi_Kh$, $q=h-\Psi(b)$, and $\beta=\beta_N(h)$. On a
sufficiently negative tail, $\beta\in\UN(a_*)$. Define
\begin{equation}
 \zeta_N(h)=P_+^\beta(s_*)q(s_*).
 \label{eq:scattering-definition}\index{E positive beta@$E_+^\beta(s_*),P_+^\beta(s_*)$}
\end{equation}
The next lemma identifies this datum with the homogeneous part of
the complement equation along the exact reduced trajectory.

\begin{lemma}\label[lemma]{lem:shadow-complement}
Let $h$ be a principal ancient SMRF with characteristic vector $a_*$,
and use $b,q,\beta$ as above. On a sufficiently negative tail,
$q\in\XP_{\kappa+1}$ and $R_h\in\YP_{\kappa+1}$, where
\begin{equation}
 R_h(s)=\mathcal V_{\Comp}(b,q)-B_\beta(s)q =Q_N(b)E_N(b) +\int_0^1\bigl[\mathrm{D}_q\mathcal V_{\Comp}(b,\theta q) -\mathrm{D}_q\mathcal V_{\Comp}(\beta,0)\bigr]q\dd\theta.
 \label{eq:shadow-complement-forcing}
\end{equation}
The complement has the decomposition
\begin{equation}
 q=\Gop_\beta R_h+z_{\beta,\zeta_N(h)},\qquad
 \|\Gop_\beta R_h\|_{\XP_{\kappa+1}}
 \le C_{G,\kappa+1}\|R_h\|_{\YP_{\kappa+1}}.
 \label{eq:shadow-complement-decomposition}\index{z homogeneous@$z_{\beta,\zeta}$}
\end{equation}
\end{lemma}
\begin{proof}
The weighted construction in \cref{thm:actual-expansion} gives
$q\in\XP_{\kappa+1}$ and
$b_s=F(b)+R_b$ with $R_b\in\YP_{\kappa+1}$, since
$|b(s)|\asymp|s|^{-1/(r-1)}$ and $\frac{N+1}{r-1}=\kappa+1$.
Theorem~\ref{thm:shadow} gives $x=b-\beta=O(|s|^{-\kappa})$.
In the equation
\[
 x_s=\left(\int_0^1\mathrm{D}F(\beta+\theta x)\dd\theta\right)x+R_b,
\]
the coefficient is $O(|s|^{-1})$. The equation therefore gives
$x_s=O(|s|^{-\kappa-1})$, and hence $b-\beta\in\XP_\kappa$.

Subtracting $B_\beta q$ from the complement equation in
\cref{lem:exact-coordinates} gives
\eqref{eq:shadow-complement-forcing}. The derivative estimates yield
\[
 \|R_h(s)\|_{1}
 \le C\left(|b|^{N+1}
       +\bigl(|b-\beta|+\|q\|_{3}\bigr)\|q\|_{3}\right).
\]
The three terms have respective orders
$|s|^{-\kappa-1}$, $|s|^{-2\kappa-1}$, and
$|s|^{-2\kappa-2}$. Taking the supremum on unit intervals
therefore gives $R_h\in\YP_{\kappa+1}$.

We have the exact equation $q_s=B_\beta q+R_h$.
Apply Proposition~\ref{prop:weighted-green} with weight
$|s|^{-\kappa-1}$ to obtain $y=\Gop_\beta R_h$ and the estimate
in \eqref{eq:shadow-complement-decomposition}.
Then $z=q-y$ is a bounded homogeneous ancient solution of
$z_s=B_\beta z$, so the same proposition identifies it with the
exponentially decaying homogeneous solution determined by $z(s_*)$.
Since $P_+^\beta(s_*)y(s_*)=0$ and $P_+^\beta(s_*)z(s_*)=z(s_*)$,
its terminal value is
$z(s_*)=P_+^\beta(s_*)q(s_*)=\zeta_N(h)$.
\end{proof}

\begin{definition}\label[definition]{def:parametrization-map}
For a smooth principal SMRF $h$ on $(-\infty,s_*]$ whose exact reduced
trajectory belongs to $\UN(a_*)$, define the \emph{parametrization map} by
\begin{equation}
 \Theta_N(h)=\bigl(\beta_N(h),\zeta_N(h)\bigr),\qquad \zeta_N(h)=P_+^{\beta_N(h)}(s_*) \bigl[h(s_*)-\Psi(\Pi_Kh(s_*))\bigr].
 \label{eq:parametrization-map}\index{Theta parametrization@$\Theta_N$}\index{zeta datum@$\zeta_N$}\index{beta reduced@$\beta_N$ (exact reduced trajectory)}
\end{equation}
Its restriction to $|\zeta_N(h)|_+<1$ takes values in $\PN(a_*)$;
the bounded domain of the bijection is specified in
\cref{thm:parametrization}.
\end{definition}

\subsection{Green estimate for the neutral equation}\label{sec:green-estimates}
We use the same strip and weighted norms for $K$-valued paths.
Since $K$ has a fixed finite basis of smooth tensors, replacing the
spatial target norms by the fixed $L^2$ norm on $K$ gives
\[
 \|x\|_{\XP(I_s)}\asymp\|x\|_{C^1(I_s;K)},\qquad
 \|v\|_{\YP(I_s)}\asymp\|v\|_{C(I_s;K)}.
\]
The constants depend only on the background.
We use the weighted norms in \eqref{eq:polynomial-norms}.
Along $\beta\in\UN(a_*)$, set
\begin{equation}
 A_\beta(s)=\mathrm{D}_b\mathcal V_K(\beta(s),0),\qquad
 (\Hop_\beta v)(s)=\int_{-\infty}^sU_{A_\beta}(s,u)v(u)\dd u.
 \label{eq:center-Green-definition}\index{A beta@$A_\beta$}\index{H Green@$\Hop_\beta$}\index{U evolution@$U_{A_\beta}$}
\end{equation}
Here $U_{A_\beta}$ is the evolution operator defined in
Section~\ref{sec:shadow}.

\begin{lemma}\label[lemma]{lem:center-green-estimate}
On a sufficiently negative common tail, for every $K$-valued
$v\in\YP_{\kappa+1}$ there is a unique $K$-valued
$x\in\XP_\kappa$ solving $x_s=A_\beta x+v$. It is
$x=\Hop_\beta v$, and
\begin{equation}
 \|\Hop_\beta v\|_{\XP_\kappa}
 \le C_H\|v\|_{\YP_{\kappa+1}}.
 \label{eq:center-Green-estimate}
\end{equation}
The constant is independent of $\beta\in\UN(a_*)$ and of the
terminal time once the common tail bounds are fixed.
\end{lemma}
\begin{proof}
Choose $\rho(a_*)<\bar\rho<\kappa$. By \eqref{eq:ledger2} and
\eqref{eq:DF-shadow},
\[
 A_\beta(s)=|s|^{-1}A_{a_*}+o(|s|^{-1}).
\]
By \eqref{eq:admissible}, \eqref{eq:shadow-homogeneity}, and
\eqref{eq:ledger2}, the choices made in Section~\ref{sec:parameters}
give
\[
 \sup_{s\le s_*}\||s|A_\beta(s)-A_{a_*}\|_{\mathrm{op}}
 \le C\bigl(\epsilon+|s_*|^{-1/(r-1)}\bigr)\le\epsilon_0.
\]
Thus \cref{lem:poly-evolution} applies uniformly in $\beta$.
For $v\in\YP_{\kappa+1}$, put
\[
 C_v=\sup_{s\le s_*}|s|^{\kappa+1}
       \|v\|_{C(I_s;K)}
 \asymp\|v\|_{\YP_{\kappa+1}}.
\]
Equation~\eqref{eq:center-absolute-bound} gives
\[
 \int_{-\infty}^s|U_{A_\beta}(s,u)v(u)|\dd u
 \le\frac{C_v}{\kappa-\bar\rho}|s|^{-\kappa}.
\]
Thus the integral defining $x=\Hop_\beta v$ converges absolutely.
Variation of constants on finite intervals gives $x_s=A_\beta x+v$,
and $|x(s)|\le CC_v|s|^{-\kappa}$.

The coefficient bound $\|A_\beta(s)\|\le C|s|^{-1}$ and
$|u|\asymp|s|$ for $u\in I_s$ give
\[
 \|x_s\|_{C(I_s;K)}
 \le \|A_\beta x\|_{C(I_s;K)}+\|v\|_{C(I_s;K)}
 \le CC_v|s|^{-\kappa-1}.
\]
Hence $\|x\|_{C^1(I_s;K)}\le CC_v|s|^{-\kappa}$,
which proves \eqref{eq:center-Green-estimate}.
The difference of two such solutions is homogeneous and
$O(|s|^{-\kappa})$; letting the initial time tend to $-\infty$
in \eqref{eq:poly-evolution} gives zero.
\end{proof}

\subsection{Existence and uniqueness for prescribed parameters}\label{sec:realization}
Fix $(\beta,\zeta)\in\PN(a_*)$. The reduced trajectory $\beta$ is
prescribed, and $\Psi(\beta)$ is its approximate realization in the
slice. The homogeneous complement solution $z=z_{\beta,\zeta}$ carries
the prescribed complementary parameter. We correct these two profiles by seeking
\begin{equation}
 b=\beta+x,\qquad q=z+y,\qquad
 h=\Psi(\beta+x)+z+y,\qquad P_+^\beta(s_*)y(s_*)=0.
\end{equation}
Thus $x$ corrects the neutral trajectory, while $y$ is the forced
complement correction. Estimates \eqref{eq:center-Green-estimate}
and \eqref{eq:complement-Green-estimate} suggest the weights $x\in\XP_\kappa$ and $y\in\XP_{\kappa+1}$.

Use the product norm
\begin{equation}
 \|(x,y)\|_* =\|x\|_{\XP_\kappa}+\|y\|_{\XP_{\kappa+1}}.
 \label{eq:correction-norm}
\end{equation}

The graph defects are
\begin{equation}
 e_K=\mathcal V_K(\beta,0)-F(\beta),\qquad e_C=\mathcal V_{\Comp}(\beta,0).
\end{equation}
By \eqref{eq:coordinate-defect} and \cref{thm:formal}, these defects are bounded by $C|\beta|^{N+1}$. Since $\frac{N+1}{r-1}=\kappa+1$, the admissibility bound gives a constant $C_e$, uniform in $\beta$ and on all sufficiently negative tails, such that
\[
 \|e_K\|_{\YP_{\kappa+1}}+\|e_C\|_{\YP_{\kappa+1}}\le C_e.
\]

Fix the defect bound $C_e$ and the uniform constants in
\eqref{eq:center-Green-estimate} and \eqref{eq:complement-Green-estimate},
and set
\begin{equation}
 C_0=(C_H+C_{G,\kappa+1})C_e,\qquad \bar R=2(C_0+1).
 \label{eq:construction-radius}
\end{equation}
These constants are fixed before the final choice of $s_*$.

\begin{theorem}[Fixed-order parametrization]\label[theorem]{thm:parametrization}\index{P parameter@$\PN(a_*)$ (parameter space)}\index{s terminal@$s_*$}
Fix $\bar R$ as in \eqref{eq:construction-radius}, and then choose the
common terminal time $s_*$ sufficiently negative. The map $\Theta_N$
is a bijection onto $\PN(a_*)$ from the set of smooth SMRFs $h$ on
$(-\infty,s_*]$ with the following property: writing
$b=\Pi_Kh$ and $q=h-\Psi(b)$, there is an exact trajectory
$\beta\in\UN(a_*)$ such that
\begin{equation}
 \zeta=P_+^\beta(s_*)q(s_*),\qquad |\zeta|_+<1,\qquad \|b-\beta\|_{\XP_\kappa} +\|q-z_{\beta,\zeta}\|_{\XP_{\kappa+1}}\le \bar R.
 \label{eq:bounded-solution-class}
\end{equation}
In this class $\beta=\beta_N(h)$, and the solution associated with
$(\beta,\zeta)$ is denoted by $h_{\beta,\zeta}$.
\end{theorem}
\begin{proof}
\emph{Construction.}
Fix $(\beta,\zeta)\in\PN(a_*)$ and write $z=z_{\beta,\zeta}$. We derive the correction equations and solve them by contraction.

The diagonal linearizations are $A_\beta,B_\beta$ from
\eqref{eq:center-Green-definition} and Section~\ref{sec:parameters}.
The two linear couplings are
\begin{equation}
 C_\beta=\mathrm{D}_q\mathcal V_K(\beta,0),\qquad D_\beta=\mathrm{D}_b\mathcal V_{\Comp}(\beta,0).
\end{equation}
Here and below, all profiles and coefficients are evaluated at time $s$.
Equation~\eqref{eq:coordinate-defect} identifies the errors as
$e_K=\Pi_KE_N(\beta)$ and $e_C=Q_N(\beta)E_N(\beta)$.
Theorem~\ref{thm:formal} bounds the defect by $C|\beta|^{N+1}$,
and \eqref{eq:admissible} gives $|\beta|\le C|s|^{-1/(r-1)}$.
Since $\frac{N+1}{r-1}=\kappa+1$, these estimates give the
first bound below. The other two follow from
\eqref{eq:ledger1} and \eqref{eq:ledger3}, evaluated at $(\beta,0)$:
\begin{equation}
 |e_K|+\|e_C\|_{1}\le C|s|^{-\kappa-1},\qquad
 \|C_\beta\|\le C|s|^{-1/(r-1)},\qquad
 \|D_\beta\|\le C|s|^{-N/(r-1)}.
 \label{eq:profile-errors}
\end{equation}
The operator norms use the spatial input and output norms specified
before \cref{lem:derivative-ledger}. Taking suprema gives the strip estimates and the uniform defect bound used above.

The remaining terms are defined by subtracting these constant and
linear parts:
\[
 \begin{aligned}
 R_K(s,x,q)&=\mathcal V_K(\beta+x,q)-\mathcal V_K(\beta,0)-A_\beta x-C_\beta q,\\
 R_C(s,x,q)&=\mathcal V_{\Comp}(\beta+x,q)-\mathcal V_{\Comp}(\beta,0)-D_\beta x-B_\beta q.
 \end{aligned}
\]
Taylor's formula in the variables $(b,q)$ gives
\[
 R_K(s,x,q)=\int_0^1(1-\theta)
 \mathrm{D}^2\mathcal V_K(\beta+\theta x,\theta q)[(x,q),(x,q)]\dd\theta,
\]
with the same formula for $R_C$ using $\mathcal V_{\Comp}$.
The second derivatives are uniformly bounded from the product of the
spatial input spaces to the corresponding output space on a fixed
small neighborhood. Hence
\begin{equation}
 |R_K(s,x,q)|+\|R_C(s,x,q)\|_{1}
 \le C\bigl(|x|+\|q\|_{3}\bigr)^2.
 \label{eq:quadratic-remainders}
\end{equation}
For two inputs, the mean-value theorem applied to the first derivative
of the remainder gives
\[
 \begin{aligned}
 &|R_K(s,x_1,q_1)-R_K(s,x_2,q_2)|
  +\|R_C(s,x_1,q_1)-R_C(s,x_2,q_2)\|_{1}\\
 &\quad\le C\bigl(|x_1|+|x_2|+\|q_1\|_{3}+\|q_2\|_{3}\bigr)
       \bigl(|x_1-x_2|+\|q_1-q_2\|_{3}\bigr).
 \end{aligned}
\]
Indeed, the derivative of each remainder is the difference between the
coordinate-field derivative at $(\beta+x,q)$ and at $(\beta,0)$.
Taking suprema on unit intervals gives these estimates with
solution norms on the inputs and forcing norms on the outputs.

Subtracting $\beta_s=F(\beta)$ and $z_s=B_\beta z$ from the exact
coordinate equations gives
\[
 x_s=A_\beta x+e_K+C_\beta(z+y)+R_K(s,x,z+y),\qquad y_s=B_\beta y+e_C+D_\beta x+R_C(s,x,z+y).
\]
By \cref{lem:center-green-estimate} and the complement construction
in Section~\ref{sec:parameters}, these equations, with the prescribed
weights and terminal condition, are equivalent to
\begin{align}
 x&=\Hop_\beta\bigl[e_K+C_\beta(z+y)+R_K(s,x,z+y)\bigr],\label{eq:LP-center}\\
 y&=\Gop_\beta\bigl[e_C+D_\beta x+R_C(s,x,z+y)\bigr].\label{eq:LP-complement}
\end{align}
We solve these equations in the product space with norm
\eqref{eq:correction-norm}.

The weighted remainder estimate is
\begin{equation}
 \|R_K(\cdot,x,q)\|_{\YP_{\kappa+1}} +\|R_C(\cdot,x,q)\|_{\YP_{\kappa+1}} \le C\bigl(|s_*|^{-(\kappa-1)}\|x\|_{\XP_\kappa}^2 +|s_*|^{-\kappa}\|x\|_{\XP_\kappa}\|q\|_{\XP_{\kappa+1}} +|s_*|^{-\kappa-1}\|q\|_{\XP_{\kappa+1}}^2\bigr).
 \label{eq:weighted-remainders}
\end{equation}
For example, the term quadratic in $x$ has strip norm bounded by
$C|s|^{-2\kappa}\|x\|_{\XP_\kappa}^2$; multiplying by the forcing
weight $|s|^{\kappa+1}$ gives $|s|^{-(\kappa-1)}$. The mixed and
quadratic $q$ terms give the other two powers in the same way.
The linear couplings satisfy
\begin{equation}
 \|C_\beta q\|_{\YP_{\kappa+1}} \le C|s_*|^{-1/(r-1)}\|q\|_{\XP_{\kappa+1}},\qquad \|D_\beta x\|_{\YP_{\kappa+1}} \le C|s_*|^{-(\kappa-1/(r-1))}\|x\|_{\XP_\kappa}.
 \label{eq:weighted-couplings}
\end{equation}
For the second bound, $\frac{N}{r-1}-1=\kappa-1/(r-1)$.
The difference estimates have the same powers, by the remainder
Lipschitz estimate above. All these exponents are positive by
\eqref{eq:radial-indicial}.

Let $\mathcal T$ be the right side of
\eqref{eq:LP-center}--\eqref{eq:LP-complement}. We apply the
contraction theorem on the closed radius-$\bar R$ ball in the product
space \eqref{eq:correction-norm}. For every pair in this ball,
$\|z+y\|_{\XP_{\kappa+1}}\le1+\bar R$, since $|\zeta|_+<1$.
On each unit strip,
\[
 \|\beta+x\|_{\XP(I_s)}+\|z+y\|_{\XP(I_s)}
 \le C\bigl(|s|^{-1/(r-1)}+\bar R |s|^{-\kappa}+(1+\bar R)|s|^{-\kappa-1}\bigr).
\]
Thus a sufficiently large $|s_*|$ places every trial pair in the
coordinate neighborhood where the preceding estimates hold. It also
makes $\bg+\Psi(\beta+x)+z+y$ positive definite.

The contribution of the two defects to $\mathcal T$ is bounded by
$C_0$. Applying \eqref{eq:center-Green-estimate} and
\eqref{eq:complement-Green-estimate} to
\eqref{eq:weighted-remainders} and \eqref{eq:weighted-couplings}
gives
\begin{align}
 \|\mathcal T(x,y)\|_*
 \le C_0+C\bigl(&|s_*|^{-1/(r-1)}(\bar R+1)+|s_*|^{-(\kappa-1/(r-1))}\bar R
 +|s_*|^{-(\kappa-1)}\bar R^2\notag\\*
 &+|s_*|^{-\kappa}\bar R(\bar R+1)+|s_*|^{-\kappa-1}(\bar R+1)^2\bigr).
 \label{eq:self-map-estimate}
\end{align}
The terms in parentheses tend to zero as $|s_*|\to\infty$.
Choose $|s_*|$ so that their contribution, including $C$, is at most
$1$. Then $\|\mathcal T(x,y)\|_*\le C_0+1=\bar R/2$, so the ball
is invariant.

For two pairs in the ball, $z$ is fixed and therefore
$q_1-q_2=y_1-y_2$. The coupling and remainder difference estimates give
\begin{equation}
 \begin{aligned}
 &\|\mathcal T(x_1,y_1)-\mathcal T(x_2,y_2)\|_*\\
 &\quad\le C\bigl(|s_*|^{-1/(r-1)}+|s_*|^{-(\kappa-1/(r-1))}
 +(\bar R+1)(|s_*|^{-(\kappa-1)}+|s_*|^{-\kappa}+|s_*|^{-\kappa-1})\bigr)
 \|(x_1-x_2,y_1-y_2)\|_*.
 \end{aligned}
 \label{eq:contraction-estimate}
\end{equation}
Increase $|s_*|$ so that the coefficient is at most $1/2$.
The product space is Banach, and the Green normalization ensures
$P_+^\beta(s_*)y(s_*)=0$ for the image of $\mathcal T$.
Banach's theorem gives a unique fixed point in the ball.

\emph{Identification of the parameters.}
Put $b=\beta+x$, $q=z+y$, and $h=\Psi(b)+q$. The integral equations
imply $b_s=\mathcal V_K(b,q)$ and $q_s=\mathcal V_{\Comp}(b,q)$. Since
$\Pi_K\Psi(b)=b$ and $q\in\Comp$, these are the intrinsic
coordinates of $h$, and
\[
 h_s=\mathrm{D}\Psi(b)\Pi_K\mathcal V(h)+Q_N(b)\mathcal V(h)=\mathcal V(h).
\]
The fixed point initially belongs to the solution space $\XP$.
Its $\|\cdot\|_3$ norm is $O(|s|^{-1/(r-1)})$ by the construction.
\Cref{lem:interior-smoothing} gives uniform higher-order bounds and
therefore smooth convergence to zero. Since
$b=\beta+x\sim a_*|s|^{-1/(r-1)}$, \cref{prop:dichotomy} applies
in its neutral case and gives
$\|h(s)\|_j=O(|s|^{-1/(r-1)})$ at every fixed spatial order.

Finally, $x=O(|s|^{-\kappa})$ and $\kappa>\frac1{r-1}$ show that $b$ has
leading coefficient $a_*$ and that the prescribed $\beta$ is its unique
exact reduced trajectory, by \cref{thm:shadow}. The terminal normalization gives
\[
 P_+^\beta(s_*)q(s_*)=P_+^\beta(s_*)z(s_*)=\zeta.
\]
These identify both components of $\Theta_N$ and prove surjectivity.

\emph{Injectivity.}
Let $h_1,h_2$ satisfy
\eqref{eq:bounded-solution-class} with the same $(\beta,\zeta)$.
The weighted neutral bound identifies $\beta$ as their common exact reduced trajectory.
Set $z=z_{\beta,\zeta}$ and
$x_i=\Pi_Kh_i-\beta$, $y_i=h_i-\Psi(\Pi_Kh_i)-z$.
Both pairs lie in the same radius-$\bar R$ ball and satisfy the exact
correction equations. Their neutral decay eliminates the homogeneous
neutral term by \cref{lem:center-green-estimate}, and
$P_+^\beta(s_*)y_i(s_*)=0$ gives the complement Green normalization.
Thus both are fixed points of the same $\mathcal T$. By
\eqref{eq:contraction-estimate},
\[
 \|(x_1-x_2,y_1-y_2)\|_*
 \le\frac12\|(x_1-x_2,y_1-y_2)\|_*.
\]
Hence $h_1=h_2$; uniqueness of continuation from an ancient interval in
Section~\ref{sec:horizontal} gives equality on their common continuation.
\end{proof}

\begin{remark}\label[remark]{rem:principal-germs}\index{Ancient germ}
The construction in \cref{thm:parametrization} works with any fixed
finite fiber radius and any larger fixed correction bound, after
choosing the common tail.
Every principal ancient germ is covered by such a bounded
parametrization: \cref{lem:shadow-complement} gives finite weighted
bounds for $b-\beta$, $q$, and $R_h$. Its Green estimate bounds
$y=\Gop_\beta R_h$, and hence $z=q-y$, uniformly when the terminal
time is moved backward. Choose finite bounds containing these
corrections and homogeneous data; the same contraction then identifies
the germ with its realization.
\end{remark}

\subsection{Equivariance and time symmetries}\label{sec:phase}
\paragraph*{Equivariance under $G$.}
Recall $G=\Isom(M,\bg,\bfv)=\Isom(M,\bg)$, acting on tensors by
pullback. For $\gamma\in G$, the action $\gamma^*$ commutes with $L$,
all fixed projections, and the coefficient recursion. Thus
\begin{equation}
 \Psi(\gamma^*b)=\gamma^*\Psi(b),\qquad F(\gamma^*b)=\gamma^*F(b).
\end{equation}
Uniqueness of the exact reduced trajectory gives
$\beta_N(\gamma^*h)=\gamma^*\beta_N(h)$. The complement evolutions
and their equivariant normalization satisfy
\begin{equation}
 P_+^{\gamma^*\beta}(s)=\gamma^*P_+^\beta(s)(\gamma^*)^{-1},
 \qquad \zeta_N(\gamma^*h)=\gamma^*\zeta_N(h).
\end{equation}
The action sends the chart over $a_*$ to the chart over $\gamma^*a_*$;
in one marked chart the residual group is the stabilizer $G_{a_*}$.
Fixed-point uniqueness makes the inverse parametrization equivariant.
Together with \cref{thm:correspondence}, this proves the spatial
quotient assertion in \cref{thm:main-intro}.

\paragraph*{SMRF-time translation and Ricci-flow rescaling.}
We identify parabolic rescaling with translation of the SMRF time. For $\lambda>0$, set
$g_{\mathrm{RF}}^\lambda(t)=\lambda g_{\mathrm{RF}}(t/\lambda)$ and
$c=\log\lambda$. The normalization \eqref{eq:time-general} gives
\[
 \hat g^\lambda(s)
 =e^s\lambda g_{\mathrm{RF}}(-e^{-s}/\lambda)
 =e^{s+c}g_{\mathrm{RF}}(-e^{-(s+c)})
 =\hat g(s+c).
\]
With compatible horizontal markings, this induces $h_c(s)=h(s+c)$.
Fix a reference solution $h=h_{\beta,\zeta}$ from
\cref{thm:parametrization}. We describe the local action at its
fixed terminal time $s_*$. The smooth solutions $h$ and $\beta$
extend slightly beyond $s_*$ by their equations. For sufficiently
small $|c|$, set
\[
 h_c(s)=h(s+c),\qquad \beta_c(s)=\beta(s+c),\qquad
 q_c(s)=q(s+c),\qquad s\le s_*.
\]
The equations are autonomous, and
\[
 |s|^{1/(r-1)}\beta_c(s)\longrightarrow a_*,\qquad
 \Pi_Kh_c(s)-\beta_c(s)=O(|s+c|^{-\kappa})=O(|s|^{-\kappa}).
\]
Thus \cref{thm:shadow} gives $\beta_N(h_c)=\beta_c$.
The strict admissibility and fiber bounds, together with the
correction bound $\|(x,y)\|_*\le\bar R/2$ in the proof of
\cref{thm:parametrization}, keep the shifted solution in the same
chart for small $|c|$. Its parameters at the fixed endpoint are
\begin{equation}
 \Theta_N(h_c)
 =\left(\beta_c,\ P_+^{\beta_c}(s_*)q(s_*+c)\right).
 \label{eq:modified-time-data}
\end{equation}
In the decomposition $q=z_{\beta,\zeta}+y$ of the reconstructed
solution, continue $z_{\beta,\zeta}$ by its linear equation when
needed and set $y=q-z_{\beta,\zeta}$. The new complementary parameter is
\[
 z_{\beta,\zeta}(s_*+c)
 +P_+^{\beta_c}(s_*)y(s_*+c),
\]
because the shifted homogeneous solution lies in the new positive
space. Thus the forced complement also contributes to the terminal
coordinate.

\paragraph*{Ricci-flow time translation.}
Original Ricci-time translation sends $g_{\mathrm{RF}}(t)$ to
$g_{\mathrm{RF}}(t+\theta)$. We use the same logarithmic normalization
for both flows and fix their relative horizontal marking as follows.
Use the unshifted reconstruction pullback, take the local slice
representative of the translated family, and remove its Killing
velocity by a $G$-valued correction tending to the identity at
$s=-\infty$. This convention fixes the time-independent
$G$-ambiguity throughout the comparison.

\begin{lemma}\label[lemma]{lem:phase-estimate}
Let $h$ be a smooth SMRF on $(-\infty,s_*]$ satisfying
\[
 \lim_{s\searrow-\infty}\||s|^{1/(r-1)} h(s)-a_*\|_3=0,
 \qquad a_*\in K\setminus\{0\},\qquad
 F_r(a_*)=\frac{a_*}{r-1}.
\]
Let $h^{(\theta)}$ represent $g_{\mathrm{RF}}(t+\theta)$ with the
normalization and horizontal marking just specified. On a common
sufficiently negative tail, locally in $\theta$,
\begin{align}
 \|h^{(\theta)}(s)-h(s)\|_3
 &\le C|\theta|e^s,\label{eq:phase-flat}\\
 J_h(s):=\left.\partial_\theta h^{(\theta)}(s)\right|_{\theta=0}
 &=-2e^s\Ric_{\bg}+O(e^s|s|^{-1/(r-1)}).\label{eq:phase-J}
\end{align}
The remainder is measured in $\|\cdot\|_3$.
\end{lemma}
\begin{proof}
Write $g(s)=\bg+h(s)$. The uniform bounds in
\cref{lem:interior-smoothing} first give smooth convergence to zero.
Since $\Pi_Kh(s)\sim a_*|s|^{-1/(r-1)}$, the neutral case of
\cref{prop:dichotomy} then gives
$\|h(s)\|_j=O(|s|^{-1/(r-1)})$ at every fixed higher spatial order.
In particular,
$\|\Ric_{g(s)}-\Ric_{\bg}\|_j=O(|s|^{-1/(r-1)})$.

Let $\phi_s$ be the reconstruction pullback, so that
$g(s)=e^s\phi_s^*g_{\mathrm{RF}}(-e^{-s})$. Hold $\phi_s$ fixed
while varying $\theta$, and put
\[
 k_\theta(s)=e^s\phi_s^*g_{\mathrm{RF}}(-e^{-s}+\theta),
 \qquad k_0(s)=g(s).
\]
Naturality and scale invariance of the Ricci tensor give
\begin{equation}
 \partial_\theta k_\theta(s)=-2e^s\Ric_{k_\theta(s)}.
 \label{eq:phase-exact}
\end{equation}
Consequently $\|k_\theta-g\|_j\le C_j|\theta|e^s$ for every
fixed $j$, locally uniformly in $\theta$ on a common tail.
The pulled-back velocity of $\phi_s$ is
\[
 W(s)=Y(h(s))-\nabla^{g(s)}f_{g(s)}.
\]
Differentiating in $s$ now gives
\begin{equation}
 (k_\theta)_s=k_\theta-2\Ric_{k_\theta}+\Lie_Wk_\theta
 =2\Phi(k_\theta)
  +\Lie_{W+\nabla^{k_\theta}f_{k_\theta}}k_\theta.
 \label{eq:phase-raw-evolution}
\end{equation}
It also gives
$\|(k_\theta)_s-g_s\|_j\le C_j|\theta|e^s$.

Apply the local slice construction of Section~\ref{sec:horizontal}:
let $\chi_\theta=\vartheta_{X_\theta}$, with $X_\theta\perp\kk$,
and set
\[
 \bg+\widetilde h_\theta=\chi_\theta^*k_\theta\in\bg+\Sli,
 \qquad X_0=0,\qquad\widetilde h_0=h.
\]
The inverse in the implicit equation \eqref{eq:static-slice-gauge}
is uniformly bounded near zero. Subtracting that equation at
$\theta$ and at zero gives $\|X_\theta\|_{j+1}\le C_j|\theta|e^s$.
Differentiating it in $s$, subtracting again, and using
\eqref{eq:phase-raw-evolution} gives
$\|\partial_sX_\theta\|_{j+1}\le C_j|\theta|e^s$.
These estimates use higher spatial orders of $k_\theta$ to account
for differentiation of the pullback.
The derivative of the slice map at $\bg$ is $\Pi_\Sli$, and
$\Pi_\Sli\Ric_{\bg}=\Ric_{\bg}$. Its derivative at $g$ differs
from this one by $O(|s|^{-1/(r-1)})$ in the required spatial norms.
Thus \eqref{eq:phase-exact} yields
\begin{equation}
 \left.\partial_\theta\widetilde h_\theta\right|_{\theta=0}
 =-2e^s\Ric_{\bg}+O(e^s|s|^{-1/(r-1)}),\qquad
 \|\widetilde h_\theta-h\|_j\le C_j|\theta|e^s.
 \label{eq:phase-slice-estimate}
\end{equation}

It remains to impose the horizontal velocity. Pulling back
\eqref{eq:phase-raw-evolution} gives
\[
 (\widetilde h_\theta)_s =2\Phi(\bg+\widetilde h_\theta) +\Lie_{Z_\theta}(\bg+\widetilde h_\theta),\qquad Z_\theta =\chi_\theta^*(W+\nabla^{k_\theta}f_{k_\theta}) +(\mathrm{d}\chi_\theta)^{-1}\partial_s\chi_\theta.
\]
Both $(\widetilde h_\theta)_s$ and $\mathcal V(\widetilde h_\theta)$ lie
in $\Sli$. Hence \eqref{eq:kernel-Ah} implies
\[
 K_\theta:=Z_\theta-Y(\widetilde h_\theta)\in\kk.
\]
At $\theta=0$, $Z_0=Y(h)$ and $K_0=0$.
The preceding estimates, together with smooth dependence of $f_g$
and $Y(h)$, give
\[
 |K_\theta(s)|\le C|\theta|e^s,\qquad
 \left|\left.\partial_\theta K_\theta(s)\right|_{\theta=0}\right|
 \le Ce^s,
\]
where $|\cdot|$ is any fixed norm on the finite-dimensional space
$\kk$. The first bound is integrable at backward infinity.
Solve
\[
 \partial_s a_\theta=-K_\theta\circ a_\theta,
 \qquad a_\theta(-\infty)=\Id,
\]
in $G$. Integration from $-\infty$ gives the unique relative
correction, with $a_\theta=\Id+O(|\theta|e^s)$ and
$\left.\partial_\theta a_\theta\right|_0=O(e^s)$ in local
coordinates on $G$.
Then $h^{(\theta)}=a_\theta^*\widetilde h_\theta$ satisfies SMRF.
Since $a_\theta$ preserves $\bg$, its infinitesimal contribution
at $\theta=0$ is a Lie derivative of $h$, bounded by
$Ce^s\|h\|_{4}=O(e^s|s|^{-1/(r-1)})$.
Equation~\eqref{eq:phase-slice-estimate} therefore proves
\eqref{eq:phase-J}. The finite-difference bounds for the slice and
for $a_\theta$ give \eqref{eq:phase-flat}.
\end{proof}

\begin{proposition}\label[proposition]{prop:phase}\index{Ricci-flow time translation}
The Ricci-flow time translation
$g_{\mathrm{RF}}(t)\mapsto g_{\mathrm{RF}}(t+\theta)$ leaves the
associated exact reduced trajectory of order $N$ unchanged:
$\beta_N(h^{(\theta)})=\beta_N(h)$.
If $E_+=\R\Ric_{\bg}$, it gives a local coordinate on each
one-dimensional fiber of complementary parameters. Thus nearby solutions with the
same exact reduced trajectory differ by a unique small Ricci-flow time translation,
and quotienting by this action removes the local scalar parameter.
Near a solution with zero complementary parameter, each such orbit has a
unique local representative with $\zeta_N=0$.
\end{proposition}
\begin{proof}
By \eqref{eq:phase-flat}, the neutral change is
$O(e^s)=o(|s|^{-\kappa})$. The same exact trajectory $\beta$
therefore approximates both neutral curves at the prescribed rate, and \cref{thm:shadow} gives
$\beta_N(h^{(\theta)})=\beta_N(h)$.

Suppose $E_+=\R\Ric_{\bg}$. Since $L\Ric_{\bg}=\Ric_{\bg}$
and $L$ is self-adjoint, $\Pi_K\Ric_{\bg}=0$.
For $b^{(\theta)}=\Pi_Kh^{(\theta)}$ and
$q^{(\theta)}=h^{(\theta)}-\Psi(b^{(\theta)})$, differentiation
and \eqref{eq:phase-J} give
\[
 \left.\partial_\theta q^{(\theta)}(s)\right|_0
 =Q_N(b(s))J_h(s)
 =-2e^s\Ric_{\bg}+O(e^s|s|^{-1/(r-1)}).
\]
The exact reduced trajectory $\beta$ is fixed, so $P_+^\beta(s_*)$ is fixed in this
differentiation. The graph construction in
Proposition~\ref{prop:weighted-green} and
\eqref{eq:beta-perturbation} give
$P_+^\beta(s_*)=\Pi_++O(|s_*|^{-1/(r-1)})$ on the fixed
finite-dimensional positive space. In particular,
\[
 r_\beta=e^{s_*}P_+^\beta(s_*)\Ric_{\bg}\ne0,
 \qquad E_+^\beta(s_*)=\R r_\beta.
\]
Write $\zeta_N(h)=\sigma(h)r_\beta$. Projecting the preceding
derivative at $s_*$ yields
\[
 \left.\frac{\mathrm{d}}{\mathrm{d}\theta}
 \sigma(h^{(\theta)})\right|_{\theta=0}
 =-2+O(|s_*|^{-1/(r-1)})\ne0
\]
on a sufficiently negative tail. Indeed,
$e^{-s_*}r_\beta\to\Ric_{\bg}$, so division by this basis vector
preserves the stated error bound.

Fix a reference solution $h_0$ in a bounded chart of
\cref{thm:parametrization}, using \cref{rem:principal-germs}
to enlarge its fixed bounds if needed.
The translated solutions stay in the chart for small $\theta$ by
the smooth estimates above. The inverse-function theorem makes
$\theta\mapsto\sigma(h_0^{(\theta)})$ a bijection between
neighborhoods of $0$ and $\sigma(h_0)$. Every parameter in this
fiber neighborhood is realized by \cref{thm:parametrization}.
For its realizing solution, choose the unique small $\theta$ with
the same scalar datum; the exact reduced trajectories already agree, so injectivity
in that theorem identifies the solution with $h_0^{(\theta)}$.
This proves the local orbit assertion. Taking $\sigma(h_0)=0$
gives the stated unique local normalization $\zeta_N=0$.

These local actions agree in overlapping local parametrizations because the
relative marking is fixed at backward infinity. Each orbit is open
in a connected local one-dimensional fiber, so such a fiber is one orbit.
When $\dim E_+>1$, the same nonzero derivative removes one
local direction and leaves the other complementary parameters.
\end{proof}

\section{Complex projective tangent shrinkers}\label{sec:cp}
In this section $(M,\bg)=(\CP^{d-1},g_{\mathrm{FS}})$, with $d\ge3$ and $\Ric_{\bg}=\frac12\bg$. The potential $\bfv$ is constant. We first compute the common spectral data and quadratic reduced field. For $d=2n+1$, we determine the characteristic matrices and stable fibers, whose geometric quotients give the Weyl-chamber models in \cref{thm:cp-moduli} and the two classes on $\CP^2$ in \cref{cor:cp2}. The final subsection constructs a nonprincipal slow solution in odd complex dimension.

\subsection{Neutral tensors and the quadratic reduced field}
The Fubini--Study metric is linearly stable and has nontrivial infinitesimal
solitonic deformations; see \cite[Definition 2.9 and the following discussion]{LZrigid}
and \cite[Lemma 2.6]{LZcp}. Its nonlinear dynamical instability was established by
Kr\"oncke \cite{KronckeEinstein}, with an independent entropy proof by
Knopf--\v{S}e\v{s}um \cite{KnopfSesum}. Its nonlinear rigidity was proved by
Kr\"oncke in even complex dimension \cite[Theorem 1.3]{Kroncke}
and by Li--Zhang in all complex dimensions \cite{LZcp}.
The following lemma collects the spectral facts. The description of the
first eigenfunctions is recalled in \cite[proof of Theorem 6.1]{Kroncke}.

\begin{lemma}\label[lemma]{lem:cp-spectrum}
For $(\CP^{d-1},g_{\mathrm{FS}})$, $d\ge3$, normalized by $\Ric_{\bg}=\bg/2$,
\begin{equation}
 E_+=\R\Ric_{\bg},\qquad L|_{E_+}=\Id,\qquad
 K=\{h_u=u\bg+2\bar\nabla^2u:\bar\Delta u+u=0\}.
 \label{eq:cp-spectrum-input}
\end{equation}
The functions satisfying $\bar\Delta u+u=0$ are exactly
\[
 u_A([z])=\frac{z^*Az}{|z|^2},\qquad A\in\Herm_0(d).\index{u projective@$u_A$}\index{Herm traceless@$\Herm_0(d)$}
\]
Their homogeneous extensions $|z|^2u_A([z])=z^*Az$ to $\C^d$ are
real-valued harmonic polynomials of bidegree $(1,1)$, since
$\Delta_{\C^d}(z^*Az)=4\tr A=0$.
\end{lemma}

The map
\begin{equation}
 A\longmapsto h_A:=u_A\bg+2\bar\nabla^2u_A
 \label{eq:cp-identification}\index{h projective@$h_A$}
\end{equation}
identifies $\Herm_0(d)$ with $K$. It is $PU(d)$-equivariant,
where the action on matrices is unitary conjugation, and
$\dim_{\R}K=d^2-1$. The invariant weighted tensor norm corresponds
to a positive multiple of the Frobenius norm.

For odd $d=2n+1$, Kr\"oncke \cite[Theorem 6.1]{Kroncke} proved that
every nonzero infinitesimal solitonic deformation of $\CP^{2n}$ is
obstructed at second order. In the terminology of
\cite[Definition 2.20 and Proposition 2.22]{LZrigid}, the Fubini--Study
shrinker therefore has an obstruction of order two. We compute the
coefficient of the reduced field in our normalization for every $d\ge3$.

\begin{proposition}\label[proposition]{prop:cp-quadratic}
Under \eqref{eq:cp-identification}, the first reduced field is
\begin{equation}
 F_2(A)=\frac2{d+2}\left(A^2-\frac{\tr(A^2)}d I\right).
 \label{eq:cp-F2}\index{F quadratic@$F_2$}
\end{equation}
If $d$ is odd, there is $c>0$ such that $|F_2(A)|\ge c|A|^2$. In this case $r=k_{\mathrm{LZ}}=2$, and every nontrivial slow Ricci-flow solution has $\ell_*=3$ and rate $|s|^{-1}$.
\end{proposition}
\begin{proof}
We first compute the projected second variation, with its normalization.
Let $u,v$ satisfy $\bar\Delta u+u=\bar\Delta v+v=0$.
Consider the conformal path $g(t)=(1+tu)\bg$, and write
$f(t)=f_{g(t)}$.
The potential linearization gives $f'(0)=(d-2)u$; see the proof of
\cite[Lemma A.2]{LZcp}. The second variation in
\cite[Lemma A.2(i), equation (A.1)]{LZcp}, with real dimension $2(d-1)$, is
\begin{equation}
 \Phi^{(2)}(u\bg,u\bg) ={}-\bar\nabla^2f''(0) -(d-2)\mathrm{d}u\otimes\mathrm{d}u -2(d-2)u\bar\nabla^2u +(u^2-|\bar\nabla u|^2)\bg .
 \label{eq:cp-conformal-variation}
\end{equation}
The conformal direction can be replaced by $h_u$ after projection onto
$K$. Indeed, if $\phi_t$ is generated by $\bar\nabla u$, then
$\tilde g(t)=\phi_t^*g(t)$ satisfies $\tilde g'(0)=h_u$.
Since $\Phi(\bg)=\mathrm{D}\Phi_{\bg}[u\bg]=0$, differentiating
$\Phi(\tilde g(t))=\phi_t^*\Phi(g(t))$ twice gives
\[
 \Phi^{(2)}(h_u,h_u)+\mathrm{D}\Phi_{\bg}[\tilde g''(0)]
 =\Phi^{(2)}(u\bg,u\bg).
\]
The linear term pairs to zero with every $h_v\in K$ by the
self-adjointness in \cite[Lemma 2.5]{LZcp}; see also
\cref{prop:slice-field}. Hence it does not affect the neutral projection.

Put $I_{u^2v}=\int_Mu^2v\,\mathrm{d}\nu$.
The weight is constant on the projective background. Using
$\bar\Delta(u^2)=2|\bar\nabla u|^2-2u^2$ and $\bar\Delta v=-v$,
integration by parts gives
\[
 \int_Mv|\bar\nabla u|^2\,\mathrm{d}\nu
 =I_{u^2v}+\frac12\int_Mu^2\bar\Delta v\,\mathrm{d}\nu
 =\frac12 I_{u^2v}.
\]
Integrating one derivative from $\bar\nabla^2v$ and using
$\bar\Delta u=-u$, we obtain
\[
 \begin{aligned}
 \int_M\bar\nabla^2v(\bar\nabla u,\bar\nabla u)\,\mathrm{d}\nu
 &=-\int_M(\bar\Delta u)
       \langle\bar\nabla u,\bar\nabla v\rangle\,\mathrm{d}\nu-\frac12\int_M
       \langle\bar\nabla v,\bar\nabla|\bar\nabla u|^2\rangle\,\mathrm{d}\nu\\
 &=-\frac12\int_Mu^2\bar\Delta v\,\mathrm{d}\nu
   +\frac12\int_M(\bar\Delta v)|\bar\nabla u|^2\,\mathrm{d}\nu\\
 &=\frac12 I_{u^2v}-\frac14 I_{u^2v}
  =\frac14 I_{u^2v}.
 \end{aligned}
\]
Since $\mathrm{div}\,h_v=0$, all Hessian terms pair to zero.
Use also
$u\bar\nabla^2u=\frac12\bar\nabla^2(u^2)-\mathrm{d}u\otimes\mathrm{d}u$
and $\tr_{\bg}h_v=2(d-2)v$ in \eqref{eq:cp-conformal-variation}.
It follows that
\[
 \langle\Phi^{(2)}(h_u,h_u),h_v\rangle=2(d-2)I_{u^2v}.
\]
On the other hand, the integrated Bochner identity gives, for
$\bar\Delta w+w=0$,
\[
 \langle h_w,h_v\rangle
 =2(d-2)\int_Mwv\,\mathrm{d}\nu.
\]
These two identities prove
\begin{equation}
 \Pi_K\Phi^{(2)}(h_u,h_u)=h_{\Pi_1(u^2)},
 \label{eq:cp-neutral-projection}
\end{equation}
where $\Pi_1$ is the scalar projection onto $\{\bar\Delta u+u=0\}$.

To compute the scalar projection, use the harmonic decomposition
recalled in \cite[proof of Theorem 6.1]{Kroncke} and
\cite[proof of Lemma 2.10]{LZcp}. Write
$(A^2)_0=A^2-\tr(A^2)I/d$. We have
\begin{equation}
 (z^*Az)^2={}P_{2,2}(z) +\frac{2|z|^2}{d+2}\,z^*(A^2)_0z +\frac{\tr(A^2)}{d(d+1)}|z|^4,
 \label{eq:cp-harmonic-decomposition}
\end{equation}
where $P_{2,2}$ is harmonic of bidegree $(2,2)$.
Indeed, for the Euclidean Laplacian on $\C^d$ and any traceless
Hermitian $B$,
\[
 \Delta_{\C^d}(z^*Az)^2=8z^*A^2z,\qquad \Delta_{\C^d}(|z|^2z^*Bz)=4(d+2)z^*Bz,\qquad \Delta_{\C^d}|z|^4=8(d+1)|z|^2.
\]
Thus subtraction of the last two terms in
\eqref{eq:cp-harmonic-decomposition} leaves a harmonic polynomial.
On the unit sphere the three terms descend to mutually orthogonal
scalar eigenspaces on $\CP^{d-1}$. Consequently
\[
 \Pi_1(u_A^2)=u_{\frac2{d+2}(A^2)_0}.
\]
Recall from \eqref{eq:quadratic-coeff} that
\[
 F_2(h_A)=\Pi_K\Phi^{(2)}(h_A,h_A).
\]
Together with \eqref{eq:cp-neutral-projection}, the scalar projection
above proves \eqref{eq:cp-F2}.

Now assume $d$ is odd. If $F_2(A)=0$, then $A^2$ is scalar. A nonzero Hermitian $A$ would have
eigenvalues $+\lambda$ and $-\lambda$ with equal multiplicities because
$\tr A=0$. This is impossible when $d$ is odd. Thus $F_2$ has no nonzero
zero. Compactness of the unit sphere gives the uniform estimate.
The first obstruction is one half of $F_2$, so its uniform obstruction
order is two; \cref{prop:sandwich} gives $\ell_*=3$.
\end{proof}

\subsection{Characteristic matrices and their conjugacy orbits}
We now take $d=2n+1$. The leading matrix is
$b(s)=|s|^{-1}(B+o(1))$, with
\begin{equation}
 F_2(B)=B,\qquad B\ne0.\label{eq:cp-characteristic}
\end{equation}

\begin{proposition}\label[proposition]{prop:cp-characteristics}
Every solution of \eqref{eq:cp-characteristic} is unitarily conjugate to exactly one matrix
\begin{equation}
 B_p=\diag(a_pI_p,b_pI_q),\qquad q=d-p,\qquad 1\le p\le n,
 \label{eq:Bp}\index{B p@$B_p,\Sigma_p$ (characteristic type)}
\end{equation}
where
\begin{equation}
 a_p=\frac{(d+2)q}{2(q-p)},\qquad
 b_p=-\frac{(d+2)p}{2(q-p)}.\label{eq:apbp}
\end{equation}
Its unitary conjugacy orbit is
\begin{equation}
 \Sigma_p=PU(d)\cdot B_p
 \simeq PU(d)/P(U(p)\times U(q))\simeq\Gr_{\C}(p,d).
 \label{eq:grassmannian}\index{B p@$B_p,\Sigma_p$ (characteristic type)}
\end{equation}
Here $\Gr_{\C}(p,d)$ denotes the Grassmannian of complex
$p$-dimensional subspaces of $\C^d$.
\end{proposition}
\begin{proof}
Every eigenvalue $x$ of $B$ satisfies $\frac2{d+2}(x^2-v)=x$, where $v=\tr(B^2)/d$. There are at most two roots. A single eigenvalue would give $B=0$ by tracelessness. If the two eigenvalues are $a,b$ with multiplicities $p,q$, subtracting their equations gives $a+b=(d+2)/2$, whereas $pa+qb=0$. Solving gives \eqref{eq:apbp}. Relabel so $p<q$. Then $a_p>0$, $b_p<0$, and $1\le p\le n$. Their multiplicities determine $p$ uniquely. The stabilizer gives \eqref{eq:grassmannian}, of real dimension $2pq$.
\end{proof}

\subsection{Linearization and stable fibers}
At $B=B_p$,
\begin{equation}
 \mathrm{D}F_2(B)H=\frac2{d+2}\left(BH+HB-\frac{2\tr(BH)}d I\right).
 \label{eq:cp-linearization}
\end{equation}
Use the orthogonal splitting
\begin{equation}
 \Herm_0(d)=\R B_p\oplus T_{B_p}\Sigma_p
       \oplus\Herm_0(p)\oplus\Herm_0(q).\label{eq:cp-splitting}
\end{equation}
The two final summands are the traceless diagonal blocks; the Grassmannian tangent space is the off-diagonal Hermitian block.

\begin{proposition}\label[proposition]{prop:cp-indicial}
On the four summands in \eqref{eq:cp-splitting}, respectively, the eigenvalues of $\mathrm{D}F_2(B_p)$ are
\begin{equation}
 2,\qquad 1,\qquad
 \lambda_p:=\frac{2q}{q-p},\qquad -\frac{2p}{q-p}.
 \label{eq:cp-indicial-values}
\end{equation}
The third summand is absent if $p=1$.
\end{proposition}
\begin{proof}
The radial eigenvalue is two by homogeneity. An off-diagonal $H$ has $\tr(B_pH)=0$ and $B_pH+HB_p=(a_p+b_p)H$, giving eigenvalue one. For $H\in\Herm_0(p)$, the trace term vanishes and the eigenvalue is $4a_p/(d+2)=2q/(q-p)$. The same calculation in the $q$-block gives $4b_p/(d+2)=-2p/(q-p)$.
\end{proof}

The largest eigenvalue controlling the exact reduced trajectory is
\begin{equation}
 \rho(B_p)=
 \begin{cases}
 2,&p=1,\\
 \dfrac{2q}{q-p},&p\ge2.
 \end{cases}
 \qquad \kappa_N=N,\qquad N>\rho(B_p).
 \label{eq:cp-threshold}\index{A characteristic@$A_{a_*},\rho(a_*)$}
\end{equation}
When $p=1$, the summand $\Herm_0(p)$ is zero, so the largest eigenvalue is the radial eigenvalue $2$.

To identify the leading stable directions, first consider the
quadratic truncation $b_s=F_2(b)$. Set $\tau=\log|s|$ and
$A(\tau)=|s|b(s)$, where $s=-e^\tau$. Its rescaled equation is
\begin{equation}
 A_\tau=G_2(A):=A-F_2(A).\label{eq:cp-rescaled}
\end{equation}
The full reduced field will be treated by the augmented system
\eqref{eq:cp-augmented}.
Put $\gamma_p=\lambda_p-1=d/(q-p)>1$. The spectrum of $\mathrm{D}G_2(B_p)$ is
\begin{equation}
 -1,\quad 0,\quad-\gamma_p,\quad+\gamma_p.
\end{equation}
The kernel is exactly $T_{B_p}\Sigma_p$. Since $\Sigma_p$ is compact and consists of equilibria, it is a normally hyperbolic invariant manifold. The stable-foliation theorem \cite[Theorem~4.1(e), (h)]{HPS} gives fibers
\begin{equation}
 W^s_{\mathrm{loc}}(\Sigma_p)=\coprod_{B\in\Sigma_p}W^s_{\mathrm{loc}}(B),
 \qquad T_BW^s_{\mathrm{loc}}(B)=E_B^s.
\end{equation}
Here the fiber over $B$ consists of initial points whose forward rescaled trajectories stay local and converge to that specified $B$. In particular,
\begin{equation}
 E^s_{B_p}=\R B_p\oplus\Herm_0(p),\qquad
 \dim W^s_{\mathrm{loc}}(B_p)=p^2.\label{eq:cp-stable-space}\index{W stable fiber@$W^s_{\mathrm{loc}}(B)$ (stable fiber)}\index{E stable@$E_B^s$}\index{B p@$B_p,\Sigma_p$ (characteristic type)}
\end{equation}
The stable-foliation theorem therefore gives a stable fiber over each marked limit. The neutral variables parametrize the Grassmannian base.

\subsection{Local Weyl-chamber models and the required order}
The local quotient that occurs in the type-$p$ moduli is
\begin{equation}
 \Herm_0(p)/U(p)\simeq
 \mathcal W_p:=\{(\xi_1,\ldots,\xi_p):\xi_1\le\cdots\le\xi_p,\ \sum_i\xi_i=0\}.
 \label{eq:Weyl}\index{W Weyl@$\mathcal W_p$ (Weyl chamber)}
\end{equation}
This is a closed Weyl chamber with generic dimension $p-1$; its walls record enhanced symmetry. The other component of the full Fubini--Study isometry group, represented by complex conjugation, preserves the unordered eigenvalues and adds no further identification in \eqref{eq:Weyl}.

\begin{theorem}[Projective parametrization]\label[theorem]{thm:cp-moduli}
For ancient Ricci flows with smooth backward tangent flow $(\CP^{2n},g_{\mathrm{FS}})$, every nontrivial slow germ has a characteristic type $p\in\{1,\ldots,n\}$. After quotienting time-independent diffeomorphisms, parabolic scaling, and original-time translation, the local slow moduli admit compatible local parametrizations on common ancient intervals with type-$p$ model $\mathcal W_p$ from \eqref{eq:Weyl}. The collection of local models is
\begin{equation}
 \coprod_{p=1}^{n}\mathcal W_p,
 \label{eq:cp-models}
\end{equation}
and these asymptotic data determine the maximal ancient continuations.
\end{theorem}
\begin{proof}
By \cref{prop:cp-quadratic}, every slow solution is principal and its
marked limit belongs to exactly one Grassmannian \eqref{eq:grassmannian}.
Fix $B_p$ and choose $N>\rho(B_p)$. We first describe the admissible
trajectories of the full reduced field $F^{[N]}$.
For an exact reduced trajectory $\beta$, set
$A=|s|\beta$, $z=|s|^{-1}$, and $\tau=\log|s|$. The augmented equation is
\begin{equation}
 A_\tau=A-\sum_{j=2}^{N}z^{j-2}F_j(A),\qquad z_\tau=-z.
 \label{eq:cp-augmented}
\end{equation}
Every $F_j$ is $PU(d)$-equivariant. The stabilizer
$G_p=P(U(p)\times U(q))$ fixes exactly $\R B_p$, so
$F_j(B_p)\in\R B_p$. Hence the coupling from the $z$-direction
at $(B_p,0)$ is radial. By \cref{prop:cp-indicial},
$\Sigma_p\times\{0\}$ is normally hyperbolic and its stable space at
$(B_p,0)$ is
\[
 \bigl(\R B_p\oplus\Herm_0(p)\bigr)\oplus\R_z.
\]
The stable-foliation theorem \cite[Theorem~4.1(e), (h)]{HPS} gives a stable fiber of
dimension $p^2+1$ over $(B_p,0)$. Its intersection with
$z=|s_*|^{-1}$ is transverse and has dimension $p^2$.
It parametrizes the local exact reduced trajectories with marked
limit $B_p$. Every such germ enters this local family on a sufficiently
remote tail. Orthogonal projection onto
$\R B_p\oplus\Herm_0(p)$ gives an equivariant coordinate chart
on this intersection with $z=|s_*|^{-1}$, after shrinking the neighborhood.

We make explicit the topology of these local models. On a common
tail, give the bounded SMRF classes the topology of $\XP_1$, using
the fixed-slice version of the norms in
Appendix~\ref{app:parabolic}. Parametrize the reduced stable fiber
by its value at the fixed terminal time $s_*$, and identify each complementary parameter space $E_+^\beta(s_*)$ with $E_+$
by $\xi=\Pi_+\zeta\in E_+$. The stable-fiber theorem gives smooth
dependence of $A(\tau)$, with all parameter derivatives bounded
for $\tau\ge\log|s_*|$ on a smaller chart. Since $\beta=|s|^{-1}A$,
this gives smooth dependence of $\beta$ in $\XP_1$.

The Green operators also depend smoothly on these parameters.
The complement inverse has the fixed terminal normalization
$\Pi_+y(s_*)=0$, and the neutral inverse is uniquely determined
in $\XP_\kappa$. Subtracting the corresponding equations gives
\[
 \Gop_\beta-\Gop_{\widetilde\beta}
 =\Gop_\beta(B_\beta-B_{\widetilde\beta})\Gop_{\widetilde\beta},
 \qquad
 \Hop_\beta-\Hop_{\widetilde\beta}
 =\Hop_\beta(A_\beta-A_{\widetilde\beta})\Hop_{\widetilde\beta}.
\]
The first identity acts from $\YP_{\kappa+1}$ to
$\XP_{\kappa+1}$ and the second from $\YP_{\kappa+1}$ to
$\XP_\kappa$. For the latter, the required coefficient norm is
$\sup_{s\le s_*}|s|\|A_\beta-A_{\widetilde\beta}\|_{\mathrm{op}}$,
which is controlled by $\|\beta-\widetilde\beta\|_{\XP_1}$.
The homogeneous equation with fixed $\xi$ has the same parameter
dependence. The defect and remainder maps in
\eqref{eq:LP-center}--\eqref{eq:LP-complement} are smooth between
the displayed weighted spaces. Their uniform contraction
therefore makes the realizing solution depend smoothly on the
finite-dimensional parameters in $\XP_1$.

Restrict to a parameter chart with compact closure inside the
admissible family and its fiber bound.
The realization is continuous and injective on that closure by
\cref{thm:parametrization}, hence is a homeomorphism onto its image
in the Hausdorff space $\XP_1$. In particular, extraction of the exact reduced trajectory and complementary
parameter is continuous on each such local solution
chart. Restriction to a common smaller tail gives compatible
charts. The isometry and time actions described in
Section~\ref{sec:phase} are continuous in this topology; their
smoothness on the marked parameter charts follows from this
dependence and the smooth dependence at positive times in the
parabolic and elliptic constructions used in the proof of
\cref{lem:interior-smoothing}.
Thus the transverse section and compact-group quotient below
carry the local quotient topology.

The line $\R B_p$ is invariant under the full reduced field, and
$\beta=aB_p$ satisfies $a_s=a^2+O(a^3)$. Choose a point $a_0B_p$ on the invariant line $\R B_p$ in this local family, with $a_0>0$. The reduced-time
translation has a nonzero radial velocity there. In the intersection with $z=|s_*|^{-1}$, take the
$G_p$-invariant section
\[
 \langle\beta(s_*),B_p\rangle=a_0|B_p|^2.
\]
It is transverse to this velocity and represents the local quotient
by reduced-time translation. The remaining orthogonal-projection
coordinate identifies the section with a neighborhood of zero in
$\Herm_0(p)$. The residual action is $H\mapsto UHU^*$,
$U\in U(p)$, which gives \eqref{eq:Weyl}.

The isometry quotient removes the Grassmannian of marked limits,
and Ricci-flow parabolic scaling is the reduced-time translation
just removed. By \eqref{eq:cp-spectrum-input} and \cref{prop:phase},
original Ricci-flow time translation removes the remaining scalar parameter
in the Ricci direction. Realization and uniqueness follow from
\cref{thm:parametrization}; \cref{prop:entrance}
and \cref{rem:principal-germs} place every slow germ in a bounded
local parametrization. Finally, uniqueness of continuation from an ancient interval in
Section~\ref{sec:horizontal} determines the maximal continuation.
\end{proof}

\begin{remark}\label[remark]{rem:kahler-rigidity}
Every nontrivial slow solution in \cref{thm:cp-moduli} is non-K\"ahler.
Indeed, suppose an ancient flow in this tangent class is K\"ahler
with respect to a fixed complex structure $J'$. Along a backward
tangent sequence, the pulled-back complex structures are parallel
and orthogonal for metrics converging smoothly to $g_{\mathrm{FS}}$.
They therefore converge subsequentially to the standard complex
structure, after composition with complex conjugation if necessary.
Thus that complex structure is adjacent to $J'$ in the sense of
Sun--Wang, and $c_1(M,J')\ne0$.
Since $H^2(\CP^{2n};\R)$ is one-dimensional and
$\partial_t[\omega]=-4\pi c_1(M,J')$, a Ricci-flow time translation
arranges that $\hat g(s)=e^s g(-e^{-s})$ has K\"ahler class
$4\pi c_1(M,J')$.
Starting at a sufficiently negative time, apply
\cite[Theorem 1.1]{SW} to $\hat g/2$. Its adjacency clause gives
forward convergence of $\hat g$ modulo diffeomorphisms to
$g_{\mathrm{FS}}$. The backward tangent gives the same entropy limit at the
other end. Perelman's monotonicity \cite{Perelman} forces the entropy
to be constant, so the flow is the canonical shrinking Ricci flow generated by the Fubini--Study shrinker.
\end{remark}
\Cref{tab:cp-orders} compares the local moduli dimensions and the least expansion orders allowed by \eqref{eq:cp-threshold} for representative characteristic types.

\begin{table}[htbp]
\centering
\caption{Indicial thresholds and dimensions of the projective local slow models.}
\label{tab:cp-orders}
\begin{tabular}{cccccc}
\toprule
Space & $p$ & $q$ & $\rho(B_p)$ & Least $N$ & Moduli dimension\\
\midrule
$\CP^2$ & 1 & 2 & 2 & 3 & 0\\
$\CP^4$ & 1 & 4 & 2 & 3 & 0\\
$\CP^4$ & 2 & 3 & 6 & 7 & 1\\
$\CP^6$ & 2 & 5 & $10/3$ & 4 & 1\\
$\CP^6$ & 3 & 4 & 8 & 9 & 2\\
$\CP^8$ & 2 & 7 & $14/5$ & 3 & 1\\
$\CP^8$ & 3 & 6 & 4 & 5 & 2\\
$\CP^8$ & 4 & 5 & 10 & 11 & 3\\
\bottomrule
\end{tabular}
\end{table}
For fixed $n\ge2$, the common choice $N=2n+3$ covers every type, since the extremal type $p=n$, $q=n+1$ has $\rho=2n+2$. For $n=1$ the least order is instead $N=3$.

For $\CP^2$, there is only $p=1$, $q=2$, and
\begin{equation}
 B_* = \diag(5,-5/2,-5/2).
\end{equation}
The stable fiber is radial; $\Herm_0(1)=0$. The scalar reduced equation has one nonzero ancient orbit germ modulo translation: separation of variables in $a_s=a^2+O(a^3)$ shows that its only integration constant is the time-translation parameter. Original-time translation removes the scalar parameter in the Ricci direction.

\begin{corollary}[The $\CP^2$ case]\label[corollary]{cor:cp2}
The ancient Ricci-flow germs with smooth backward tangent $(\CP^2,g_{\mathrm{FS}})$ have exactly two classes modulo time-independent diffeomorphism, parabolic rescaling, and original-time translation: the canonical shrinking Ricci flow generated by the Fubini--Study shrinker and one nontrivial slow class.
\end{corollary}
\begin{proof}
The exponentially convergent case is the canonical shrinking Ricci flow by \cref{cor:fast-Ricci}. All slow solutions have the unique characteristic type and the scalar reduced orbit just described. Surjectivity produces a nontrivial representative. After quotienting by parabolic rescaling and Ricci-flow time translation, injectivity identifies all such slow germs. Its nonzero neutral leading term distinguishes it from the canonical shrinking flow.
\end{proof}
A marked leading representative is
\begin{equation}
 h(s)=\frac{u_*\bg+2\nabla^2u_*}{|s|}+o(|s|^{-1}),\qquad
 u_*([z])=\frac52\,
 \frac{2|z_0|^2-|z_1|^2-|z_2|^2}{|z_0|^2+|z_1|^2+|z_2|^2}.
\end{equation}

\subsection{A nonprincipal slow solution in odd complex dimension}
\label{sec:cp-nonprincipal}\index{Nonprincipal slow solution}
We now take $d=2n+2$, so that $M=\CP^{2n+1}$, with $n\ge1$.
Set
\[
 A_0=\diag(I_{n+1},-I_{n+1}),\qquad
 \Gamma=P(U(n+1)\times U(n+1))\subset PU(2n+2).\index{Gamma projective symmetry@$\Gamma$ (projective symmetry)}
\]
The two blocks have equal size. The quadratic field vanishes in this
direction, as required by \cref{rem:nonprincipal-direction} for a
nonprincipal characteristic vector. The cubic term produces an ancient
solution with this direction.

\begin{proposition}
\label[proposition]{prop:cp-nonprincipal}
For every $n\ge1$, there is a smooth noncollapsed ancient Ricci-flow germ
on $\CP^{2n+1}$ with Fubini--Study backward tangent Ricci shrinker and a
$\Gamma$-invariant horizontal representative satisfying
\begin{equation}
 h(s)=\sqrt{\frac{(n+2)(2n+5)}{8n}}\,|s|^{-1/2}h_{A_0}
       +O_{C^k}(|s|^{-1})
 \label{eq:cp-nonprincipal-profile}
\end{equation}
for every fixed $k$. For the full fixed-slice equation, $r=2$ and this
solution has $\ell_*=4>r+1$.
\end{proposition}
\begin{proof}
Restrict the entire fixed-slice equation to $\Gamma$-invariant tensors.
For each tensor space, the superscript $\Gamma$ denotes the subspace
fixed by pullback under every element of $\Gamma$.
A traceless Hermitian matrix commuting with $U(n+1)\times U(n+1)$ is a
multiple of $A_0$, so
\[
 K^\Gamma=\R h_{A_0},\qquad\index{K invariant neutral@$K^\Gamma$}
 \Sli^\Gamma=\R h_{A_0}\oplus\Comp^\Gamma,
 \qquad \Comp^\Gamma=\Comp\cap\Sli^\Gamma.
\]
Averaging over $\Gamma$ commutes with $L$ and its spectral projections.
Hence $L$ is invertible on $\Comp^\Gamma$, with the same elliptic
estimates. Equivariance of $\mathcal V$ and uniqueness in the graph and Green
constructions show that the expansions and the realization argument of
Section~\ref{sec:principal} restrict to these spaces. Every
$F_j(xh_{A_0})$ lies in $\R h_{A_0}$. Block interchange sends $h_{A_0}$ to
$-h_{A_0}$ and normalizes $\Gamma$, so all even restricted coefficients
vanish. In particular, $F_2(xh_{A_0})=0$.

We compute the first nonzero restricted coefficient:
\begin{equation}
 F_3(xh_{A_0})=\gamma_nx^3h_{A_0},\qquad
 \gamma_n=\frac{4n}{(n+2)(2n+5)}>0.
 \label{eq:cp-balanced-cubic}
\end{equation}
Put $g_x=\bg+xh_{A_0}+x^2H_2(h_{A_0})$.
By \cref{lem:low-order} and $F_2(h_{A_0})=0$, the defect
$\Phi(g_x)$ vanishes through degree two. Its third derivative is
\[
 T=\Phi^{(3)}(h_{A_0},h_{A_0},h_{A_0}) +6\Phi^{(2)}(h_{A_0},H_2(h_{A_0})),\qquad \Pi_KT=3F_3(h_{A_0}).
\]
Differentiating the nonlinear Bianchi identity
\eqref{eq:nonlinear-bianchi} gives $\mathrm{div}_{\bg}T=0$.
Write $u=u_{A_0}$. Pulling $g_x$ back by the flow of $-\bar\nabla u$
changes its first variation to $u\bg$ and leaves its third defect $T$
unchanged, since the lower defects vanish. The resulting second
variation solves the second-order conformal deformation equation.
The linearized term involving the third metric variation pairs to zero
with $u\bg$, by self-adjointness and $\mathrm{D}\Phi_{\bg}[u\bg]=0$. Also
$\langle T,2\bar\nabla^2u\rangle_{L^2}=0$.
Consequently the signed identity of Li--Zhang
\cite[Theorem 3.15]{LZcp}, divided by the background volume, gives
\begin{equation}
 3\langle F_3(h_{A_0}),h_{A_0}\rangle_{L^2}
 =\langle T,u\bg\rangle_{L^2}
 =\frac{48n^2}{(n+2)(2n+3)(2n+5)}.
 \label{eq:cp-balanced-pairing}
\end{equation}
The integrated Bochner identity and
$\int_Mu^2\,\mathrm{d}\nu=(2n+3)^{-1}$
\cite[Lemma 3.4(c)]{LZcp} give
\[
 \|h_{A_0}\|_{L^2}^2=4n\int_Mu^2\,\mathrm{d}\nu
                   =\frac{4n}{2n+3}.
\]
Dividing \eqref{eq:cp-balanced-pairing} by this norm and by three proves
\eqref{eq:cp-balanced-cubic}.

The restricted first nonzero degree is therefore three. Its order-five
reduced equation is
\[
 \beta=xh_{A_0},\qquad x_s=\gamma_nx^3+c_5x^5
\]
for a real coefficient $c_5$. Choose $x(s_0)=x_0>0$ so small that
$\gamma_n+c_5x^2>0$ for $0\le x\le x_0$. Backward in time, $x$ decreases
toward zero, and separation of variables gives
\[
 s_0-s=\int_{x(s)}^{x_0}
       \frac{\mathrm{d}t}{\gamma_nt^3+c_5t^5}.
\]
The integral diverges as its lower endpoint tends to zero, so the
trajectory exists for all $s\le s_0$. Moreover,
\[
 \frac{\mathrm{d}}{\mathrm{d}s}(x^{-2})
       =-2\gamma_n-2c_5x^2.
\]
For sufficiently small $x_0$, this first gives $x^{-2}\asymp|s|$ as
$s\to-\infty$. Thus $x^2=O(|s|^{-1})$, and integrating once more gives
$x^{-2}=2\gamma_n|s|+O(\log|s|)$. Taking the inverse square root yields
\[
 x(s)=(2\gamma_n|s|)^{-1/2}
          +O(|s|^{-3/2}\log|s|).
\]
At its characteristic vector $a_*=(2\gamma_n)^{-1/2}h_{A_0}$, the
restricted indicial operator is $\mathrm{D}F_3(a_*)|_{K^\Gamma}=\frac32\Id$.
The order $N=5$ satisfies the threshold \eqref{eq:indicial-gap}, since
\[
 \kappa_5=(5+1)\frac12-1=2>\frac32.
\]
Apply \cref{thm:parametrization} to the invariant equation with this
trajectory and zero complementary parameter. The resulting smooth SMRF has
$h=\beta+O_{C^k}(|s|^{-1})$, which is
\eqref{eq:cp-nonprincipal-profile}.
The reconstruction in Section~\ref{sec:horizontal} gives a noncollapsed
ancient Ricci-flow germ with the prescribed tangent.
Finally, \eqref{eq:cp-F2} is nonzero on the full neutral space, so
$r=2$, whereas \eqref{eq:cp-nonprincipal-profile} gives $\ell_*=4$.
\end{proof}

\section{Forward convergence to compact Ricci shrinkers}\label{sec:forward}
We consider Ricci flows with a smooth compact tangent flow at a finite singular time. Chan--Ma--Zhang \cite[Theorems~1.2 and~1.7, Corollary~1.4]{CMZ} prove uniqueness of this tangent flow, the Type~I curvature bound, and polynomial convergence of the modified flow in logarithmic time. For integrable shrinkers, Choi--Lai \cite[Theorem~1.8]{CL} give an exponential convergence estimate determined by the spectrum. We determine the leading tensor and convergence rate without an integrability assumption, and obtain the finite-order expansion of every slow solution. The equation and its reduced coefficients are those of Sections~\ref{sec:geometry}--\ref{sec:exponent}.

\subsection{Normalization and the main theorem}
Let $(M,g_{\mathrm{RF}}(t))$, $0\le t<T<\infty$, be a closed Ricci
flow with first singular time $T$. Suppose one tangent flow at $T$ is
smooth and compact, and identify its normalized shrinker with $(M,\bg,\bfv)$.
Set
\begin{equation}
 \tau=T-t=e^{-s},\qquad
 \hat g(s)=\tau^{-1}g_{\mathrm{RF}}(T-\tau),\qquad
 \partial_s\hat g=\hat g-2\Ric_{\hat g}.
 \label{eq:forward-normalization}
\end{equation}
Thus $s\nearrow+\infty$ as $t\nearrow T$. We use $s^*\ge2$ for the starting time of a forward SMRF. In a horizontal gauge its relation to the Ricci flow is
\begin{equation}
 \tau^{-1}\phi_t^*g_{\mathrm{RF}}(t)=\bg+h(s),\qquad
 h_s=\mathcal V(h),\qquad h(s)\longrightarrow0,
 \quad s\ge s^*.
 \label{eq:forward-gauge}
\end{equation}
The spaces $\Sli,K,\Comp$, the operator $L$, and the group $G$ retain their definitions from Section~\ref{sec:geometry}. All tensor asymptotics below hold in $C^k$ for every integer $k\ge0$.

\begin{theorem}[Forward rate and tensor asymptotics]\label[theorem]{thm:forward}
Let $g_{\mathrm{RF}}(t)$, $0\le t<T<\infty$, be a Ricci flow on a closed manifold, with first singular time $T$. Suppose one tangent flow at $T$ is smooth and compact, and identify its normalized shrinker with $(M,\bg,\bfv)$. There is a horizontal representative \eqref{eq:forward-gauge}. Exactly one of the following holds.
\begin{enumerate}
\item The slice solution is zero on a tail, and the Ricci flow there is the canonical shrinking flow up to a time-independent spatial identification.
\item There are $\lambda\ge\lambda_->0$ and a nonzero smooth tensor $a\in\Sli$ such that
\begin{equation}
 La=-\lambda a,\qquad
 h(s)=e^{-\lambda s}\bigl(a+o_{C^k}(1)\bigr).
 \label{eq:forward-spectral}
\end{equation}
\item There are a rational number $\ell^*\ge3$ and $a^*\in K\setminus\{0\}$ such that
\begin{equation}
 b(s)=s^{-1/(\ell^*-2)}(a^*+o(1)),\qquad
 h(s)=s^{-1/(\ell^*-2)}\bigl(a^*+o_{C^k}(1)\bigr),
 \quad b=\Pi_Kh.
 \label{eq:forward-neutral}\index{ell characteristic@$\ell_*,\ell^*$ (characteristic exponents)}\index{a characteristic@$a_*,a^*$ (characteristic vectors)}
\end{equation}
For each integer $N\ge2$, the maps of \cref{thm:formal} satisfy
\begin{equation}
 h(s)=\Psi^{[N]}(b(s))+O_{C^k}(|b(s)|^{N+1}),\qquad b_s=F^{[N]}(b)+O(|b|^{N+1}).
 \label{eq:forward-expansion}
\end{equation}
The first nonzero reduced degree $r=\min\{j\ge2:F_j\not\equiv0\}$ exists, and $\ell^*\ge r+1$. The possible exponents $\ell^*$ lie in a finite set determined by the reduced entropy. In the principal case $\ell^*=r+1$,
\begin{equation}
 F_r(a^*)=-\frac{a^*}{r-1}.
 \label{eq:forward-balance}
\end{equation}
\end{enumerate}
Changing the horizontal marking acts on the limiting tensor by one element of $G$.
\end{theorem}

In the original Ricci-flow time, the two nonstationary alternatives give, respectively,
\begin{equation}
 \phi_t^*g_{\mathrm{RF}}(t) =\tau\bg+\tau^{1+\lambda}\bigl(a+o_{C^k}(1)\bigr),\qquad \phi_t^*g_{\mathrm{RF}}(t) =\tau\bg+\tau|\log\tau|^{-1/(\ell^*-2)} \bigl(a^*+o_{C^k}(1)\bigr).
 \label{eq:forward-ricci-rates}
\end{equation}

\subsection{The spectral alternatives}
Let $h_s=\mathcal V(h)$ be a smooth solution with $h(s)\to0$ smoothly as
$s\to+\infty$. Use the three $L^2$ norms
$U_\pm=\|\Pi_\pm h\|_{L^2}$ and $U_0=|b|$, where $b=\Pi_Kh$.
They satisfy \cref{lem:projected-inequalities}, with
$\epsilon(s)\to0$.

\begin{proposition}\label[proposition]{prop:forward-dichotomy}
Apart from a stationary tail, either $\|h(s)\|_k=O(e^{-\gamma s})$
for some $\gamma>0$ and every fixed $k$, or $b(s)\ne0$ on a future tail and
\begin{equation}
 \|\Pi_{\Comp}h(s)\|_k=o(|b(s)|),\qquad
 |b_s|\le C|b|^2,\qquad |b(s)|\ge\frac{c}{s},\qquad
 |\partial_s\log|b(s)||\longrightarrow0.
 \label{eq:forward-slow}
\end{equation}
The alternatives are disjoint.
\end{proposition}
\begin{proof}
The comparison is the forward version of \cref{lem:two-mode}.
Since a stationary tail is excluded, $h(s)\ne0$ by uniqueness.
For fixed $\delta>0$, the function $U_+-\delta(U_0+U_-)$ has positive
derivative at every zero on a sufficiently late tail.
Its nonnegative region is forward invariant, and there
$\partial_sU_+\ge\lambda_+U_+/2$ after increasing the starting time.
Convergence excludes this region. Hence $U_+=o(U_0+U_-)$.
At a zero of $U_--\delta U_0$, the projected inequalities
now give a negative derivative. Either $U_-=o(U_0)$, or
$U_-\ge\delta U_0$ on an entire future tail for some $\delta>0$.
In the latter case,
$\partial_sU_-\le-\lambda_-U_-/2$ on a later tail,
so $\|h\|_{L^2}=O(e^{-\lambda_-s/2})$.
Interpolation with smooth convergence gives exponential decay in
every fixed spatial norm, as in \cref{prop:dichotomy}.

In the remaining case, $U_++U_-=o(|b|)$.
Hence $b\ne0$ on a late tail, and \eqref{eq:ODIzero} gives
$|b_s|=o(|b|)$.
The regularity argument is the one in \cref{prop:dichotomy}.
Indeed, interpolation and \eqref{eq:L2-nonlinear} give
\[
 \|\mathcal U(h(s))\|_{L^2}\le C_\theta|b(s)|^{2-\theta}
\]
for every sufficiently small $\theta>0$.
For $q=\Pi_{\Comp}h$, the fixed-operator formula
\eqref{eq:forward-green-formula} gives
\[
 q(s)=e^{(s-s^*)L_{\Comp}}h_-(s^*)
       +\Gop_0\bigl(\Pi_{\Comp}\mathcal U(h)\bigr)(s).
\]
The same spectral estimate as \eqref{eq:append-L2}, with
$w=|b|^{2-\theta}$, bounds the forced term by $C_\theta|b|^{2-\theta}$
in $L^2$. Slow variation of $|b|$ makes the exponentially decaying
homogeneous term smaller than this weight.
Thus $\|q\|_{L^2}=O(|b|^{2-\theta})$, and the second interpolation
in \cref{prop:dichotomy} gives $\|q\|_k=o(|b|)$ at every fixed order.
In particular $\|h\|_3\le C|b|$, so \eqref{eq:L2-nonlinear} gives
$|b_s|\le C|b|^2$. Integration yields
\[
 |b(s)|\ge\frac1{|b(s^*)|^{-1}+C(s-s^*)}\ge\frac{c}{s}.
\]
This lower bound also proves disjointness.
\end{proof}
\begin{proposition}\label[proposition]{prop:forward-leading}
Every nonstationary exponentially convergent forward solution satisfies \eqref{eq:forward-spectral}.
\end{proposition}
\begin{proof}
Exponential convergence and \eqref{eq:L2-nonlinear} give
\[
 h_s=Lh+\mathcal U(h),\qquad
 \|\mathcal U(h(s))\|_{L^2}\le Ce^{-\gamma s}\|h(s)\|_{L^2}
\]
for some $\gamma>0$.
Since a stationary tail is excluded, uniqueness implies $h(s)\ne0$.
Define
\[
 Q=-\frac{\langle Lh,h\rangle}{\|h\|_{L^2}^2},\qquad
 D=\frac{\|(L+Q)h\|_{L^2}}{\|h\|_{L^2}}.
\]
Differentiating the Dirichlet quotient gives
\begin{equation}
 Q'=-2D^2-
 \frac{2\langle(L+Q)h,\mathcal U(h)\rangle}{\|h\|_{L^2}^2}
 \le-D^2+C^2e^{-2\gamma s}.
 \label{eq:forward-dirichlet}
\end{equation}
Thus $Q(s)+\int_s^\infty C^2e^{-2\gamma v}\dd v$ decreases.
The spectrum of $-L$ is bounded below, so $Q$ is bounded below.
Consequently $Q\to\lambda<\infty$ and
$\int_{s^*}^\infty D^2\dd s<\infty$.
Along a sequence $s_i\nearrow\infty$, $D(s_i)\to0$. Since
\[
 \operatorname{dist}(Q(s_i),\Spec(-L))\le D(s_i),
\]
the discrete spectrum gives $\lambda\in\Spec(-L)$. Moreover,
\begin{equation}
 \partial_s\log\|h(s)\|_{L^2}=-Q+O(e^{-\gamma s}),\qquad
 \lim_{s\to\infty}\frac{\log\|h(s)\|_{L^2}}s=-\lambda<0.
 \label{eq:forward-log-rate}
\end{equation}
The strict inequality follows from exponential decay, so
$\lambda\ge\lambda_-$.

Let $P_0,P_>,P_<$ project onto the eigenvalues of $L$ equal to,
greater than, and smaller than $-\lambda$.
For each sufficiently small $\delta>0$, \eqref{eq:forward-log-rate} gives
\[
 \|h(s)\|_{L^2}\le C_\delta e^{-(\lambda-\delta)s},\qquad
 \|\mathcal U(h(s))\|_{L^2}\le
 C_\delta e^{-(\lambda+\gamma-\delta)s}.
\]
Choose $\delta<\gamma/2$ smaller also than the spectral gap above
$-\lambda$. Integrating
$\partial_s(e^{\lambda s}P_0h)=e^{\lambda s}P_0\mathcal U(h)$
gives a limit $a$ with $La=-\lambda a$.
Variation of constants on the other two spaces gives
\begin{equation}
 P_>h(s)=-\int_s^\infty e^{(s-v)L}P_>\mathcal U(h(v))\dd v,\qquad P_<h(s)=e^{(s-s^*)L}P_<h(s^*) +\int_{s^*}^{s}e^{(s-v)L}P_<\mathcal U(h(v))\dd v.
 \label{eq:forward-spectral-integrals}
\end{equation}
The boundary term at infinity vanishes by
\eqref{eq:forward-log-rate}.
The spectral gaps and the forcing bound therefore give
\begin{equation}
 \|h(s)-e^{-\lambda s}a\|_{L^2}
 =O(e^{-(\lambda+\sigma)s})
 \label{eq:forward-spectral-remainder}
\end{equation}
for some $\sigma>0$, decreased if a resonant factor $s$ occurs.
If $a=0$, this contradicts \eqref{eq:forward-log-rate}.
Hence $a\ne0$.

The remainder has uniformly bounded spatial norms of every fixed order.
For a prescribed $k$, interpolate \eqref{eq:forward-spectral-remainder}
with a sufficiently high such norm to obtain
\[
 \|h(s)-e^{-\lambda s}a\|_k
 \le C e^{-(\lambda+\sigma)(1-\theta)s}
 =o(e^{-\lambda s}),
 \qquad 0<\theta<\frac{\sigma}{\lambda+\sigma}.
\]
This proves \eqref{eq:forward-spectral} in every fixed spatial norm.
\end{proof}

\subsection{Complement estimates and finite-order expansion}
Use the strip norms $\XP(I_s)$ and $\YP(I_s)$ of \eqref{eq:append-norms}, with $I_s=[s-1,s]$. On the forward half-line their weighted versions are
\begin{equation}
 \|u\|_{\XP_w}=\sup_{s\ge s^*+1}
       \frac{\|u\|_{\XP(I_s)}}{w(s)},\qquad
 \|R\|_{\YP_w}=\sup_{s\ge s^*+1}
       \frac{\|R\|_{\YP(I_s)}}{w(s)}.
 \label{eq:forward-weighted-norms}
\end{equation}
The positive weight $w$ satisfies \eqref{eq:append-tempered} for time arguments in $[s^*,\infty)$. Fix $0<\gamma<\min\{\lambda_+,\lambda_-\}$. We use the following estimate with smooth coefficients and data arising from the fixed-slice equation.

\begin{lemma}\label[lemma]{lem:forward-green}
Let $R$ be smooth with $\|R\|_{\YP_w}<\infty$, and let $B(s)=L_{\Comp}+A(s)$, where
\[
 \|Au\|_{\YP(I_s)}\le\delta\|u\|_{\XP(I_s)},
 \qquad s\ge s^*+1.
\]
For $\delta$ sufficiently small, the normalized forced solution $\rho=\Gop_BR$ satisfies
\begin{equation}
 \rho_s=B\rho+R,\qquad \Pi_-\rho(s^*)=0,\qquad
 \|\rho\|_{\XP_w}\le C_w\|R\|_{\YP_w}.
 \label{eq:forward-green-estimate}
\end{equation}
It is unique in the weighted class. For $\delta$ small enough also for the constant and exponential weights, every smooth homogeneous solution with uniformly bounded unit-strip norms satisfies
\begin{equation}
 \|z\|_{\XP(I_s)}
 \le Ce^{-\gamma(s-s^*)}\|\Pi_-z(s^*)\|_3.
 \label{eq:forward-homogeneous}
\end{equation}
Consequently, if $w\to0$, every smooth solution of $q_s=Bq+R$ with uniformly bounded unit-strip norms decomposes as $q=\Gop_BR+z_-$ with these bounds. The constants are independent of the strip location and, for fixed weight and coefficient bounds, of $s^*$.
\end{lemma}
\begin{proof}
We adapt \cref{prop:weighted-green} and the contraction argument in \cref{thm:fast-classification}. For the fixed operator, the negative component is determined at $s^*$ and the positive component at infinity:
\begin{equation}
 (\Gop_0R)(s)
 =\int_{s^*}^{s}e^{(s-u)L_{\Comp}}\Pi_-R(u)\dd u
 -\int_s^\infty e^{(s-u)L_{\Comp}}\Pi_+R(u)\dd u.
 \label{eq:forward-green-formula}
\end{equation}
The spectral bounds used in \eqref{eq:append-L2} give
\[
 \sup_{s\ge s^*}\frac{\|\Gop_0R(s)\|_{L^2}}{w(s)}
 \le C_w\left(\frac1{\lambda_--\epsilon}
              +\frac1{\lambda_+-\epsilon}\right)
       \|R\|_{\YP_w}.
\]
On $s\ge s^*+2$, the spatial regularity estimate \eqref{eq:append-regularity} applies unchanged. On the initial interval $[s^*,s^*+2]$, its version with initial data adds $\|\Gop_0R(s^*)\|_3$. This is controlled by the displayed $L^2$ bound because $\Pi_-\Gop_0R(s^*)=0$ and $E_+$ is finite-dimensional. Weight comparability on adjacent strips therefore gives
$\|\Gop_0R\|_{\XP_w}\le C_w\|R\|_{\YP_w}$, including the initial strip. Smooth data give the time continuity required by these norms. Uniqueness follows coefficient by coefficient: the negative initial value vanishes, and a nonzero positive homogeneous coefficient violates the weighted bound at infinity.

Now solve $\rho=\Gop_0(A\rho+R)$ by the same contraction as in the appendix. For smooth negative initial data $\eta$, apply it to
\[
 z(s)=e^{(s-s^*)L_{\Comp}}\eta+\Gop_0(Az)(s)
\]
in the weight $e^{-\gamma(s-s^*)}$. The initial-data estimate gives \eqref{eq:forward-homogeneous}. Every bounded homogeneous solution satisfies this integral equation with $\eta=\Pi_-z(s^*)$; uniqueness in the constant weight identifies it with the exponentially decaying solution. Subtraction proves the decomposition.
\end{proof}

\begin{proposition}\label[proposition]{prop:forward-expansion}
Every neutral-dominated forward solution satisfies \eqref{eq:forward-expansion}. More precisely, for each fixed $N\ge2$,
\begin{equation}
 h=\Psi^{[N]}(b)+\rho_N+z_{-,N},\qquad \|\rho_N(s)\|_k\le C_{N,k}|b(s)|^{N+1},\qquad \|z_{-,N}(s)\|_k\le C_{N,k}e^{-\gamma(s-s^*)}.
 \label{eq:forward-slaving}
\end{equation}
In particular, $h=b+O_{C^k}(|b|^2)$.
\end{proposition}
\begin{proof}
The proof is that of \cref{thm:actual-expansion}, using \cref{lem:forward-green} in place of \cref{lem:slaving}. The formal coefficients and the exact coordinate identities in \cref{lem:exact-coordinates} are unchanged, since $h_s=\mathcal V(h)$. Thus $q=h-\Psi^{[N]}(b)$ satisfies
\[
 q_s=(L_{\Comp}+A_N(s))q+R_N(s),\qquad
 \|R_N(s)\|_{1}\le C_N|b(s)|^{N+1},
\]
with a small coefficient $A_N$. By \eqref{eq:forward-slow}, the weight $|b|^{N+1}$ is tempered with arbitrarily small exponent on a sufficiently late tail. The forward complement estimate gives \eqref{eq:forward-slaving}. Since $|b|\ge c/s$, the homogeneous remainder is smaller than every fixed power of $|b|$. Substitution into the neutral identity gives the second equation in \eqref{eq:forward-expansion}, exactly as in \cref{thm:actual-expansion}. The argument applies at every fixed spatial regularity.
\end{proof}

\subsection{The reduced entropy and decay exponent}
Use the static graph $b+W(b)$ and reduced entropy $\Ent$ from
Section~\ref{sec:exponent}.
The polynomial convergence estimate of Chan--Ma--Zhang
\cite[Corollary~1.4]{CMZ}, followed by the smooth slice construction
in Section~\ref{sec:horizontal}, gives
$\|h(s)\|_{L^2}\le Cs^{-\alpha_0}$ for some $\alpha_0>0$.
Combining this with \eqref{eq:forward-slow} and decreasing $\alpha_0$
if necessary, we obtain
\begin{equation}
 cs^{-1}\le|b(s)|\le Cs^{-\alpha_0},\qquad 0<\alpha_0\le1.
 \label{eq:forward-coarse-rates}
\end{equation}

\begin{proposition}\label[proposition]{prop:forward-exponent}
A neutral-dominated forward solution has the asymptotics \eqref{eq:forward-neutral}, with $\ell^*\ge r+1$, and
\begin{equation}
 -\Ent(b(s))\asymp|b(s)|^{\ell^*}.
 \label{eq:forward-entropy-rate}
\end{equation}
The possible exponents form a finite set. In the principal case,
\begin{equation}
 F_r(a^*)=-\frac{a^*}{r-1},\qquad
 b_s=-\frac1{r-1}s^{-r/(r-1)}(a^*+o(1)).
 \label{eq:forward-principal-velocity}
\end{equation}
\end{proposition}
\begin{proof}
The induction in \cref{prop:gradient-error}, using
\cref{lem:forward-green} for the complement equation, gives
\begin{equation}
 |b_s-\nabla\Ent(b)|
 \le C_J\bigl(|b|\,|\nabla\Ent(b)|+|b|^J\bigr),
 \qquad J\ge2.
 \label{eq:forward-subordinate}
\end{equation}
Indeed, the static coordinate identities and the inductive velocity
bound are unchanged. The forward complement estimate adds an
exponentially decaying homogeneous term, which is absorbed into
$|b|(|\nabla\Ent(b)|+|b|^j)$ by $|b|\ge c/s$.
This gives the same induction. Since $b\ne0$, increasing each $C_J$
over a finite interval makes all orders valid on one fixed tail.

The nonconstancy argument in \cref{thm:ell}, using
\eqref{eq:forward-coarse-rates}, shows that $\Ent$ is nonconstant.
Apply the remainder of that proof to $f=-\Ent$:
\eqref{eq:forward-subordinate} and \eqref{eq:forward-coarse-rates}
verify the hypotheses of Choi--Hung \cite[Theorem~7.1]{CH}.
That theorem gives the limiting direction and rational decay exponent,
with exponents in a finite set, and \cite[Lemma~7.10]{CH} gives
\eqref{eq:forward-entropy-rate}.
By \cref{lem:gradient-obstruction}, the first nonzero reduced
coefficient $F_r$ exists and $\Ent(b)=O(|b|^{r+1})$.
Comparison with \eqref{eq:forward-entropy-rate} gives
$\ell^*\ge r+1$.
The tensor limit follows from \cref{prop:forward-expansion}.

If $\ell^*=r+1$, put $A(s)=s^{1/(r-1)}b(s)$.
The reduced equation gives
\[
 sA_s=\frac1{r-1}A+F_r(A)+O(s^{-1/(r-1)})
 \longrightarrow\frac{a^*}{r-1}+F_r(a^*).
\]
Since $A(s)\to a^*$, this limit must vanish.
Thus $F_r(a^*)=-\frac{a^*}{r-1}$, and substitution into
$b_s=F_r(b)+O(|b|^{r+1})$ gives \eqref{eq:forward-principal-velocity}.
\end{proof}

\subsection{Geometric consequences}
\begin{proof}[Proof of \cref{thm:forward}]
By \cite[Theorems~1.2 and~1.7, Corollary~1.4]{CMZ}, the tangent is the canonical shrinking flow of a compact Ricci shrinker, and the canonically modified flow converges smoothly to it. The compact tangent manifold is diffeomorphic to the original manifold: the embeddings in smooth convergence are eventually defined on the whole compact limit, and their images are open and closed.

Normalize by \eqref{eq:forward-normalization} and use the potential gauge to obtain MRF. The slice construction and removal of the Killing component in Section~\ref{sec:horizontal} apply on $[s^*,\infty)$. The Killing correction is an ODE in the compact group $G$, so it exists along the whole future tail and preserves convergence. This gives \eqref{eq:forward-gauge}. The comparison of two horizontal gauges in that section is local in time; their relative element of $G$ is constant. Hence it also gives the marking assertion here.

Apply \cref{prop:forward-dichotomy}. The spectral alternative follows from \cref{prop:forward-leading}; the neutral expansion and rational decay exponent follow from \cref{prop:forward-expansion,prop:forward-exponent}. A zero slice tail reconstructs the canonical shrinking Ricci flow. Finally, $e^{-\lambda s}=\tau^\lambda$ and $s=|\log\tau|$ give \eqref{eq:forward-ricci-rates}.
\end{proof}

\begin{corollary}\label[corollary]{cor:forward-obstruction}
If the tangent shrinker has finite Li--Zhang obstruction order $k_{\mathrm{LZ}}$, every slow forward solution satisfies
\begin{equation}
 r+1\le\ell^*\le k_{\mathrm{LZ}}+1,\qquad
 \|h(s)\|_k=O\bigl(s^{-1/(k_{\mathrm{LZ}}-1)}\bigr).
 \label{eq:forward-obstruction-bound}
\end{equation}
If $k_{\mathrm{LZ}}=2$, then $r=2$, $\ell^*=3$, and
\[
 h(s)=s^{-1}\bigl(a^*+o_{C^k}(1)\bigr),\qquad
 F_2(a^*)=-a^*.
\]
\end{corollary}
\begin{proof}
The proof is that of \cref{prop:sandwich}, with integration toward the future. Put $k_0=k_{\mathrm{LZ}}$. The obstruction inequality and \eqref{eq:forward-subordinate} give
$\partial_s\Ent(b(s))\ge c|b(s)|^{2k_0}$. Using \eqref{eq:forward-entropy-rate} and integrating over $[s,2s]$ yields
\[
 cs^{1-2k_0/(\ell^*-2)}
 \le\Ent(b(2s))-\Ent(b(s))
 \le-\Ent(b(s))\le Cs^{-\ell^*/(\ell^*-2)}.
\]
Thus $1-2k_0/(\ell^*-2)\le-\ell^*/(\ell^*-2)$, equivalently $\ell^*\le k_0+1$. The remaining assertions follow from \cref{thm:forward}.
\end{proof}

\begin{corollary}\label[corollary]{cor:forward-cp}
Suppose the compact singularity tangent is $(\CP^{2n},g_{\mathrm{FS}})$, normalized by $\Ric_{g_{\mathrm{FS}}}=g_{\mathrm{FS}}/2$. Every slow forward solution satisfies, after a background isometry,
\begin{equation}
 h(s)=-\frac{h_{B_p}}s+o_{C^k}(s^{-1})
 \label{eq:forward-cp-profile}
\end{equation}
for exactly one $p\in\{1,\ldots,n\}$, where $B_p$ is given by \eqref{eq:Bp}. Its finite-order expansions are \eqref{eq:forward-expansion}. For $\CP^2$,
\begin{equation}
 \tau^{-1}\phi_t^*g_{\mathrm{RF}}(t) =\bg-\frac{u_*\bg+2\bar\nabla^2u_*}{|\log\tau|} +o_{C^k}(|\log\tau|^{-1}),\qquad u_*([z])=\frac52\, \frac{2|z_0|^2-|z_1|^2-|z_2|^2} {|z_0|^2+|z_1|^2+|z_2|^2}.
 \label{eq:forward-cp2-profile}
\end{equation}
\end{corollary}
\begin{proof}
By \cref{prop:cp-quadratic,cor:forward-obstruction}, the matrix $A^*$ corresponding to $a^*$ under \eqref{eq:cp-identification} satisfies $F_2(A^*)=-A^*$. Since $F_2$ is quadratic, $B=-A^*$ satisfies $F_2(B)=B$. Thus \cref{prop:cp-characteristics} identifies $B$ with exactly one $B_p$ up to unitary conjugation. This gives \eqref{eq:forward-cp-profile}; its specialization to $p=n=1$ is \eqref{eq:forward-cp2-profile}.
\end{proof}

If the static obstruction vanishes identically, then $\Ent\equiv0$ by \cref{lem:gradient-obstruction}. The proof of \cref{prop:forward-exponent} excludes the neutral alternative, and every nonstationary convergent flow has the spectral asymptotics \eqref{eq:forward-spectral}. In particular, the rate bound $\lambda\ge\lambda_-$ recovers the integrable case of \cite[Theorem~1.8]{CL}.

\begin{appendices}
\titleformat{\section}{\normalfont\Large\bfseries}{\appendixname~\thesection}{0.75em}{}
\section{Weighted estimates for the complement equation}\label[appendix]{app:parabolic}
The spectral theorem gives weighted $L^2$ estimates for the fixed operator $L_\Comp$. Spatial Schauder estimates then give the spatial regularity, and the equation controls the time derivative. We first treat the fixed operator and then absorb small perturbations.

Fix $s_*\le-2$. Recall from \eqref{eq:complement-spaces}
that $\Comp_j$ is the subspace of $\Comp$ with its spatial
$C^{j,1/2}$ norm. On $I_s=[s-1,s]$, define
\begin{equation}
 \|u\|_{\XP(I_s)} =\|u\|_{C(I_s;\Comp_3)} +\|\partial_su\|_{C(I_s;\Comp_{1})},\qquad \|R\|_{\YP(I_s)} =\|R\|_{C(I_s;\Comp_{1})}.
 \label{eq:append-norms}\index{I interval@$I_s$}\index{X Y strip@$\XP(I_s),\YP(I_s)$}
\end{equation}
Thus
\[
 \XP(I_s)=C(I_s;\Comp_3)
          \cap C^1(I_s;\Comp_{1}),\qquad
 \YP(I_s)=C(I_s;\Comp_{1}).
\]
The norms on the target spaces are understood. Continuity and
differentiability in time are taken in these Banach spaces.
For fixed-slice paths we use the same definitions with
$\Comp_j$ replaced by $\Sli\cap C^{j,1/2}$; for $K$-valued paths
we use the fixed norm on $K$.
For a positive weight $w$, put
\begin{equation}
 \|u\|_{\XP_w}=\sup_{s\le s_*}\frac{\|u\|_{\XP(I_s)}}{w(s)},\qquad
 \|R\|_{\YP_w}=\sup_{s\le s_*}\frac{\|R\|_{\YP(I_s)}}{w(s)}.
 \label{eq:append-weights}\index{X Y weighted@$\XP_w,\YP_w$}
\end{equation}

Write $\XP_p,\YP_p$ when $w(s)=|s|^{-p}$. Let $\lambda_+>0$ and $-\lambda_-<0$ be the smallest positive and largest negative eigenvalues of $L_\Comp$, as in Section~\ref{sec:dichotomy}. We use weights satisfying
\begin{equation}
 w(r)\le C_we^{\epsilon|r-s|}w(s),\qquad
 0\le\epsilon<\min\{\lambda_-,\lambda_+\}.
 \label{eq:append-tempered}\index{Tempered weight}
\end{equation}
These include the constant weight, polynomial weights on a sufficiently remote tail, and the slowly varying weights of Sections~\ref{sec:normal}--\ref{sec:exponent}. Fix also $0<\gamma<\min\{\lambda_-,\lambda_+\}$; the weight $e^{\gamma(s-s_*)}$ satisfies \eqref{eq:append-tempered}.

\begin{proposition}[Weighted Green estimate]\label[proposition]{prop:weighted-green}\index{L complement@$L_{\Comp}$}
Let $B(s)=L_\Comp+A(s)$ act on $\Comp$ at each time, with
\begin{equation}
 \|Au\|_{\YP(I_s)}\le\delta\|u\|_{\XP(I_s)},\qquad s\le s_*.
 \label{eq:append-smallness}
\end{equation}
For each weight satisfying \eqref{eq:append-tempered}, if $\delta$ is sufficiently small, every $R\in\YP_w$ has a unique solution $u=\Gop_BR\in\XP_w$ of
\begin{equation}
 u_s=B(s)u+R,\qquad \Pi_+u(s_*)=0,\qquad
 \|u\|_{\XP_w}\le C_w\|R\|_{\YP_w}.
 \label{eq:append-green}\index{G Green appendix@$\Gop_0,\Gop_B$}
\end{equation}
For the same choice of sufficiently small $\delta$, homogeneous ancient solutions with uniformly bounded unit-strip norms are parametrized by $\xi=\Pi_+z(s_*)\in E_+$ and satisfy
\begin{equation}
 \|z_\xi\|_{\XP(I_s)}\le Ce^{\gamma(s-s_*)}|\xi|.
 \label{eq:append-positive}\index{E positive B@$E_+^B(s_*),P_+^B(s_*)$}\index{E positive B@$E_+^B(s_*),P_+^B(s_*)$}
\end{equation}
Their terminal values form a graph $E_+^B(s_*)$ over $E_+$ along $E_-$. Projection onto this graph along $E_-$ defines $P_+^B(s_*)$, and the terminal condition in \eqref{eq:append-green} is equivalent to $P_+^B(s_*)u(s_*)=0$.
\end{proposition}
\begin{proof}
First take $A=0$. The negative component is determined by its growth at backward infinity, and the positive component by its terminal value. Thus
\begin{equation}
 (\Gop_0R)(s)=\int_{-\infty}^s e^{(s-r)L_\Comp}\Pi_-R(r)\dd r
 -\int_s^{s_*}e^{(s-r)L_\Comp}\Pi_+R(r)\dd r.
 \label{eq:append-G0}\index{G Green appendix@$\Gop_0,\Gop_B$}
\end{equation}
To justify the improper integral, put
$C_R=\sup_{r\le s_*}\|R(r)\|_{L^2}/w(r)$.
The spectral theorem and \eqref{eq:append-tempered} give the absolute bounds
\begin{align*}
 \int_{-\infty}^s
 \|e^{(s-r)L_\Comp}\Pi_-R(r)\|_{L^2}\dd r
 &\le C_wC_Rw(s)\int_0^\infty
        e^{-(\lambda_--\epsilon)v}\dd v=\frac{C_wC_Rw(s)}{\lambda_--\epsilon},\\
 \int_s^{s_*}
 \|e^{(s-r)L_\Comp}\Pi_+R(r)\|_{L^2}\dd r
 &\le C_wC_Rw(s)\int_0^{s_*-s}
        e^{-(\lambda_+-\epsilon)v}\dd v\le\frac{C_wC_Rw(s)}{\lambda_+-\epsilon}.
\end{align*}
Thus the negative integral converges
absolutely in $L^2$; the positive integral is over a finite interval in
the finite-dimensional space $E_+$. The triangle inequality gives
\begin{equation}
 \sup_{s\le s_*}\frac{\|\Gop_0R(s)\|_{L^2}}{w(s)}
 \le C_w\left(\frac1{\lambda_--\epsilon}+\frac1{\lambda_+-\epsilon}\right)
 \sup_{s\le s_*}\frac{\|R(s)\|_{L^2}}{w(s)}.
 \label{eq:append-L2}
\end{equation}
The integrals define a locally continuous $L^2$ path. Each spectral
coefficient satisfies its scalar inhomogeneous ODE, so testing against
finite sums of eigentensors and then using their density gives
$u_s=L_\Comp u+R$ in distributions. The terminal condition is
$\Pi_+u(s_*)=0$.

We now improve the regularity of this $L^2$ solution on the enlarged
past interval $I_s^*=[s-2,s]$. The required estimate is
\begin{equation}
 \|u\|_{\XP(I_s)}\le C\left(
 \sup_{r\in I_s^*}\|u(r)\|_{L^2}+\|R\|_{\YP(I_s^*)}\right).
 \label{eq:append-regularity}
\end{equation}
Here the norms on $I_s^*$ have the same definitions as
\eqref{eq:append-norms}.
For the fixed smooth operator $L_\Comp$, interior spatial Schauder
estimates uniform in time give
\[
 \sup_{r\in I_s}\|u(r)\|_3
 \le C\left(\sup_{r\in I_s^*}\|u(r)\|_{L^2}
             +\sup_{r\in I_s^*}\|R(r)\|_{1}\right).
\]
These are the spatial Schauder estimates for parabolic equations
with data continuous in time; see \cite[Theorem~2.2(I)]{DongZhang}
and \cite{Brandt,Lorenzi}.
Apply the second-order case of that theorem with spatial exponent
$1/2$, localize in fixed coordinate charts, and commute one spatial
derivative with $L_\Comp$.
The commutators have lower spatial order and are absorbed by
interpolation; the interior $L^2$-to-$L^\infty$ estimate controls
the remaining lower-order term. The equation then gives the same
bound for $\sup_{r\in I_s}\|u_s(r)\|_{1}$.

We also verify continuity in the full target norms. For
$0<\tau<1/2$, the backward time difference
$v_\tau(r)=u(r-\tau)-u(r)$ satisfies
\[
 \partial_rv_\tau=L_\Comp v_\tau+R(r-\tau)-R(r).
\]
Apply the same interior spatial estimate on nested intervals inside
$I_s^*$. The forcing tends to zero uniformly in $\Comp_{1}$,
and $v_\tau$ tends to zero uniformly in $L^2$, by continuity of
$R$ and the previously established $L^2$ continuity of $u$.
Consequently $u\in C(I_s;\Comp_3)$. Since
$L_\Comp:\Comp_3\to\Comp_{1}$ is bounded, the equation
gives $u_s\in C(I_s;\Comp_{1})$ as well.
This proves \eqref{eq:append-regularity}. The fixed coefficients
and interval lengths make the constant independent of $s$.

Comparability of $w$ on adjacent unit intervals now combines
\eqref{eq:append-L2} and \eqref{eq:append-regularity} into
\begin{equation}
 \|\Gop_0R\|_{\XP_w}\le C_w\|R\|_{\YP_w}.
 \label{eq:append-fixed-bound}
\end{equation}

This constructs the required solution for $R\in\YP_w$.
For uniqueness, let $v\in\XP_w$ solve the homogeneous equation and
$\Pi_+v(s_*)=0$. At fixed $s$ and $a<s$, the negative component satisfies
\[
 \|\Pi_-v(s)\|_{L^2}
 \le e^{-\lambda_-(s-a)}\|v(a)\|_{L^2}
 \le C_w\|v\|_{\XP_w}w(s)
       e^{-(\lambda_--\epsilon)(s-a)}.
\]
Letting $a\searrow-\infty$ gives $\Pi_-v(s)=0$.
The positive component vanishes by its terminal condition, so $v=0$.

For general $A$, solve
\begin{equation}
 u=\Gop_0(Au+R).
 \label{eq:append-Neumann}
\end{equation}
By \eqref{eq:append-smallness} and \eqref{eq:append-fixed-bound}, this is a contraction when $C_w\delta<1$, and
\[
 \|u\|_{\XP_w}\le\frac{C_w}{1-C_w\delta}\|R\|_{\YP_w}.
\]
The fixed-operator equation for $\Gop_0$ shows that \eqref{eq:append-Neumann} solves $u_s=(L_\Comp+A)u+R$. Conversely, the fixed-operator uniqueness shows that every solution in this weighted class satisfies \eqref{eq:append-Neumann}. This proves \eqref{eq:append-green}. The estimate allows $A$ to contain second spatial derivatives, since its output is measured in $\YP$.

For $\xi\in E_+$, solve
\begin{equation}
 z_\xi(s)=e^{(s-s_*)L_\Comp}\xi+\Gop_0(Az_\xi)(s)
 \label{eq:append-homogeneous}
\end{equation}
in the weight $e^{\gamma(s-s_*)}$. Choose $\delta$ small enough for this contraction and the unweighted one as well. This proves \eqref{eq:append-positive}. For a homogeneous solution with uniformly bounded unit-strip norms, variation of constants gives \eqref{eq:append-homogeneous} in the unweighted space, where uniqueness identifies it with $z_\xi$. Since $\Pi_+z_\xi(s_*)=\xi$, the terminal values form the stated graph, and
\[
 P_+^B(s_*)v=z_{\Pi_+v}(s_*)
\]
is its projection along $E_-$. The construction is equivariant under background isometries acting on the coefficients.
\end{proof}

\end{appendices}

\printindex
\clearpage

\begingroup
\fontsize{10.5pt}{13pt}\selectfont

\endgroup

\vskip10pt
\noindent
Yu Li, Institute of Geometry and Physics, University of Science and Technology of China, No. 96 Jinzhai Road, Hefei, Anhui Province, 230026, China; Hefei National Laboratory, No. 5099 West Wangjiang Road, Hefei, Anhui Province, 230088, China; E-mail: \href{mailto:yuli21@ustc.edu.cn}{yuli21@ustc.edu.cn}.

\end{document}